\documentclass[12pt]{amsart}
\usepackage{amsfonts,amsmath,amssymb,amsthm}
\allowdisplaybreaks[1]
\usepackage{tikz-cd}
\usepackage{tikz-3dplot}
\usepackage{subcaption}
\usepackage{extarrows}
\usepackage{enumerate}
\usepackage[a4paper, margin=1in]{geometry}
\usepackage{hyperref} 
\usepackage{cleveref} 

\newtheorem{theorem}{Theorem}[section]
\newtheorem{lemma}[theorem]{Lemma}
\newtheorem{proposition}[theorem]{Proposition}
\newtheorem{corollary}[theorem]{Corollary}
\newtheorem*{theoremA}{Theorem A}
\newtheorem*{theoremB}{Theorem B}
\newtheorem*{theoremC}{Theorem C}
\newtheorem*{theoremD}{Theorem D}

\theoremstyle{definition}
\newtheorem{definition}[theorem]{Definition}

\newtheorem{example}[theorem]{Example}

\newtheorem{settings}[theorem]{Settings}
\newtheorem{notation}[theorem]{Notation}
\newtheorem{conjecture}[theorem]{Conjecture}

\theoremstyle{remark}
\newtheorem{remark}[theorem]{Remark}

\newcommand{\Supp}{\textup{Supp}}

\newcommand{\rank}{\textup{rank}}

\hypersetup{
    colorlinks=true,
    linkcolor=black,
    filecolor=magenta,      
    urlcolor=cyan,
    pdftitle={Your Document Title},
    bookmarks=true
}

\begin{document}
\begin{sloppypar}
\title{Limits of $F$-invariants and Riemann-Stieltjes integral}
\author{Cheng Meng}
\address{Yau Mathematical Sciences Center, Tsinghua University, Beijing 100084, China.}
\email{cheng319000@tsinghua.edu.cn}
\date{\today}
\begin{abstract}
This paper proves several results on $F$-invariants of Fermat hypersurfaces, including the proof of an inequality on the Hilbert-Kunz multiplicity of Fermat quadric hypersurfaces conjectured by Watanabe and Yoshida, the asymptotic behavior of the Hilbert-Kunz multiplicity for Fermat cubic hypersurfaces, and a strict inequality of the $F$-signature of a Fermat hypersurface whose degree is equal to its dimension. To address the above problems, this paper introduces a numerical invariant for local rings of characteristic $p$ called multivariate $h$-function. It is a real function of several variables that recovers both the Hilbert-Kunz multiplicity and the $F$-signature of hypersurface rings. We prove the above results by developing integral formulas for the $h$-function of hypersurfaces defined by polynomials of the form $\phi(f_1,\ldots,f_s)$ in terms of the Riemann-Stieltjes integral, where $\phi$ is a polynomial and $f_i$'s are polynomials in independent sets of variables, and explore how taking derivatives and taking limit of the characteristic interact with the integrals.

\end{abstract}
\maketitle
\tableofcontents
\section{Introduction}
\subsection{Numerical invariants in characteristic $p$} Let $R$ be a Noetherian local ring of characteristic $p$ and $I$ be an ideal of finite colength in $R$. For such a pair, Monsky introduced a characteristic $p$ invariant known as the Hilbert-Kunz multiplicity $e_{HK}(R,I)$ in \cite{Mon83}. This is a positive real number given by
$$e_{HK}(R,I)=\lim_{e \to \infty}\frac{l(R/I^{[p^e]})}{p^{e\dim R}}\geq 1.$$
By \cite[Theorem 1.5]{WYeHK=1}, for an formally unmixed local ring $R$, $e_{HK}(R,\mathfrak{m})=1$ if and only if $R$ is regular local. 

There is another important invariant for $F$-finite local domains $R$ of characteristic $p$, namely the $F$-signature $s(R)$. By definition, if $a_e$ is the number of free summands of $F^e_*R$, where $F^e_*R$ is the pushforward of $R$ along $e$-th iteration of Frobenius map as an $R$-module, then we define
$$s(R)=\lim_{e \to \infty}\frac{a_e}{\rank_R F^e_*R}.$$
By the main result of \cite{Tuc12}, $s(R)$ always exists. It is a real number in $[0,1]$, and $s(R)=1$ if and only if $R$ is regular local. So, the two numerical invariants $e_{HK}(R,\mathfrak{m})$ and $s(R)$ give two ways to measure the singularity of the ring $R$.

Concrete computation and concrete values of these invariants are very difficult to study. These invariants are not necessarily integers, and worse still, the Hilbert-Kunz multiplicity may not be rational by a well-known unpublished result of Brenner. Up to now, there is no reliable algorithm to compute $e_{HK}(R)$ and $s(R)$ for a ring $R$ with a general explicit representation, and the values of these invariants are only known to very special classes of rings.  For example, the Hilbert-Kunz multiplicity of $\mathbb{F}_2[[x,y,z,u,v]]/(x^3+y^3+xyz+uv)$ is only conjecturally known in \cite{Monsky08conj}.

\subsection{Conjectures for Fermat hypersurface rings}
A natural question is how close the Hilbert–Kunz multiplicity of a singularity can be to 1. It turns out that such Hilbert-Kunz multiplicity cannot be arbitrarily close when $d=\dim R$ is fixed; for example, when $e_{HK}(R)<1+\frac{1}{d(d!(d-1)+1)^d}$, $R$ is regular by \cite[Theorem 4.12]{AE08}, so $e_{HK}(R)=1$. A better result in \cite[Chapter 8]{Huneke} shows that if $e_{HK}(R)<1+\frac{1}{d^dd!}$, then $R$ is regular. So it arouses one's interest to find an exact description of the lower bound for $e_{HK}(R)$ of singular rings $R$ depending on both $p,d$ or solely on $d$. In 2005, Watanabe and Yoshida made the following conjecture which predicts the ring achieving this minimal Hilbert-Kunz multiplicity, and the behavior of this multiplicity with respect to the characteristic:
\begin{conjecture}[\cite{WYconj05}, Conjecture 4.2]
Let $p$ be an odd prime, $S_{p,n,2}=\mathbb{F}_p[[x_0,\ldots,x_n]]/(x_0^2+\ldots+x_n^2)$, which is a singular ring of characteristic $p$ and dimension $n$. Then:
\begin{enumerate}
\item for any formally unmixed non-regular local ring $R$ of characteristic $p$ and dimension $n$, $e_{HK}(R)\geq e_{HK}(S_{p,n,2})$.
\item $e_{HK}(S_{p,n,2})\geq \lim_{p \to \infty}e_{HK}(S_{p,n,2})$.
\end{enumerate}    
\end{conjecture}
If the above conjecture is true, then $e_{HK}(S_{p,n,2})$ is a lower bound in characteristic $p$ and dimension $n$, and $\lim_{p \to \infty}e_{HK}(S_{p,n,2})$ would be a common lower bound for $e_{HK}(R)$ for all characteristics in dimension $n$. Here are the first few values of $m_n=\lim_{p \to \infty}e_{HK}(S_{p,n,2})-1$:
$$m_2=\frac{1}{2!},m_3=\frac{2}{3!},m_4=\frac{5}{4!},m_5=\frac{16}{5!},m_6=\frac{61}{6!}.$$
So $1+m_n$ is a great improvement from the known bound $1+\frac{1}{d^dd!}$.

The first statement in Watanabe-Yoshida's conjecture is studied in many cases. It has been proved in the complete intersection case by \cite{ES05}. It has also been proved in the low dimensional cases, including dimension $3$ in \cite{WY01}, $4$ in \cite[Theorem 4.3]{WYconj05}, $5$ and $6$ in \cite[Theorem 5.2]{AE13}, and more recently in dimension $7$ in \cite[Theorem 1.2]{AC24}. The second inequality is proved in \cite[Theorem A and B]{TriWYinequ} for the case $p>n-2$, along with a more precise description of $e_{HK}(S_{p,n,2})$ as a function of $p$ for $p>2^{\lfloor n/2 \rfloor}(n-2)$.

An unpublished result of Gessel-Monsky describes $\lim_{p \to \infty}e_{HK}(S_{p,n,2})$ depending on $n$.
\begin{theorem}[Gessel-Monsky]
$$\lim_{p \to \infty}e_{HK}(S_{p,n,2})=1+m_n,$$
where $m_n$ is the coefficient of $x^n$ of the Taylor expansion of $\sec(x)+\tan(x)$ near $0$.
\end{theorem}
It is interesting to see the appearance of trigonometric functions which is seemingly unrelated to the problem.

The above conjecture and theorem focus on a particular ring $S_{p,n,2}$. More generally, we can consider rings of the form $S_{p,n,d}=\mathbb{F}_p[[x_0,\ldots,x_n]]/(x_0^d+\ldots+x_n^d)$, which we may call \textit{the Fermat hypersurface ring}. The $F$-invariants of this kind of ring are studied in a more recent unpublished work by Caminata-Shideler-Tucker-Zerman.
\begin{proposition}[Caminata-Shideler-Tucker-Zerman]\label{1.3}
Assume $d$ is odd and $p=d^2-d-1$ is a prime number, then
$$s(S_{p,d,d})<\lim_{p \to \infty}s(S_{p,d,d})=\frac{1}{2^{d-1}(d-1)!}.$$
\end{proposition}
Thus assuming Bunyakovski's conjecture in number theory holds, we see there are infinitely many primes of the form $d^2-d-1$ for odd $d$, so there are infinitely many pairs $(p,d)$ such that the strict inequality holds. However, Bunyakovski's conjecture is a hard conjecture in number theory whose validity remains to be shown.
\subsection{Main results}
Here are the main results of this paper. The first result is on the second inequality conjectured by Watanabe-Yoshida.
\begin{theoremA}[See Theorem \ref{7.1 WY inequality proof}]\label{1 Theorem 1.1}
For any fixed characteristic $p\geq 3$,
$e_{HK}(S_{p,n,2}) \geq \lim_{p \to \infty}e_{HK}(S_{p,n,2}).$ Moreover, the equality holds when $n \leq 3$, otherwise the inequality is strict.
\end{theoremA}
This confirms the second inequality of Watanabe-Yoshida in \cite{WYconj05} in full generality and gives the condition when the inequality is strict.

We remark that although Watanabe-Yoshida's original statement is only for odd prime $p$, the above result has a natural extension to characteristic $2$. In this case, we replace $S_{p,n,2}$ with another ring $S'_{p,n,2}$, the hypersurface ring defined by $x_0x_1+\ldots+x_{n-1}x_n$ or $x_0x_1+\ldots+x_{n-2}x_{n-1}+x_n^2$ depending on parity of $n$. We see $S'_{p,n,2}\cong S_{p,n,2}$ when $p>2$ by a linear change of coordinates. This analogous statement appears in \cite{JNSWY23}, and can be confirmed using an unpublised result of Castillo-Rey. In this sense, we have:
\begin{corollary}
For any fixed characteristic $p\geq 2$,
$e_{HK}(S'_{p,n,2}) \geq \lim_{p \to \infty}e_{HK}(S'_{p,n,2}).$ Moreover, the equality holds when $n \leq 3$, otherwise the inequality is strict.    
\end{corollary}

Next, in Chapter 8, we reprove Gessel-Monsky's result in a new way. Our method explains the appearance of trigonometric functions, which are solutions to a linear system of ordinary differential equation with constant coefficient matrix. We give an algorithm to solve the corresponding problem for $S_{p,n,d}$ and run the algorithm for $d=3$, which yields the second result:
\begin{theoremB}[See Corollary \ref{8 corollary d=3}]\label{1 Theorem 1.4}
Set $\lim_{p \to \infty}e_{HK}(S_{p,n,3})=1+c_n$ and $\lim_{p \to \infty}s(S_{p,n,3})=c'_n.$
Then 
\begin{align*}
\sum_{n \geq 0}c_n\alpha^n=2\sqrt{3}\cdot(\frac{\sqrt{3} \cos \left( \frac{\sqrt{3} \alpha}{2} \right) + \sin \left( \sqrt{3} \alpha \right) }{1 + 2 \cos \left( \sqrt{3} \alpha \right)}),\\
\sum_{n \geq 0}c'_n\alpha^n=-\frac{1}{1-\alpha}+\frac{\sqrt{3}(2\sin \left( \frac{\sqrt{3} \alpha}{2} \right)+\sqrt{3})}{1 + 2\cos \left( \sqrt{3} \alpha \right)}.    
\end{align*}
\end{theoremB}

A third result is the strict inequality in Proposition \ref{1.3} in full generality.
\begin{theoremC}[See Proposition \ref{7.2 WY strict inequality on F-sig}]\label{1 Theorem 1.3}
Assume $p$ is a prime number, $d$ is an integer such that $p>d\geq 3$. Then
$$s(S_{p,d,d})<\lim_{p \to \infty}s(S_{p,d,d})=\frac{1}{2^{d-1}(d-1)!}.$$
\end{theoremC}
In other words, the inequality is strict in all cases we considered and thus the infinity of pairs $(p,d)$ is proved independent of Bunyakovski's conjecture.

\subsection{Central object: $h$-function}
The central object we study is a function of several real variables, called the multivariate $h$-function. For a Noetherian ring $R$, an $R$-ideal $I$, a sequence $\underline{f}=(f_1,\ldots,f_s)$ of elements in $R$ of length $s$ such that $\sqrt{(I,\underline{f})}=\mathfrak{m}$ is a maximal $R$-ideal, and an $s$-tuple of real numbers $(t_1,\ldots,t_s)\in\mathbb{R}^s$, we define
$$h_{R,I,\underline{f}}(t_1,\ldots,t_s)=\lim_{q \to \infty}\frac{l(R/(I^{[q]},f_1^{\lceil qt_1 \rceil},\ldots,f_s^{\lceil qt_s \rceil}))}{q^{\dim R_{\mathfrak{m}}}}$$
whenever the limit exists. This definition makes $h_{R,I,\underline{f}}$ a function of $s$-variables. Under mild hypotheses like $R$ is a domain, $I$ is $\mathfrak{m}$-primary and $f_i$'s are nonzero, this function exists and is continuous, which allows us to do integration.

The idea of this concept comes from an earlier work of the author and Mukhopadhyay, where we constructed a function $h_{R,I,J}(s)$ with respect to triples $(R,I,J)$ where $R$ is a local ring and $I,J$ are two $R$-ideals such that $I+J$ is of finite colength. In this case, we define
$$h_{R,J,I}(s)=\lim_{e \to \infty}\frac{l(R/I^{\lceil sq \rceil}+J^{[q]})}{p^{\dim R}}.$$
The existence and continuity of $h_{R,J,I}$\footnote{Here the notation is different from the original notation in that the order of $I,J$ is switched. However, the role of $I,J$ is not changed; we take Frobenius powers of $J$ and ordinary powers of $I$.} on $(0,\infty)$ is proved in the same work. The most commonly studied cases are $J=\mathfrak{m}$ and either $I=\mathfrak{m}$ or $I=fR$ is principal. This notion summarizes many notions defined before in different forms:
\begin{enumerate}
\item (\cite{BST13}) When $R$ is regular and $I=(f)$ is principal, $s(R,f^t)=1-h_{R,\mathfrak{m},f}(t)$.
\item (\cite{Teixeirathesis}) When $R$ is regular, $I=(f)$ and $t=a/q$, then $h_{R,\mathfrak{m},f}=\frac{1}{q^e}l(R/\mathfrak{m}^{[q]},f^a)$ is independent of choice of $a,q$.
\item (\cite{Tridensity}) When $R$ is standard graded, $I=\mathfrak{m}$, then $h'_{R,J,\mathfrak{m}}(s)$ is equal to the Hilbert-Kunz density function of the pair $(R,J)$.
\item (\cite{Otha}) The case where $R$ is regular and $I=(f)$ is principal.
\item (\cite{AE13},\cite{AC24}) The function $h_{R.\mathfrak{m},f}$ where $I=(f)$ is principal but $R$ is not necessarily regular is used to bound the Hilbert-Kunz multiplicity of singular rings.
\item (\cite{Taylor}) The notion $h_s(I,J)$ refers to $h_{R,I,J}(s)$, and the $s$-multiplicity $e_s(I,J)=h_{R,I,J}(s)/h_{R_0,m_0,m_0}(s)$ where $R_0$ is a regular local ring with $\dim R_0=\dim R$ and $m_0$ is the maximal ideal of $R_0$. After rescaling, we can study when the value $1$ characterizes regularity.
\end{enumerate}

One particular case of $h$-function of special importance is the $F$-signature of pairs, or more precisely, $1$ minus this $F$-signature, viewed as a function of $t$. For this specialization to work we assume $R$ is regular, and define
$$s(R,f^t)=\lim_{q \to \infty}\frac{l(R/(\mathfrak{m}^{[q]}:f^{\lceil tq\rceil}))}{q^{\dim R}}.$$
The notion appears in \cite{BST12} as a generalization of the $F$-signature to ideal pairs. The importance of this notion comes from the following facts
\begin{align*}
h'_{R,\mathfrak{m},f,+}(0)=-\tfrac{d}{dt^+}s(R,f^t)|_{t=0}=e_{HK}(R/f),\\
h'_{R,\mathfrak{m},f,-}(1)=-\tfrac{d}{dt^-}s(R,f^t)|_{t=1}=s(R/f).    
\end{align*}
Thus the $F$-signature of pairs or the $h$-function recovers both $e_{HK}(R/f)$ and $s(R/f)$. This allows the computation of $F$-invariants of hypersurface rings from $h$-function.

\subsection{Integral formula arising from representation theory}
Upon the introduction of $h$-function, the difficulty of computation passes from Hilbert-Kunz multiplicity and $F$-signature to the $h$-function. This seems a more expensive task. However, an astonishing fact is that we can recover some complicated $h$-functions from simpler $h$-functions. For example, when $R$ is a polynomial ring $k[\underline{x},\underline{y}]$ in two disjoint set of variables $\underline{x},\underline{y}$, and $f=f_1(\underline{x})+f_2(\underline{y})$ is a sum of two polynomials in disjoint set of variables, then we can compute $h_{R,(\underline{x},\underline{y}),f}$ from $h_{k[\underline{x}],(\underline{x}),f_1}$ and $h_{k[\underline{y}],(\underline{y}),f_2}$. In this way, the computation of $h$-functions concerning $f=x_0^d+\ldots+x_n^d$ can be computed from the $h$-function of a single monomial $x_0^d$, which is very easy.

We introduce the precise general statement here, which is the technical core of the paper.
\begin{theoremD}[Integral formula of $h$-function, See Theorem \ref{5.3 integral formula threshold}]\label{1 Theorem A}
\begin{align*}
h_{R,I,f}(r)=\int_{\prod_{1 \leq i \leq s}[0,C^+_i]}D_\phi(t_1,\ldots,t_s,r) \prod_{1 \leq i \leq s}d(-h'_{R_i,I_i,f_i}(t_i)).
\end{align*}    
\end{theoremD}
Here $R=\otimes_kR_i$, $I_i \subset R_i$, $f_i \in R_i$ for $1 \leq i \leq s$. $\phi$ is a polynomial in $s$-variables with coefficients in $k$, $f=\phi(f_1,\ldots,f_s) \in R$, and $C_i$ is the $F$-threshold of $f_i$ with respect to $I_i$. For other assumption of these notions, see Settings \ref{5.1 Tensor product setting}. In the above formula, the integral is a Riemann-Stieltjes integral of a continuous function over a product of functions which are not defined at countably many points, which makes sense by Proposition \ref{2.2 integral welldef up to points}.

The difficulty of concrete computation often lies in the mysterious function $D_\phi$, which by definition is equal to $h_{k[T_1,\ldots,T_s],(0),(T_1,\ldots,T_s,\phi)}$, so it is a particular kind of $h$-function.  Therefore, if the $h$-function of $s$-many elements $f_i$ in separated variables is known and $D_\phi$ is known, then we can compute the $h$-function of $\phi(\underline{f})$.

For example, in case $f=f_1(\underline{x})+f_2(\underline{y})$, the above formula specializes to the following formula:
\begin{align*}
h_{k[\underline{x},\underline{y}],(\underline{x},\underline{y}),f}(r)=\\
\int_{[0,C_1^+]\times[0,C_2^+]}D_{T_1+T_2}(t_1,t_2,r) d(-h'_{k[\underline{x}],(\underline{x}),f_1(\underline{x})}(t_1))d(-h'_{k[\underline{y}],(\underline{y}),f_2(\underline{y})}(t_2)).
\end{align*}  

We can also rewrite the above integral as
$$\int_{[0,C_1^+]\times[0,C_2^+]}D_{T_1+T_2}(t_1,t_2,r) (-h''_{k[\underline{x}],(\underline{x}),f_1(\underline{x})}(t_1))(-h''_{k[\underline{y}],(\underline{y}),f_2(\underline{y})}(t_2))dt_1dt_2$$
if the second derivative $h''$ is properly defined. When the $h$-function is a $C^2$-function, then $h''$ is just pointwise second derivative; if it is continuous and piecewise $C^2$, then $h''$ is its pointwise second derivative plus some Dirac delta function. 

Further results show that we can also calculate the derivative of $h$-functions or the limiting $h$-function in a reduction mod $p$ process using integrals. Actually, under some hypotheses of the function $D_\phi$, the integral in Theorem D commutes with the symbols $\frac{d}{dr}$ and $\lim_{p \to \infty}$. The commutativity of these symbols is proved in Theorem \ref{5.4 integral formula limit char}, Theorem \ref{5.5 integral formula derivative} and Theorem \ref{5.5 integral formula derivative limit char}. 

The most important example of known $D_\phi$ is $D_{T_1+T_2}$; we give a detailed description of its property in Chapter 4. It satisfies all the properties allowing the symbols $\frac{d}{dr}$ and $\lim_{p \to \infty}$ to commute with integrals.

With the integral formula for the $h$-function established, the proof of Theorems A, B, and C is straightforward. For Theorem A, we observe that to compare Hilbert-Kunz multiplicities, it suffices to compare $h$-functions near $0$. We realize $h$-function in characteristic $p$ as an integral over a function $D_p$ and its limit as an integral over a limit function $D_\infty$. We can prove $D_p \geq D_\infty$, and this implies the integral of $D_p$ is larger than $D_\infty$, so the $h$-function in fixed characteristic is larger than its limit. The proof of Theorem C is similar; we just analyze the points where $D_p>D_\infty$ is a strict inequality, which implies the integral of $D_p$ is strictly bigger than that of $D_\infty$. In this case the $h$-function in fixed characteristic is larger, so its left derivative at $1$, which is the $F$-signature of the hypersurface ring, is smaller. For Theorem B, denote the $h$-function of $x_0^d+\ldots+x_n^d$ by $\phi_n$. We construct series $\Phi(\alpha,x)=\sum_{n \geq 0}\phi_n(x)\alpha^n$. The integral formula tells us that $\phi_{n+1}$ is an integral transform of $\phi_n$ with respect to some kernel function. So $\Phi(\alpha,x)$ is an integral transform of itself, or in other words, $\Phi(\alpha,x)$ subjects to certain integral equations. Solving the equation gives us the expression of $\Phi(\alpha,x)$, and the series constructed from Hilbert-Kunz multiplicities is $\partial/\partial x^+\Phi(\alpha,0)$. This method also gives the series constructed from $F$-signature which is $\partial/\partial x^-\Phi(\alpha,1)$.

The idea of this statement is based on the results of Han-Monsky on representation rings of $k$-objects in \cite{HM93}. By definition, a $k$-object is a $k[T]$-module where the element $T$ acts nilpotently. Han and Monsky studied the action of $T_1+T_2$ on the tensor product of a $k[T_1]$-module and a $k[T_2]$-module, which can be viewed as a diagonal action, and expressed this tensor product as direct sums of $k[T_1+T_2]$-modules. It turns out that the $h$-function $h_{R,I,f}$ recovers the structure of the $k$-object $R/I^{[q]}$ with respect to $f$, and as a $k$-object we have $R/I^{[q]}$ is the tensor product of $R_i/I_i^{[q]}$. Therefore, if we know how $R_i/I_i^{[q]}$ decomposes as a $k$-object and how the tensor product for indecomposable $k$-object decomposes, we know the $k$-object structure of $R/I^{[q]}$. In the integral formula every term appears naturally; the term $D_\phi$ describes the decomposition of the product of indecomposable $k$-objects and the function $h_{R_i,I_i,f_i}$ describes the limiting behavior of the decomposition of $R_i/I_i^{[q]}$ as a $k$-object with respect to $f_i$.

The computation in Han and Monsky's paper is done in a combinatorial way. Although this method would give an accurate answer, the complexity of the final expression usually prevents us from deriving more results. Our integral formula avoids this difficulty, and allows us to analyze the Hilbert-Kunz multiplicity in different characteristics.

\subsection{Other results proved} The integral formula allows us to verify \cite[Theorem 15, Theorem 16]{shidelerthesis}, which is a main result in \cite{shidelerthesis}; this is done in Remark \ref{6.2 remark on CSTZ result}. The result says that the $h$-function of a Fermat hypersurface has a limit function, and one-sided derivatives of the $h$-function also converge to the one-sided derivatives of the limit function. It is actually proved for a more general hypersurface called the diagonal hypersurface, which is the hypersurface determined by $x_1^{d_1}+\ldots+x_s^{d_s}$.

Moreover, we also derive other new results on the $h$-functions, Hilbert-Kunz multiplicities and $F$-signatures of different types of rings, including:
\begin{enumerate}
\item (Theorem \ref{6.2 h-function of diagonal}, Corollary \ref{6.2 convergence of kernel function of s addition}, and Remark \ref{6.2 remark on CSTZ result}) The existence of a limit function $D_{\phi,\infty}(\mathbf{t},r)$ whose derivative is the limit of the derivative. When specified at $\mathbf{t}=(1/d_1,\ldots,1/d_s)$ and rescaled, this function gives the $h$-function of diagonal hypersurfaces.
\item (Theorem \ref{6.2 h-function of binomial}) Computations of $h$-functions of binomial hypersurfaces.
\item (Remark \ref{6.2 BCPT rmk6.5 disproved}) Proof of the fact that certain $F$-signature of pairs in characteristic $p$ may stabilize for large $p$ at certain points inside $(0,1)$, but not all of $(0,1)$.
\end{enumerate}

\subsection{The outline of the paper} This paper consists of 8 sections. Section 1 is the introduction to the whole paper. Section 2 is an introduction to the Riemann-Stieltjes integral and distributions on continuous functions, which lays the foundation for the analysis part of the paper. Section 3 is on the existence and properties of the multivariate $h$-function, and Section 4 focuses on its particular case $D_{T_1+T_2}$ which has many extra properties allowing machinery in the latter sections to run. In Section 5, we prove the integral formulas for $h$-function, and some concrete computations are done in Section 6. Section 7 is devoted to the proof of Theorem A and C. In Section 8, we focus on Gessel-Monsky's result and prove Theorem B.

\section{Properties of the Riemann-Stieltjes integral}\label{section 2}
This section deals with fundamental concepts and results related to the Riemann-Stieltjes integral.

Here are some notations used in the paper. For $s \in \mathbb{N}$, let $\mathbb{R}^s$ be the $s$-dimensional Euclidean space. We use bold font letters for elements $\mathbf{t}=(t_1,\ldots,t_s)\in\mathbb{R}^s$ and multiindices $\mathbf{i}=(i_1,\ldots,i_s)$ in a sum. For $t \in \mathbb{R}$, denote $\mathbf{t}=(t,t,\ldots,t)\in \mathbb{R}^s$. For $\mathbf{a},\mathbf{b} \in \mathbb{R}^s$ where $\mathbf{a}=(a_1,\ldots,a_s)$ and $\mathbf{b}=(b_1,\ldots,b_s)$, we say $\mathbf{a}\geq \mathbf{b}$ if $a_i \geq b_i$ for all $i$. In this case, $\geq$ is a partial order on $\mathbb{R}^s$. We say $\mathbf{a}> \mathbf{b}$ if $a_i > b_i$ for all $i$. An increasing function $f:\Omega \to \mathbb{R}$ for $\Omega \subset \mathbb{R}^s$ is a function satisfying the following property: whenever $\mathbf{a},\mathbf{b}\in\Omega$ and $\mathbf{a}\geq \mathbf{b}$, $f(\mathbf{a})\geq f(\mathbf{b})$. If $\Omega=\mathbb{R}^s$, this is saying $f$ is increasing in each variable. A decreasing function is defined in the same manner. For $\mathbf{a}\leq \mathbf{b}$, the symbol $[\mathbf{a},\mathbf{b}]=\prod_{1\leq i \leq s}[a_i,b_i]$ is the $s$-dimensional interval (or rectangle) defined by $\mathbf{a},\mathbf{b}$. The $1$-norm on $\mathbb{R}^s$ is the norm $||\mathbf{a}||_1=\sum_{1 \leq i \leq s}|a_i|$, and $d^*:(\mathbf{a},\mathbf{b}) \to ||\mathbf{a}-\mathbf{b}||_1$ is the metric induced by the $1$-norm. If $K$ is a set and $\mathbf{t}$ is a point, denote $d^*(\mathbf{t},K)=\inf\{ d^*(\mathbf{t},\mathbf{t}')|\mathbf{t}' \in K\}$. For any subset $X \subset \mathbb{R}^s$, let $C(X)$ be the set of continuous functions on $X$. If $X$ is compact, then for $f \in C(X)$ we can define its $\infty$-norm $||f||_\infty=\sup_{x \in X}f(x)$, and this norm makes $C(X)$ a complete metric space.
\subsection{Riemann-Stieltjes integral}
We recall the definition of $s$-dimensional Riemann-Stieltjes integral. For properties of this integral, one might check \cite[Chapter 12]{Morrey1991analysis} for reference.

\begin{definition}
Let $s \in \mathbb{N}$, $\mathbf{a}=(a_1,\ldots,a_s) \in \mathbb{R}^s$, $\mathbf{b}=(b_1,\ldots,b_s) \in \mathbb{R}^s$ with $\mathbf{a}<\mathbf{b}$.  Let $P_j$ be a partition of $[a_j,b_j]$ given by $x_{j,0}<\ldots<x_{j,n_j}$ and $P$ be the collection $P_1,\ldots,P_s$ of partitions. Denote $d(P)=\max\{x_{j,i}-x_{j,i-1}, 1 \leq j \leq s,1 \leq i \leq n_j\}$. Let $\alpha_j$ be a monotone increasing function on $[a_j,b_j]$, and $\alpha$ is the collection of $s$ functions $\alpha_1,\ldots,\alpha_s$. Denote $\Delta \alpha_{j,i}=\alpha_j(x_{j,i})-\alpha_j(x_{j,i-1})$. Let $f$ be a function on $[\mathbf{a},\mathbf{b}]$. Let $\mathbf{i}=(i_1,\ldots,i_s)$ be a multiindex, and $M_\mathbf{i}=\sup f|_{\prod_j[x_{j,i_j-1},x_{j,i_j}]}$, $m_\mathbf{i}=\inf f|_{\prod_j[x_{j,i_j-1},x_{j,i_j}]}$. Take any $\xi_\mathbf{i} \in \prod_j[x_{j,i_j-1},x_{j,i_j}]$, we say the sum
$$RS(P,\xi,f,\alpha)=\sum_{\mathbf{i}}f(\xi_\mathbf{i})\Delta\alpha_{1,i_1}\Delta\alpha_{2,i_2}\ldots\Delta\alpha_{s,i_s}$$
is the Riemann sum of the data $(P,\xi,f,\alpha)$, where $\mathbf{i}$ runs through all multiindices such that $1 \leq i_j \leq n_j$. We say the sum
$$U(P,f,\alpha)=\sum_{\mathbf{i}}M_\mathbf{i}\Delta\alpha_{1,i_1}\Delta\alpha_{2,i_2}\ldots\Delta\alpha_{s,i_s}$$
is the upper Riemann sum of $(P,f,\alpha)$, and
$$L(P,f,\alpha)=\sum_{\mathbf{i}}m_\mathbf{i}\Delta\alpha_{1,i_1}\Delta\alpha_{2,i_2}\ldots\Delta\alpha_{s,i_s}$$
is the lower Riemann sum of $(P,f,\alpha)$. 
\end{definition}
By definition, for any choice of $\xi$, we have
$$L(P,f,\alpha) \leq RS(P,\xi,f,\alpha) \leq U(P,f,\alpha)$$
and
$$U(P,f,\alpha)=\sup_\xi RS(P,\xi,f,\alpha), L(P,f,\alpha)=\inf_\xi RS(P,\xi,f,\alpha).$$
\begin{definition}
Fix a function $f:[\mathbf{a},\mathbf{b}] \to \mathbb{R}$. Suppose for any $\epsilon>0$, there is a $\delta>0$ such that whenever $d(P)<\delta$, $U(P,f,\alpha)-L(P,f,\alpha)<\epsilon$, i.e, $\lim_{d(P)\to 0}RS(P,\xi,f,\alpha)$ exists, then we say $f$ is Riemann-Stieltjes integrable with respect to $\alpha$, denoted by $f \in \mathcal{R}(\alpha)$, and define
$$\int_{[\mathbf{a},\mathbf{b}]}fd\alpha_1d\alpha_2\ldots d\alpha_s=\lim_{d(P)\to 0}RS(P,\xi,f,\alpha),$$
called the Riemann-Stieltjes integral of $f$ with respect to $\alpha$ on $[\mathbf{a},\mathbf{b}]$. In case $s=1$ where $\alpha=\alpha_1\in C[a,b]$ is an univariate function, the integral is also denoted by $\int_a^bfd\alpha$.
\end{definition}
\begin{remark}
We would like to mention an alternative definition of Riemann-Stieltjes integral. For example, in the definition of Riemann-Stieltjes integral in \cite{Rudinmathanalysis}, the limit is taken with respect to the directed set of all partitions under refinement. We choose to adopt the definition that takes limit with respect to the upper bound of diameter, which is stronger since any partition has a refinement of sufficiently small diameter. This definition allows us to avoid certain points on the interval, which brings convenience to the proof of the convergence result in Subsection 2.4. Also, continuous functions are integrable with respect to increasing functions under both definitions, so we do not need to worry about integrability.    
\end{remark}
The Riemann-Stieltjes integral satisfies many properties similar to the Riemann integral, including multilinearity on $(f,\alpha)$ and change of interval; see \cite[Theorem 12.9 and Theorem 12.11]{Morrey1991analysis}. It is also easy to see from definition that the integral of positive functions over increasing functions is nonnegative. In particular, the linearity with respect to $\alpha$ allows us to extend the definition from monotone functions to function of bounded variation:
\begin{definition}
Let $\alpha$ be a function of bounded variation on $[a,b]$. Then there are increasing functions $\alpha^+$ and $\alpha^-$ on $[a,b]$ such that $\alpha=\alpha^+-\alpha^-$. We define
$$\int_a^bfd\alpha=\int_a^bfd\alpha^+-\int_a^bfd\alpha^-.$$
The multilinear integral over $s$-tuples of functions of bounded variation is defined similarly.
\end{definition}
\begin{proposition}[\cite{Morrey1991analysis}, Theorem 12.14]
If $f \in C[\mathbf{a},\mathbf{b}]$ and $\alpha=(\alpha_1,\ldots,\alpha_s)$ where $\alpha_1,\ldots,\alpha_s$ are functions of bounded variation, then $\int_{[\mathbf{a},\mathbf{b}]}fd\alpha_1d\alpha_2\ldots d\alpha_s$ exists. In particular, when $s=1$, $\alpha=\alpha_1$ is an increasing function on $[a,b]$, $C([a,b]) \subset \mathcal{R}(\alpha)$.   
\end{proposition}
\begin{lemma}\label{2.1 Intermediatevalue}
Let $f \in C([\mathbf{a},\mathbf{b}]\times[\mathbf{a}',\mathbf{b}'])$. Then the function
$$(y_1,\ldots,y_{s'}) \to \int_{[\mathbf{a},\mathbf{b}]}f(x,y)d\alpha_1(x_1)d\alpha_2(x_2)\ldots d\alpha_s(x_s)$$
is a continuous function. Moreover, for each $[\mathbf{a}'',\mathbf{b}'']\subset[\mathbf{a},\mathbf{b}]$ there is $\xi=\xi(y)$ depending on $y$ such that
\begin{align*}
\int_{[\mathbf{a}'',\mathbf{b}'']}f(x,y)d\alpha_1(x_1)d\alpha_2(x_2)\ldots d\alpha_s(x_s)\\
=
\int_{[\mathbf{a}'',\mathbf{b}'']}f(\xi,y)d\alpha_1(x_1)d\alpha_2(x_2)\ldots d\alpha_s(x_s).    
\end{align*}
\end{lemma}
\begin{proof}
This is true since $f$ is uniformly continuous on $[\mathbf{a},\mathbf{b}]\times[\mathbf{a}',\mathbf{b}']$ and satisfies intermediate value theorem.  
\end{proof}
The above lemma allows us to define iterated Riemann-Stieltjes integral. We have:
\begin{proposition}[Fubini's theorem]\label{2.1 Fubini}
\begin{enumerate}
\item  Let $\pi$ be a permutation of $1,2,\ldots,s$. Then
$$\int_{[\mathbf{a},\mathbf{b}]}fd\alpha_1d\alpha_2\ldots d\alpha_s=\int_{[\mathbf{a},\mathbf{b}]}fd\alpha_{\pi(1)}d\alpha_{\pi(2)}\ldots d\alpha_{\pi(s)}.$$
That is, the order of integration does not affect the value of integration.
\item Let $f \in C[\mathbf{a},\mathbf{b}]$, then
$$\int_{[\mathbf{a},\mathbf{b}]}fd\alpha_1d\alpha_2\ldots d\alpha_s=\int_{a_s}^{b_s}\ldots(\int_{a_2}^{b_2}(\int_{a_1}^{b_1}fd\alpha_1)d\alpha_2)\ldots d\alpha_s.$$
That is, a multivariate Riemann-Stieltjes integral is an iterated univariate Riemann-Stieltjes integral.
\end{enumerate}
\end{proposition}
\begin{proof}
(1) is true since by definition $U(P,f,\alpha)$ and $L(P,f,\alpha)$ are both finite sums, and do not change after permuting the variables. (2) is true since by Lemma \ref{2.1 Intermediatevalue}, for each partition $P_1,\ldots,P_s$ there is a choice $\xi$ such that the right hand side is equal to $RS(P,\xi,f,\alpha)$, so we can take the limit when $d(P)\to 0$.   
\end{proof}
\subsection{The Riemann-Stieltjes integral as a distribution and extension of definition}
\begin{definition}
For a compact set $X \subset \mathbb{R}^s$, a \textbf{distribution} is a continuous linear functional $F:C(X) \to \mathbb{R}$ where the topology on $C(X)$ is given by the $||\cdot||_\infty$-norm.   
\end{definition}
There are two typical examples.
\begin{example}
Let $X \subset \mathbb{R}^s$ be a compact subset. It is Lebesgue measurable with finite measure. Let $g$ be an absolutely measurable function on $X$ with respect to Lebesgue measure, that is, $\int_X |g|<\infty$. Then $F:f \in C(X) \to \int_X f(x)g(x)dx$ is a distribution. 
\end{example}
\begin{example}
Let $c \in [a,b]$, then $f \to f(c)$ is a distribution on $C([a,b])$, called the Dirac delta distribution, denoted by $\delta_c$.
\end{example}
In this subsection, we will discuss the properties of a Riemann-Stieltjes integral as a distribution on continuous functions. We will see that as a distribution, the integral does not change if we modify the value of $\alpha$ at countably many points in the interior of the interval, so we can define the integral over classes of functions instead of an actual function. This allows us to integrate over the derivative of a concave function, which may not exist at all points.
\begin{proposition}
Let $[\mathbf{a},\mathbf{b}] \subset \mathbb{R}^s$ be a rectangle in $\mathbb{R}^s$, $\alpha_1,\ldots,\alpha_s$ be $s$ increasing functions on $[a_i,b_i],1 \leq i \leq s$ respectively. Then the following functional
$$f \to \int_{[\mathbf{a},\mathbf{b}]}fd\alpha_1\ldots  d\alpha_s$$
is a distribution on $C[\mathbf{a},\mathbf{b}]$, and its operator norm is $\prod_i (\alpha_i(b_i)-\alpha_i(a_i))$.
\end{proposition}
\begin{proof}
If $|f|\leq\epsilon$, then 
$$|\int_{[\mathbf{a},\mathbf{b}]}fd\alpha_1\ldots  d\alpha_s| \leq \int_{[\mathbf{a},\mathbf{b}]}\epsilon d\alpha_1\ldots  d\alpha_s=\epsilon\prod_i (\alpha_i(b_i)-\alpha_i(a_i))$$
and equality holds when $f=\epsilon$, so we are done.
\end{proof}
\begin{proposition}\label{2.2 integral welldef up to points}
Let $f$ be a continuous function on $[\mathbf{a},\mathbf{b}]$, $(\alpha_1,\beta_1),(\alpha_2,\beta_2),\ldots,(\alpha_s,\beta_s)$ be $s$ pairs of increasing functions on $[a_i,b_i],1 \leq i \leq s$ respectively. Assume for any $i$, $\alpha_i=\beta_i$ on endpoints $a_i,b_i$ and all but countably many points in $(a_i,b_i)$. Then
$$\int_{[\mathbf{a},\mathbf{b}]}fd\alpha_1\ldots  d\alpha_s=\int_{[\mathbf{a},\mathbf{b}]}fd\beta_1\ldots  d\beta_s.$$
\end{proposition}
\begin{proof}
In the definition of Riemann-Stieltjes integrals, the partitions can be taken arbitrary as long as their diameters go to $0$. Since $\alpha_i,\beta_i$ coincide on endpoints and all but countably many points, we can always choose the partition such that the set of points in the partition avoids all the points where $\alpha_i\neq \beta_i$, and they will give the same upper Riemann sum and the same lower Riemann sum. So taking the limit, we see the two integrals are equal.    
\end{proof}
Therefore, the following definition makes sense:
\begin{definition}
For $1 \leq i \leq s$, let $\alpha_i$ be increasing functions defined on $[a_i,b_i]\backslash\Omega_i$ where $\Omega_i$ is a countable subset of $(a_i,b_i)$. For $f \in C([\mathbf{a},\mathbf{b}])$, define
$$\int_{[\mathbf{a},\mathbf{b}]}fd\alpha_1\ldots  d\alpha_s=\int_{[\mathbf{a},\mathbf{b}]}fd\tilde{\alpha}_1\ldots  d\tilde{\alpha}_s,$$
where $\tilde{\alpha}_i$ is any increasing extension of $\alpha_i$ on $[a_i,b_i]$.
\end{definition}
Two natural candidates for $\tilde{\alpha}_i$ are the functions satisfying $\tilde{\alpha}_i(x)=\alpha_i(x^+),x \in \Omega_i$ and $\tilde{\alpha}_i(x)=\alpha_i(x^-),x \in \Omega_i$. 

Here the values $\alpha_i(a_i),\alpha_i(b_i)$ at endpoints affect the value of the integral. This leads to the following definition:
\begin{definition}
Let $\mathbf{a}=(a_1,\ldots,a_s)$, $\mathbf{b}=(b_1,\ldots,b_s)$, $f \in C[\mathbf{a},\mathbf{b}]$, $\sigma_{1,i},\sigma_{2,i} \in \{+,-,\textup{null}\}$ be two signs for $1 \leq i \leq s$. Let $\alpha_i$ be a function defined on an open interval containing $[a_i,b_i]$ of bounded variation. In this case $\alpha_i(x^\pm)$ exists for all $x \in [a_i,b_i]$. Define
$$\int_{\prod_{1 \leq i \leq s}[a_i^{\sigma_{1,i}},b_i^{\sigma_{2,i}}]}fd\alpha_1\ldots d\alpha_s=\int_{[\mathbf{a},\mathbf{b}]}fd\tilde{\alpha}_1\ldots d\tilde{\alpha}_s$$
where
$$\tilde{\alpha}_i(x)=\begin{cases}
\alpha_i(x) & x \neq a,b \\
\alpha_i(a^{\sigma_1}) & x =a\\
\alpha_i(b^{\sigma_2}) & x =b.
\end{cases}$$
\end{definition}
\begin{proposition}\label{2.2 integration with signed interval definition}
Let $\alpha$ be an increasing function defined on an open subset containing $[a,b]$ and suppose $f$ is continuous on an open set containing $[a,b]$. Then we have
\begin{align*}
\int_a^{b^\pm}fd\alpha=\lim_{c\to b^\pm}\int_a^cfd\alpha,\int_{a^\pm}^{b}fd\alpha=\lim_{c\to a^\pm}\int_{c}^{b}fd\alpha,\\
\int_{a^\pm}^{b^\pm}fd\alpha=\lim_{c_1\to a^\pm, c_2 \to b^\pm}\int_{c_1}^{c_2}fd\alpha.    
\end{align*}
\end{proposition}
\begin{proof}
We prove this proposition for the integral $\int_a^{b^+}$ and other signs can be proved similarly. Let
$$\tilde{\alpha}_i(x)=\begin{cases}
\alpha_i(x) & x \neq b \\
\alpha_i(b^+) & x =b.
\end{cases}$$
Note that for any $\epsilon>0$, $b \in (a,b+\epsilon)$ and $\alpha_i$, $\tilde{\alpha}_i$ only differ at $b$, so
$$\int_a^{b+\epsilon}fd\alpha=\int_a^{b+\epsilon}fd\tilde{\alpha}.$$
So it suffices to prove
$$\lim_{\epsilon\to 0^+}\int_a^{b+\epsilon}fd\tilde{\alpha}=\int_a^bfd\tilde{\alpha}$$
or
$$\lim_{\epsilon\to 0^+}\int_b^{b+\epsilon}fd\tilde{\alpha}=0.$$
We fix $\epsilon_0>0$, then $f$ is bounded on $[b,b+\epsilon_0]$. We assume $|f|\leq C$ on $[b,b+\epsilon_0]$. Then
$$|\int_b^{b+\epsilon}fd\tilde{\alpha}|\leq \int_b^{b+\epsilon}Cd\tilde{\alpha}=C(\tilde{\alpha}(b+\epsilon)-\tilde{\alpha}(b))\xrightarrow[]{\epsilon_0>\epsilon\to 0^+}0.$$
So we are done.
\end{proof}
In general, we can extend this definition to $s$-dimensional cubes using Fubini's theorem. The theme of this proposition is that integral over interval with signs is a limit of the integral over the usual interval; thus such integral possesses the same property as the usual definition of the integral. This will be used in the proof of integration by parts.

We see here if $\alpha(a)$ is not defined, then the integral $\int_{a^\sigma}^{b^{\sigma'}}fd\alpha$ makes sense if $a$ is equipped with sign $\sigma=\pm$. There is one particular case where it makes sense for $\sigma=$ null:
\begin{lemma}\label{2.2 sign of integration can change at zero of f}
Let $f \in C[a,b]$ and $\alpha$ be an increasing function on an open interval containing $[a,b]$. Suppose $f(a)=0$, then
$$\int_{a^-}^{b}fd\alpha=\int_a^{b}fd\alpha=\int_{a^+}^{b}fd\alpha.$$
\end{lemma}
\begin{proof}
For any $\epsilon>0$, by continuity of $f$ we can find $\delta>0$ such that $|f|<\epsilon$ on $[a-\delta,a+\delta]$ and $\alpha$ is defined on $[a-\delta,a+\delta]$. Since $\alpha$ is increasing, it is bounded on $[a-\delta,a+\delta]$. We assume $|\alpha|\leq C$ on $[a-\delta,a+\delta]$. Thus
$$|\int_a^{a+\delta}fd\alpha|\leq |\int_a^{a+\delta}\epsilon d\alpha|=\epsilon|\alpha(a+\delta)-\alpha(a)|\leq 2C\epsilon.$$
When $\epsilon \to 0$, $\int_a^{a+\delta}fd\alpha \to 0$. So $\int_a^bfd\alpha=\int_{a^+}^bfd\alpha$. The proof for $\int_{a^-}^b$ is similar.
\end{proof}
\begin{definition}
Let $\alpha$ be a function on $[a,b]$ and $f \in \mathcal{R}(\alpha)$. We formally write
$$\int_a^b f(x)\alpha^D(x)dx=\int_a^bf(x)d\alpha(x).$$
Here $\alpha^D(x)$ is a symbol representing the distribution $f \to \int_a^bf(x)d\alpha(x)$. We say it is the \textbf{derivative of $\alpha$ in the distribution sense.}
\end{definition}
\begin{example}
Let $c \in (a,b)$. Take
$$\alpha(x)=\begin{cases}
0 & x<c\\
\textup{arbitrary} & x=c\\
1 & x>c.
\end{cases}$$
Then $\alpha^D(x)=\delta_c$. In fact, for $f \in C[a,b]$ we have
$$\int_a^bf(x)d\alpha(x)=f(c)=\int_a^bf(x)\delta_c(x)dx.$$
\end{example}
\begin{remark}
In general, we need to distinguish between $\alpha^D$ and the pointwise derivative $\alpha'$. However, if $\alpha$ is an absolutely continuous function, these two coincide. More generally, if $\alpha$ is the sum of an absolutely continuous function and a jump function, then $\alpha^D$ is equal to $\alpha'$ plus a countable sum of $\delta_c$. In this case we still write $\alpha^D=\alpha'$ taking into the possibility when $\alpha'$ may be a delta distribution. This is true when $\alpha$ is the derivative of a piecewise $C^1$-function. 
\end{remark}

We recall the following propositions on convex and concave functions. Recall that a convex function on an interval $\mathcal{I}\subset \mathbb{R}$ is a function $f$ satisfying
$$f(\lambda a+(1-\lambda)b)\leq \lambda f(a)+(1-\lambda)f(b)$$
for any $a,b \in \mathcal{I}$ and $\lambda \in [0,1]$. If $f$ is continuous, then it suffices to check when $\lambda=1/2$, that is,
$$f(\frac{a+b}{2})\leq \frac{f(a)+f(b)}{2}.$$
A concave function $f$ is a function such that $-f$ is convex. See \cite{Convex} for properties of convex and concave functions.
\begin{proposition}\label{2.2 convex function property}
Let $\gamma:[a,b] \to \mathbb{R}$ be a concave function. Then:
\begin{enumerate}
\item $\gamma'_\pm(x)$ exist at all $x \in (a,b)$, and each of $\gamma'_+(a)$ and $\gamma'_-(b)$ either exists or is infinity. 
\item $\gamma'(x)$ exists at all but countably many points.
\item $\gamma'_\pm(x)$ are decreasing functions, and $\gamma'(x)$ is decreasing on its domain.
\item $\gamma'_{\pm}(x)=\gamma'(x^\pm)=\lim_{c \to x^\pm}\gamma'(c)$. The limit makes sense by (2) and (3).
\item $\gamma$ is absolutely continuous. Therefore, $\gamma^D=\gamma'$.
\end{enumerate}
\begin{lemma}\label{2.2 convex function property of derivative}
Let $\gamma:[a,b] \to \mathbb{R}$ be a concave function. Then:
$$\gamma'_+(x)=\lim_{0<\epsilon<\epsilon'\to 0}\frac{\gamma(x+\epsilon')-\gamma(x+\epsilon)}{\epsilon'-\epsilon}$$
and
$$\gamma'_-(x)=\lim_{0<\epsilon<\epsilon'\to 0}\frac{\gamma(x-\epsilon)-\gamma(x-\epsilon')}{\epsilon'-\epsilon}.$$
\end{lemma}
\begin{proof}
This comes from the inequality on three chords \cite[Equation 1.16]{Convex} and (4) of Proposition \ref{2.2 convex function property}.    
\end{proof}
   
\end{proposition}\label{2.2 integration along derivative of convex}
Proposition \ref{2.2 integral welldef up to points} and Proposition \ref{2.2 convex function property} lead to the following result:
\begin{proposition}
Let $f \in C[a,b]$, $\gamma$ be a concave function defined on $[a,b]$. If $\gamma$ is also defined on an open set containing $[a,b]$, then the following integral
$$\int_{a^{\sigma_1}}^{b^{\sigma_2}}fd\gamma'$$
is well-defined for $\sigma_1,\sigma_2 \in \{+,-\}$. If $\gamma$ is defined on $[a,b]$, concave and $\gamma'_+(a),\gamma'_-(b)<\infty$, then we can extend $\gamma$ to a concave function $\tilde{\gamma}$ defined on an open set containing $[a,b]$, then
$$\int_{a^{\sigma_1}}^{b^{\sigma_2}}fd\tilde{\gamma}$$
is well-defined for each choice of $\tilde{\gamma}$, and only depends on $\tilde{\gamma}'(a^{\sigma_1})=\tilde{\gamma}'_{\sigma_1}(a),\tilde{\gamma}'(b^{\sigma_2})=\tilde{\gamma}'_{\sigma_2}(b)$.
\end{proposition}
\begin{remark}
More generally, for $s$-tuples of convex functions on $[a_i,b_i]$ and $f \in C[\mathbf{a},\mathbf{b}]$, the integral
$$\int_{[\mathbf{a}^{\sigma_1},\mathbf{b}^{\sigma_2}]}fd\gamma'_1\ldots d\gamma'_s$$
is well defined for any choice of symbols $\sigma_1,\sigma_2$ that are not null. Moreover, by Lemma \ref{2.2 sign of integration can change at zero of f}, if $f(\mathbf{t})=0$ whenever $t_i=a_i$, then 
$$\int_{[\mathbf{a},\mathbf{b}^{\sigma_2}]}fd\gamma'_1\ldots d\gamma'_s=\int_{[\mathbf{a}^{\sigma_1},\mathbf{b}^{\sigma_2}]}fd\gamma'_1\ldots d\gamma'_s$$
is also well-defined.
\end{remark}

\subsection{Integration by parts}
We introduce the following version of integration by parts on Riemann-Stieltjes integrals of one variable.
\begin{theorem}[\cite{Morrey1991analysis}, Theorem 12.12]\label{2.3 integration by parts}
Let $f,\alpha$ be functions of bounded variation on $[a,b]$. Assume $f \in \mathcal{R}(\alpha)$. Then $\alpha \in \mathcal{R}(f)$, and
$$\int_a^bfd\alpha=f\alpha|_a^b-\int_a^b\alpha df.$$
\end{theorem}
\begin{corollary}
In Theorem \ref{2.3 integration by parts}, if we can extend $f,\alpha$ to functions on an open interval containing $[a,b]$ and $f$ is still Riemann-Stieltjes integrable with respect to $\alpha$ on this larger interval, then the above equation still holds if we replace $a,b$ by $a^{\sigma_1},b^{\sigma_2}$ where $\sigma_1,\sigma_2 \in \{+,-,\textup{null}\}$.    
\end{corollary}
\begin{corollary}\label{2.3 integration by parts twice}
Let $f,\alpha$ be two concave or piecewise $C^2$ continuous functions on $[a,b]$ whose left and right derivatives are bounded. Then
\begin{align*}
\int_a^bfd\alpha'=f\alpha'|_a^b-\int_a^b\alpha' df
=f\alpha'|_a^b-\int_a^b\alpha'f' dx\\
=f\alpha'|_a^b-\int_a^bf' d\alpha=f\alpha'|_a^b-f'\alpha|_a^b+\int_a^b\alpha df'(x).
\end{align*}
We can write $\int_a^b\alpha df'(x)=\int_a^b\alpha f'^D(x)dx$, which is equal to $\int_a^b\alpha f''(x)dx$ if $f'$ is the sum of an absolutely continuous function and countably many jump functions. The above equations remain true when we equip $a$ or $b$ with signs.
\end{corollary}
\subsection{Convergence of Riemann-Stieltjes integral}
In this subsection, we introduce some results on the convergence of Riemann-Stieltjes integrals. We include their proof for completeness, whose univariate version appears in \cite[Theorem 12.16]{Morrey1991analysis}.
\begin{theorem}\label{2.4 convergence of RS integral on f and alpha side}
Let $f_n$ be a sequence of uniformly bounded functions on $[\mathbf{a},\mathbf{b}]$ such that $f_n \to f$ pointwisely on $[\mathbf{a},\mathbf{b}]$. Assume $\alpha_{n,j}, 1 \leq j \leq s$ consist of $s$-tuples of sequences of uniformly bounded increasing functions on $[a_j,b_j],1 \leq j \leq s$, $\alpha_j$ is an $s$-tuple of increasing functions on $[a_j,b_j]$ such that $\alpha_{n,j} \to \alpha_j$ for all but countably many points inside $(a_j,b_j)$. Moreover, we assume either (1) $f_n \to f$ is uniform on $[\mathbf{a},\mathbf{b}]$, or (2) $f_n$ is increasing on $[\mathbf{a},\mathbf{b}]$ for any $n$. Then
$$\int_{[\mathbf{a},\mathbf{b}]} f_nd\alpha_{n,1}\ldots d\alpha_{n,s} \to \int_{[\mathbf{a},\mathbf{b}]} fd\alpha_{1}\ldots d\alpha_{s}$$
assuming all the Riemann-Stieltjes integrals above exist.
\end{theorem}
\begin{proof}
For each partition $P=(P_j), P_j=(a_j=x_{j,0}\leq x_{j,1}\leq \ldots \leq x_{j,n}=b_j)$, we have
$$\int_{[\mathbf{a},\mathbf{b}]} f_nd\alpha_{n,1}\ldots d\alpha_{n,s}\leq U(P,f_n,\alpha_n)=\sum_{\mathbf{i}}M_{n,\mathbf{i}}\prod_{1 \leq j \leq s}(\alpha_{n,j}(x_{j,i_j})-\alpha_{n,j}(x_{j,i_j-1}))$$
where $M_{n,\mathbf{i}}=\sup f_n|_{\prod_j [x_{j,i_j-1},x_{j,i_j}]}$.

If the partition $P$ avoids all points of non-convergence of $\alpha_j,1 \leq j \leq s$, then $\alpha_{n,j}(x_{j,i_j})\to \alpha_j(x_{j,i_j})$ for any $i$. If $f_n \to f$ uniformly, then for any interval $I$, $|\sup f_n|_I-\sup f|_I| \leq ||f_n-f||_\infty \to 0$, so $\sup f_n|_I \to \sup f|_I$. If $f_n$'s are all increasing, then so is $f$. We see every interval $I$ has a maximal element; if $I=[\mathbf{a},\mathbf{b}]$, then $\max I=\mathbf{b}$. Thus $\sup f_n|_I=f_n(\max I) \to f(\max I)= \sup f|_I$. In both cases we see for any interval $I$, $\sup f_n|_I \to \sup f|_I$. So $U(P,f_n,\alpha_n)\to U(P,f,\alpha)$ for any partition $P$ avoiding points of non-convergence. We fix such a partition and take $n \to \infty$, then
$$\overline{\lim_{n \to \infty}}\int_{[\mathbf{a},\mathbf{b}]}f_nd\alpha_{n,1}\ldots d\alpha_{n,s}\leq U(P,f,\alpha).$$
Since there are only countably many points of non-convergence for each $\alpha_i$, we can always choose $P$ avoiding those points while the diameter $d(P)$ is sufficiently small. Thus letting $d(P)\to 0$, we get
$$\overline{\lim_{n \to \infty}}\int_{[\mathbf{a},\mathbf{b}]}f_nd\alpha_{n,1}\ldots d\alpha_{n,s}\leq \int_{[\mathbf{a},\mathbf{b}]} fd\alpha_{1}\ldots d\alpha_{s}.$$
Similarly, we get
$$\underline{\lim}_{n \to \infty}\int_{[\mathbf{a},\mathbf{b}]}f_nd\alpha_{n,1}\ldots d\alpha_{n,s}\geq \int_{[\mathbf{a},\mathbf{b}]} fd\alpha_{1}\ldots d\alpha_{s}.$$
So we are done. 
\end{proof}
\begin{corollary}\label{2.4 convergence of RS integral continuous version}
Let $r$ be a parameter in some set $\Omega \subset \mathbb{R}^{s'}$ for some $s' \in \mathbb{N}$, and $r_0 \in \Omega$. Let $f_r$ be a collection of uniformly bounded functions such that $f_r \to f_{r_0}$ pointwisely on $[\mathbf{a},\mathbf{b}]$ as $r \to r_0$ in $\Omega$. Assume $\alpha_{r,j}, 1 \leq j \leq s$ consist of $s$-tuples of sequences of uniformly bounded increasing functions on $[a_j,b_j],1 \leq j \leq s$, $\alpha_{r_0,j}$ is an $s$-tuple of increasing functions on $[a_j,b_j]$ such that $\alpha_{r,j} \to \alpha_{r_0,j}$ for all $j$ and all but countably many points inside $(a_j,b_j)$. Moreover, we assume either (1) $f_r \to f_{r_0}$ is uniform on $[\mathbf{a},\mathbf{b}]$, or (2) $f_r$ is increasing on $[\mathbf{a},\mathbf{b}]$ for any $n$. Then
$$\int_{[\mathbf{a},\mathbf{b}]} f_rd\alpha_{r,1}\ldots d\alpha_{r,s} \to \int_{[\mathbf{a},\mathbf{b}]} f_{r_0}d\alpha_{r_0,1}\ldots d\alpha_{r_0,s}$$ 
assuming all the Riemann-Stieltjes integrals above exist.
\end{corollary}
\begin{proof}
Suppose the statement fails, then we can find a sequence $r_n \to r_0$ such that 
$$\underline{\lim}_{n \to \infty}|\int_{[\mathbf{a},\mathbf{b}]} f_{r_n}d\alpha_{r_n,1}\ldots d\alpha_{r_n,s}- \int_{[\mathbf{a},\mathbf{b}]} f_{r_0}d\alpha_{r_0,1}\ldots d\alpha_{r_0,s}|>0$$
and this contradicts Theorem \ref{2.4 convergence of RS integral on f and alpha side}. So we are done.   
\end{proof}
\begin{lemma}\label{5.5 lem: partial commutes with integration}
Suppose $f(\mathbf{t},r)$ is a function on $[0,\infty)^{s+1}$ which is concave in $r$. Let $\alpha_1,\ldots,\alpha_s$ be increasing functions. Fix $r_0 \geq 0$. Let $\mathcal{I}$ be an interval such that for any $r$, the following integrals on $\mathcal{I}$
$$\int_{\mathcal{I}}f(\mathbf{t},r)d\alpha_1\ldots d\alpha_s,\int_{\mathcal{I}}\frac{\partial}{\partial r^\pm}f(\mathbf{t},r_0)d\alpha_1\ldots d\alpha_s$$
are well-defined and $\mathbf{t} \to \frac{\partial}{\partial r^\pm}f(\mathbf{t},r_0)$ is an increasing function of $\mathbf{t}$, then for such $\mathcal{I}$ we have
$$\frac{d}{d r^\pm}\lvert_{r=r_0}\int_{\mathcal{I}}f(\mathbf{t},r)d\alpha_1\ldots d\alpha_s=\int_{\mathcal{I}}\frac{\partial}{\partial r^\pm}f(\mathbf{t},r_0)d\alpha_1\ldots d\alpha_s.$$
\end{lemma}
\begin{proof}
We prove for $\frac{\partial}{\partial r^+}$, and $\frac{\partial}{\partial r^-}$ can be proved similarly. We take any $\epsilon''>\epsilon'>0$. Then by convexity of $f$ in $r$,
$$\mathbf{t} \to f(\mathbf{t},r_0+\epsilon')-f(\mathbf{t},r_0+\epsilon)=\int_{r_0+\epsilon}^{r_0+\epsilon'}\frac{\partial}{\partial r^\pm}f(\mathbf{t},x)dx$$
is increasing in $\mathbf{t}$. Also, $r \to \int_{\mathcal{I}}f(\mathbf{t},r)d\alpha_1\ldots d\alpha_s$ is convex in $r$ for any $r$. Thus
\begin{align*}
\frac{d}{d r^+}\lvert_{r=r_0}\int_{\mathcal{I}}f(\mathbf{t},r)d\alpha_1\ldots d\alpha_s\\
=\lim_{0<\epsilon'<\epsilon'' \to 0}\frac{1}{\epsilon''-\epsilon'}(\int_{\mathcal{I}}f(\mathbf{t},r_0+\epsilon'')d\alpha_1\ldots d\alpha_s-\int_{\mathcal{I}}f(\mathbf{t},r_0+\epsilon')d\alpha_1\ldots d\alpha_s)\\
=\lim_{0<\epsilon'<\epsilon'' \to 0}\int_{\mathcal{I}}\frac{f(\mathbf{t},r_0+\epsilon'')-f(\mathbf{t},r_0+\epsilon')}{\epsilon''-\epsilon'}d\alpha_1\ldots d\alpha_s\\
=\int_{\mathcal{I}}\frac{\partial}{\partial r^+}f(\mathbf{t},r_0)d\alpha_1\ldots d\alpha_s.
\end{align*}
Here we apply Corollary \ref{2.4 convergence of RS integral continuous version} in the last equality and use $\epsilon',\epsilon''$ as parameters. So we are done.
\end{proof}

\subsection{Nonnegativity and positivity of integrals}
In this subsection, we prove the nonnegativity and positivity of certain integrals.
\begin{lemma}\label{7.1 nonnegativity lemma}
Let $\gamma_1$, $\gamma_2$ be two concave functions defined on $[a,b+\epsilon]$ for some $\epsilon>0$. Suppose $f(a)=\gamma_1(a)=\gamma_2(a)=0$, $\gamma_1(b^+)=\gamma_2(b^+)$, $\gamma'_1(b^+)=\gamma'_2(b^+)$, $f$ is a continuous concave function on $[a,b]$. Suppose for any $t \in [a,b]$, $\gamma_1(t)\leq \gamma_2(t)$. Then
$$\int_{a}^{b^+}fd(-\gamma'_1)\leq\int_{a}^{b^+}fd(-\gamma'_2).$$
\end{lemma}
\begin{proof}
Using integration by parts, we get
\begin{align*}
\int_{a}^{b^+}fd(-\gamma'_1)=-f\gamma'_1|_{a}^{b^+}+\int_{a}^{b^+}\gamma'_1df=-f\gamma'_1|_{a}^{b^+}+\int_{a}^{b^+}f'd\gamma_1\\
=-f\gamma'_1|_{a}^{b^+}+f'\gamma_1|_{a}^{b^+}+\int_{a}^{b^+}\gamma_1 d(-f') 
\end{align*}
and similar for $\gamma_2$. By the condition we see all the boundary terms $f\gamma'_i(a)$, $f\gamma'_i(b^+)$, $f'\gamma_i(a)$, $f'\gamma_i(b^+)$ do not depend on $i=1$ or $2$. So it suffices to check
$$\int_{a}^{b^+}\gamma_1d(-f')\leq\int_{a}^{b^+}\gamma_2d(-f'),$$
which is true since $-f'$ is increasing and $\gamma_1\leq \gamma_2$.
\end{proof}
\begin{definition}
Let $D$ be a distribution on $C(X)$ where $X \subset \mathbb{R}^s$ is a compact set. The support of $D$, denoted by $\Supp(D)$, is the complement in $X$ of the largest subset $U \subset X$ such that the restriction of $D$ on $U$ is $0$, that is, for any $f \in C(X)$ with support of $f$ lying in $U$, $D(f)=0$.
\end{definition}
We see if $\alpha_1,\alpha_2,\ldots,\alpha_s$ is an $s$-tuple of increasing functions, then $f \to \int_{[\mathbf{a},\mathbf{b}]}fd\alpha_1\ldots d\alpha_s$ is a distribution, so we can talk about its support. We have:
\begin{proposition}\label{7.2 lem: support given by jumps}
If for any $1 \leq i \leq s$, there exists $a_i<b_i$ such that $\alpha_i(a_i^+)<\alpha_i(b_i^-)$, then $\Supp(\alpha_1^D\ldots\alpha_s^D)\cap (\mathbf{a}_i,\mathbf{b}_i)\neq \emptyset$.    
\end{proposition}
\begin{proof}
It suffices to prove that there exists a continuous function $f$ supported on $(\mathbf{a}_i,\mathbf{b}_i)$ such that $\int_{[\mathbf{a},\mathbf{b}]}fd\alpha_1\ldots d\alpha_s\neq 0$. We may define $f=f_1(t_1)f_2(t_2)\ldots f_s(t_s)$ in separate variables, thus by Fubini's theorem, we only need to prove the case $s=1$, where $\alpha$ is an increasing function with $\alpha(a^+)<\alpha(b^-)$ and we need to prove $\int_a^bfd\alpha>0$ for some $f$ supported in $(a,b)$. In this case, by assumption we can choose $0<\epsilon_1<\epsilon_2$ sufficiently small such that
$$\alpha(a+\epsilon_2)<\alpha(b-\epsilon_2).$$
We define $f$ to be the continuous piecewise linear function as follows
$$f(x)=\begin{cases}
0 & x \leq a+\epsilon_1\\
\textup{linear} & a+\epsilon_1 \leq x \leq a+\epsilon_2\\
1 & a+\epsilon_2 \leq x \leq b-\epsilon_2\\
\textup{linear} & b-\epsilon_2 \leq x \leq b-\epsilon_1\\
0 & x \geq b-\epsilon_1.
\end{cases}$$
Then $f$ is supported on $[a+\epsilon_1,b-\epsilon_1]\subset (a,b)$. Let $P$ be the following partition of $[a,b]$: $P=\{a=x_0,a+\epsilon_1=x_1,a+\epsilon_2=x_2,b-\epsilon_2=x_3,b-\epsilon_1=x_4,b=x_5\}$. We see $\inf f=1$ on the interval $[x_2,x_3]$ and is $0$ on other intervals, so
\begin{align*}
\int_a^bfd\alpha\geq L(P,f,\alpha)=\alpha(b-\epsilon_2)-\alpha(a+\epsilon_2)>0.    
\end{align*}
Thus $\Supp\alpha^D\cap (a,b)\neq \emptyset$. 
\end{proof}
\begin{proposition}\label{7.2 lem: support gives positivity}
Let $\alpha_1,\ldots,\alpha_s$ be increasing functions, $f \in C[\mathbf{a},\mathbf{b}]$, $f \geq 0$. Assume $K_i=\Supp\alpha^D_i$, $K=\prod_i K_i \subset [\mathbf{a},\mathbf{b}]$. If there is $x \in K$ such that $f(x)>0$, then $\int_{[\mathbf{a},\mathbf{b}]}fd\alpha_1\ldots d\alpha_s>0$.    
\end{proposition}
\begin{proof}
Since $f(x)>0$, in a small neighbourhood $U$ of $x$, $f$ has a positive infimum. By definition, there is a continuous function $g$ supported inside $U$ such that $\int_{[\mathbf{a},\mathbf{b}]}gd\alpha_1\ldots d\alpha_s\neq 0$. By taking absolute value, we may assume $g \geq 0$ and $\int_{[\mathbf{a},\mathbf{b}]}gd\alpha_1\ldots d\alpha_s>0$. Since $g$ is a continuous function on $[\mathbf{a},\mathbf{b}]$, it has a maximum. So there is constant $C>0$ such that $f \geq Cg$, so $\int_{[\mathbf{a},\mathbf{b}]}fd\alpha_1\ldots d\alpha_s\geq \frac{1}{C}\int_{[\mathbf{a},\mathbf{b}]}gd\alpha_1\ldots d\alpha_s>0$.    
\end{proof}

\section{Multivariate $h$-function}
In this section, we assume $R$ is a Noetherian ring of characteristic $p>0$. From now on, we fix the notation that $q=p^e$ is always a power of $p$ where $e$ is a nonnegative integer, and we use $\lim_{e \to \infty}$ and $\lim_{q \to \infty}$ interchangeably when there is a sequence involving $q$. We will use the symbol $\underline{f}$ for a sequence of elements in rings and the symbol $\mathbf{t}$ for a sequence of numbers or a point in an Euclidean space. If $f=(f_1,\ldots,f_s)$ and $\mathbf{t}=(t_1,\ldots,t_s) \in \mathbb{Z}^s$, we define $\underline{f}^{\mathbf{t}}=(f_1^{t_1},\ldots,f_s^{t_s})$ which is a sequence in the same ring, and $(\underline{f}^{\mathbf{t}})$ is the ideal generated by the sequence $\underline{f}^{\mathbf{t}}$. We use the convention that a nonpositive power of an element generates the unit ideal. If $\mathbf{t}=(t_1,\ldots,t_s) \in \mathbb{R}^s$, we define $\lceil\mathbf{t}\rceil=(\lceil t_1\rceil,\ldots,\lceil t_s\rceil) \in \mathbb{Z}^s$ and define $\underline{f}^{\lceil\mathbf{t}\rceil}$ as above. The scalar multiplication and addition of $\mathbb{R}^s$ are as usual. We define $\mathbf{v}_i=(0,\ldots,1,\ldots,0) \in \mathbb{R}^s$ to be the unit vector in $i$-th coordinate direction of $\mathbb{R}^s$.

Now we introduce the concept of multivariate $h$-function.
\begin{definition}\label{3 h-function definition}
Let $R$ be a Noetherian ring, $\underline{f}=f_1,\ldots,f_s$ be a sequence in $R$ of length $s$, $I$ be an $R$-ideal such that $(I,\underline{f})$ is an $\mathfrak{m}$-primary ideal for some maximal ideal $\mathfrak{m}$. Denote $d=\dim R_\mathfrak{m}$. For $\mathbf{t}\in\mathbb{R}^s$, define
$$H_{e,R,I,\underline{f}}(\mathbf{t})=l(R/(I^{[q]},\underline{f}^{\lceil q\mathbf{t}\rceil})),h_{e,R,I,\underline{f}}(\mathbf{t})=\frac{l(R/(I^{[q]},\underline{f}^{\lceil q\mathbf{t}\rceil}))}{q^d}$$
and whenever the limit exists, define
$$h_{R,I,\underline{f}}(\mathbf{t})=\lim_{q \to \infty}\frac{l(R/(I^{[q]},\underline{f}^{\lceil q\mathbf{t}\rceil}))}{q^d}.$$
We will call the function $h_{R,I,\underline{f}}(\mathbf{t})$ \textbf{the $h$-function of the triple $(R,I,\underline{f})$.} We omit $R,I$ or $\underline{f}$ if they are clear from context.
\end{definition}
\begin{remark}
We make the following convention: if we only mention an element $f$ lying in some ambient polynomial ring or power series ring $S$ and do not specify $(R,I)$, then $R$ is a polynomial ring or power series ring containing $f$ over some variables of $S$, and $I$ is the $R$-ideal generated by these variables. For example, if $f=x_0^2 \in k[[x_0,\ldots,x_n]]$, then $h_{x_0^2}$ refers to $h_{k[[x_0]],(x_0),x_0^2}$. We can prove that adjoining a variable for both $R$ and $I$ does not change the $h$-function. For example, we have
$$h_{k[[x_0,x_1]],(x_0,x_1),x_0^2}=h_{k[[x_0]],(x_0),x_0^2}$$
Therefore, the notion $h_f$ is independent of $R$ chosen.
\end{remark}
\begin{remark}\label{3 locality remark}
In Definition \ref{3 h-function definition}, if we replace $R,I,\underline{f}$ by $R_\mathfrak{m},IR_\mathfrak{m},\underline{f}R_\mathfrak{m}$, then the lengths do not change. So the $h$-function of a non-local ring and the $h$-function of a local ring are equivalent. We will apply this remark later in Settings \ref{5.1 Tensor product setting} to tensor product of two rings, which may not be local even if the two rings are both local.  
\end{remark}
Whenvever $h_{R,I,\underline{f}}$ is well-defined, it is a function from $\mathbb{R}^s$ to $\mathbb{R}$ satisfying the following proposition.
\begin{proposition}\label{3 h-function basic property}
Assume $(R,\mathfrak{m},k)$ is a Noetherian local ring, $I$ is an $R$-ideal, $\underline{f}$ is a sequence in $R$ such that $(I,\underline{f})$ is $\mathfrak{m}$-primary. Let $h_e=h_{e,R,I,\underline{f}}$, $h=h_{R,I,\underline{f}}$ and assume $h(\mathbf{t})$ exists for all $\mathbf{t} \in \mathbb{R}^s$.
\begin{enumerate}
\item $h(\mathbf{t})=0$ whenever $t_i\leq 0$ for some $i$.
\item $h$ is increasing.
\item If $\dim R/f_i<\dim R$ for all $i$, then $h$ is Lipschitz continuous on any bounded set. If moreover the image of $I$ is $\mathfrak{m}$-primary in $R/f_i$ for all $i$, then it is Lipschitz continuous on $\mathbb{R}^s$.
\item If $I$ is $\mathfrak{m}$-primary in $R$, then $h(\mathbf{t})\leq e_{HK}(I,R)$ on $\mathbb{R}^s$ and there is a constant $C$ such that whenever $\mathbf{t}=(t_1,\ldots,t_s)\in\mathbb{R}^s$ with $t_i \geq C$, $h(t_1,\ldots,t_s)=h(t_1,\ldots,t_{i-1},C,t_{i+1},\ldots,t_s)$.
\item If $\dim R/f_i<\dim R$ for all $i$, then for each $s-1$-tuples $t_1,\ldots,\hat{t}_i, \ldots t_s$, the function $h(t_1,\ldots ,t_{i-1},\bullet,t_{i+1},\ldots t_s)$ is concave on $[0,\infty)$.
\item If $R$ is regular, $q_0=p^{e_0}$ is a power of $p$, and $q_0\mathbf{t} \in \mathbb{Z}^s$, then for any $q=p^e \geq q_0$,
$$h_{R,I,\underline{f}}(\mathbf{t})=h_{e,R,I,\underline{f}}(\mathbf{t})=\frac{H_{e,R,I,\underline{f}}(\mathbf{t})}{q^d}.$$
\end{enumerate}
\end{proposition}
\begin{proof}
(1) This is true since $t_i\leq 0$ implies $\lceil t_iq \rceil \leq 0$ for any $q$, so $\underline{f}^{\lceil q\mathbf{t}\rceil}$ generates the unit ideal.

(2) This is true since $H_{e,R,I,\underline{f}}$ and $h_{e,R,I,\underline{f}}$ are increasing in each variable.

(3) We first prove that it is Lipschitz continuous with respect to $1$-norm when $t_i \in \mathbb{Z}[1/p]$. It suffices to prove for any such $t_i$, there is constant $C_i$ such that for any $\epsilon>0$,
$$h(t_1,\ldots,t_i+\epsilon_i,\ldots,t_s)\leq h(t_1,\ldots,t_i,\ldots,t_s)+C_i\epsilon_i.$$
Now for sufficiently large $q$, $qt_j$ are all integers for $1 \leq j \leq i$. For such $q$ we have
\begin{align*}
H_e(t_1,\ldots,t_i+\epsilon_i,\ldots,t_s)-H_e(t_1,\ldots,t_i,\ldots,t_s)\\
=l(R/(I^{[q]},f_1^{ t_1q },\ldots,f_i^{ t_iq+\lceil \epsilon_i q\rceil },\ldots,f_s^{ t_sq }))-l(R/(I^{[q]},f_1^{t_1q },\ldots,f_i^{ t_iq },\ldots,f_s^{ t_sq }))
\end{align*}
\begin{align*}
=l((I^{[q]},f_1^{ t_1q },\ldots,f_i^{ t_iq },\ldots,f_s^{ t_sq })/(I^{[q]},f_1^{t_1q },\ldots,f_i^{ t_iq +\lceil \epsilon_i q\rceil},\ldots,f_s^{ t_sq }))\\
=\sum_{j=0}^{\lceil \epsilon_i q\rceil-1}l((I^{[q]},f_1^{ t_1q },\ldots,f_i^{ t_iq+j },\ldots,f_s^{ t_sq })/(I^{[q]},f_1^{t_1q },\ldots,f_i^{ t_iq +j+1},\ldots,f_s^{ t_sq }))\\
=\sum_{j=0}^{\lceil \epsilon_i q\rceil-1}l(R/(I^{[q]},f_1^{t_1q },\ldots,f_i^{ t_iq +j+1},\ldots,f_s^{ t_sq }):f_i^{t_iq+j})\\
\leq \lceil \epsilon_iq \rceil l(R/(I^{[q]},f_1^{t_1q },\ldots,f_i,\ldots,f_s^{ t_sq })).
\end{align*}
The inequality comes from the containment
$$(I^{[q]},f_1^{t_1q },\ldots,f_i,\ldots,f_s^{ t_sq }) \subset (I^{[q]},f_1^{t_1q },\ldots,f_i^{ t_iq +j+1},\ldots,f_s^{ t_sq }):f_i^{t_iq+j}.$$
Suppose $\mathbf{t}$ lies in a bounded set of $\mathbb{R}^s$. Then there is an $\mathfrak{m}$-primary ideal $J$ such that for all $q$,
$$J^{[q]} \subset (I^{[q]},f_1^{ t_1q },\ldots,f_i^{ t_iq },\ldots,f_s^{ t_sq }).$$
Thus we have
$$\lceil \epsilon_iq \rceil l(R/(I^{[q]},f_1^{t_1q },\ldots,f_i,\ldots,f_s^{ t_sq }))\leq \lceil \epsilon_iq \rceil l(R/J^{[q]},f_i).$$
Since $\dim R/f_i<\dim R$, $l(R/J^{[q]},f_i)\leq C_iq^{\dim R-1}$ for some $C_i$ and all $q$. So
\begin{align*}
h(t_1,\ldots,t_i+\epsilon_i,\ldots,t_s)- h(t_1,\ldots,t_i,\ldots,t_s)\\
=\lim_{q \to \infty}\frac{1}{q^{\dim R}}(H_e(t_1,\ldots,t_i+\epsilon_i,\ldots,t_s)-H_e(t_1,\ldots,t_i,\ldots,t_s))\\
\leq \lim_{q \to \infty}\frac{1}{q^{\dim R}}\lceil \epsilon_iq \rceil l(R/(I^{[q]},f_1^{t_1q },\ldots,f_i,\ldots,f_s^{ t_sq }))\\
\leq \lim_{q \to \infty}\frac{1}{q^{\dim R}}\lceil \epsilon_iq \rceil l(R/J^{[q]},f_i)\leq C_i\epsilon_i.
\end{align*}
Suppose $\mathbf{t}$ is not necessarily bounded, but $I$ is $\mathfrak{m}$-primary modulo $f_i$ for any $i$. Then
$$\lceil \epsilon_iq \rceil l(R/(I^{[q]},f_1^{t_1q },\ldots,f_i,\ldots,f_s^{t_sq}))\leq \lceil \epsilon_iq \rceil l(R/I^{[q]},f_i).$$
Replacing $J$ with $I$ in the previous argument, we get
$$h(t_1,\ldots,t_i+\epsilon_i,\ldots,t_s)- h(t_1,\ldots,t_i,\ldots,t_s) \leq C_i\epsilon_i.$$
Thus $h$ is Lipschitz continuous in either case. Now we prove the general case. Take $\mathbf{t}_1,\mathbf{t}_2 \in \mathbb{R}^s$ and $\mathbf{t}_3,\mathbf{t}_4,\mathbf{t}_5,\mathbf{t}_6 \in \mathbb{Z}[1/p]^s$ such that $\mathbf{t}_3\leq\mathbf{t}_1\leq\mathbf{t}_4$, $\mathbf{t}_5\leq\mathbf{t}_2\leq\mathbf{t}_6$. Then
$$h(\mathbf{t}_1)-h(\mathbf{t}_2)\leq h(\mathbf{t}_4)-h(\mathbf{t}_5)\leq C||\mathbf{t}_4-\mathbf{t}_5||_1$$
and
$$h(\mathbf{t}_2)-h(\mathbf{t}_1)\leq h(\mathbf{t}_6)-h(\mathbf{t}_3)\leq C||\mathbf{t}_6-\mathbf{t}_3||_1.$$
Since $\mathbb{Z}[1/p]^s$ is dense in $\mathbb{R}^s$, we can take $\mathbf{t}_3,\mathbf{t}_4\to\mathbf{t}_1$ and $\mathbf{t}_5,\mathbf{t}_6 \to\mathbf{t}_2$ to get Lipschitz continuity of $h$ on $\mathbb{R}^s$.

(4) Assume $I$ is $\mathfrak{m}$-primary. Then $H_e(\mathbf{t})=l(R/(I^{[q]},\underline{f}^{\lceil q \mathbf{t}\rceil}))\leq l(R/I^{[q]})$. Dividing by $q^{\dim R}$ and taking limit, we get $h(\mathbf{t})\leq e_{HK}(I,R)$. Also, there is a sufficiently large integer $C$ such that $f_i^C \in I$ for all $i$, which implies $f_i^{Cq} \in I^{[q]}$ for all $q$. Thus for $t_i \geq C$,
$$(I^{[q]},f_1^{\lceil t_1q \rceil},\ldots,f_i^{\lceil t_iq\rceil},\ldots,f_s^{\lceil t_sq \rceil})=(I^{[q]},f_1^{\lceil t_1q \rceil},\ldots,f_i^{Cq},\ldots,f_s^{\lceil t_sq \rceil}),$$
so $H_e(t_1,\ldots,t_i,\ldots,t_s)=H_e(t_1,\ldots,C,\ldots,t_s)$. Dividing by $q^{\dim R}$ and taking limit yields
$$h(t_1,\ldots,t_i,\ldots,t_s)=h(t_1,\ldots,C,\ldots,t_s)$$
whenever $t_i \geq C$.

(5) By Lipschitz continuity, it suffices to prove the case $t_1,\ldots,\hat{t}_i,\ldots,t_s \in \mathbb{Z}[1/p]\cap [0,\infty)$. By definition of $h$-function it suffices to prove that for large $e$, $H_e(t_1,\ldots,\hat{t}_i,\ldots,t_s)$ is concave on $1/q\mathbb{N}$, that is, for fixed $q,t_1,\ldots,\hat{t}_i,\ldots,t_s$,
$$t_i \to l(R/I^{[q]}, f_1^{t_1q},\ldots,f_i^{t_i},\ldots,f_s^{t_sq})=H_e(t_1,\ldots, t_i/q,\ldots,t_s)$$
is concave for $t_i \in \mathbb{N}$. We set $\bar{R}=R/(I^{[q]}, f_1^{t_1q},\ldots,\hat{f_i^{t_i}},\ldots,f_s^{t_sq})$, then
$$H_e(t_1,\ldots, t_i/q,\ldots,t_s)=l(\bar{R}/f_i^{t_i}\bar{R}),$$
which is concave since
\begin{align*}
2l(\bar{R}/f_i^{t_i}\bar{R})-l(\bar{R}/f_i^{t_i+1}\bar{R})-l(\bar{R}/f_i^{t_i-1}\bar{R})=l(f_i^{t_i-1}\bar{R}/f_i^{t_i}\bar{R})-l(f_i^{t_i}\bar{R}/f_i^{t_i+1}\bar{R})\\
=l(\bar{R}/f_i^{t_i}\bar{R}:f_i^{t_i-1})-l(\bar{R}/f_i^{t_i+1}\bar{R}:f_i^{t_i})\geq 0.
\end{align*}
The last inequality comes from the containment
$$f_i^{t_i}\bar{R}:f_i^{t_i-1} \subset f_i^{t_i+1}\bar{R}:f_i^{t_i}.$$
(6) Since $R$ is regular, we have $l(R/J^{[q]})=q^dl(R/J)$ for any $R$-ideal $J$. If $q\geq q_0$, then $q\mathbf{t}\in\mathbb{Z}^s$, so
$$(I^{[pq]},\underline{f}^{\lceil pq\mathbf{t} \rceil})=(I^{[pq]},\underline{f}^{ pq\mathbf{t}})=(I^{[q]},\underline{f}^{ q\mathbf{t}})^{[q]}=(I^{[q]},\underline{f}^{\lceil q\mathbf{t} \rceil})^{[q]}.$$
So
$$\frac{H_{e+1}(\mathbf{t})}{p^dq^d}=\frac{l(R/(I^{[pq]},\underline{f}^{\lceil pq\mathbf{t} \rceil}))}{p^dq^d}=\frac{l(R/(I^{[q]},\underline{f}^{\lceil q\mathbf{t} \rceil}))}{q^d}=\frac{H_e(\mathbf{t})}{q^d}$$
for all $e \geq e_0$. Thus,
$$h(\mathbf{t})=\lim_{q \to \infty}\frac{H_e(\mathbf{t})}{q^d}=\frac{H_e(\mathbf{t})}{q^d}=h_e(\mathbf{t}) \forall e \geq e_0.$$ 
\end{proof}
\begin{proposition}\label{3 h-function welldef domain}
The $h$-function is well-defined when $R$ is a domain.
\end{proposition}
\begin{proof}
If some $f_i=0$, then
$$h_{R,I,(f_1,\ldots,f_s)}(t_1,\ldots,t_s)=\begin{cases}
0 & t_i \leq 0\\
h_{R,I,(f_1,\ldots,\hat{f_i},\ldots,f_s)}(t_1,\ldots,\hat{t}_i,\ldots,t_s) & t_i>0.
\end{cases}$$ 
So we can drop $f_i$, and the length of the sequence $\underline{f}$ decreases by 1. Therefore, we may assume $f_i \neq 0$ for all $i$ without loss of generality. Then $\dim R/f_iR<\dim R$ for any $i$, so $h$-function is Lipschitz continuous on a bounded set. Suppose $\mathbf{t}=(t_1,\ldots,t_s) \in \mathbb{Z}[1/p]^s$, then we can choose sufficiently large $q_0$ such that $q_0\mathbf{t} \in \mathbb{Z}^s$, and
\begin{align*}
h_{R,I,\underline{f}}(\mathbf{t})=\lim_{q \to \infty}\frac{l(R/I^{[q]},\underline{f}^{\lceil q\mathbf{t}\rceil})}{q^d}=\lim_{q \to \infty}\frac{l(R/I^{[q_0q]},\underline{f}^{\lceil qq_0\mathbf{t}\rceil})}{q_0^dq^d}\\
=\lim_{q \to \infty}\frac{l(R/I^{[q_0q]},\underline{f}^{q_0\mathbf{t}[q]})}{q_0^dq^d}=\frac{e_{HK}(I^{[q_0]},\underline{f}^{q_0\mathbf{t}})}{q_0^d}
\end{align*}
exists. Suppose $\mathbf{t} \in \mathbb{R}^s$. If $t_i \leq 0$ for some $i$, then $h_{e,R,I,\underline{f}}(\mathbf{t})=0$ for all $e$, and there is nothing to prove. Now we assume $\mathbf{t}>\mathbf{0}$. Choose $\mathbf{0}<\mathbf{t}'\leq \mathbf{t} \leq \mathbf{t}''$ with $\mathbf{t}',\mathbf{t}'' \in \mathbb{Z}[1/p]$. Since $h$ and $h_e$ are increasing,
\begin{align*}
\lim_{q \to \infty}\frac{l(R/I^{[q]},\underline{f}^{\lceil q\mathbf{t}'\rceil})}{q^d}\leq\underline{\lim}_{q \to \infty}\frac{l(R/I^{[q]},\underline{f}^{\lceil q\mathbf{t}\rceil})}{q^d}\\
\leq\overline{\lim}_{q \to \infty}\frac{l(R/I^{[q]},\underline{f}^{\lceil q\mathbf{t}\rceil})}{q^d}\leq\lim_{q \to \infty}\frac{l(R/I^{[q]},\underline{f}^{\lceil q\mathbf{t}''\rceil})}{q^d}.    
\end{align*}
The leftmost term is $h(\mathbf{t}')$ and the rightmost term is $h(\mathbf{t}'')$. We have
$$h(\mathbf{t}')\leq h(\mathbf{t}'') \leq h(\mathbf{t}')+C||\mathbf{t}''-\mathbf{t}'||_1.$$
Let $\mathbf{t}',\mathbf{t}'' \to \mathbf{t}$, we get
$$\underline{\lim}_{q \to \infty}\frac{l(R/I^{[q]},\underline{f}^{\lceil q\mathbf{t}\rceil})}{q^d}=\overline{\lim}_{q \to \infty}\frac{l(R/I^{[q]},\underline{f}^{\lceil q\mathbf{t}\rceil})}{q^d},$$
which means $h(\mathbf{t})$ exists.
\end{proof}
We want to control when the $h$-function is eventually a constant. Such an estimate depends on a numerical data of $(I,f_i)$, called the $F$-threshold. Recall:
\begin{definition}
Let $R$ be a Noetherian ring, $I$ be an $R$-ideal, $f \in R$. Suppose $f \in \sqrt{I}$, then the $F$-threshold of $f$ with respect to $I$, denoted by $c^I(f)$, is the following limit
$$\lim_{q \to \infty}\frac{\min\{i:f^i \in I^{[q]}\}}{q}.$$
\end{definition}
The $F$-threshold always exists by \cite[Theorem A]{Fthresholdexists}. For $c<c^I(f)$, we see $f^{cq} \notin I^{[q]}$ for $q \gg 0$, and for $c>c^I(f)$, we see $f^{cq} \in I^{[q]}$ for $q \gg 0$.
\begin{notation}
We set
$$D_ih_{e,R,I,\underline{f}}(\mathbf{t})=q(h_{e,R,I,\underline{f}}(\mathbf{t}+1/q\mathbf{v}_i)-h_{e,R,I,\underline{f}}(\mathbf{t})),$$
$$D_iH_{e,R,I,\underline{f}}(\mathbf{t})=q(H_{e,R,I,\underline{f}}(\mathbf{t}+1/q\mathbf{v}_i)-H_{e,R,I,\underline{f}}(\mathbf{t})).$$
They are the difference quotients of $h_e$ and $H_e$ in $\mathbf{v}_i$-direction. We omit $i$ if $s=1$.
\end{notation}
\begin{proposition}\label{3 F-threshold property}
Let $R$ be a Noetherian ring, $I$ be an $R$-ideal, $\underline{f}$ be a sequence such that $(I,\underline{f})$ has finite length. Suppose for some $i$, $f_i \in \sqrt{I}$. Then for any $t_i,t'_i \geq c^I(f_i)$,
$$h_{R,I,\underline{f}}(t_1,\ldots,t_i,\ldots,t_s)=h_{R,I,\underline{f}}(t_1,\ldots,t'_i,\ldots,t_s).$$
Also, $D_ih_e(\mathbf{t})=0$ and $\frac{\partial h}{\partial t_i^+}(\mathbf{t})=0$ when $t_i\geq c^I(f_i)$ and $\frac{\partial h}{\partial t_i^-}(\mathbf{t})=0$ when $t_i>c^I(f_i)$.
\end{proposition}
\begin{proof}
The proof of the first claim is the same as (4) of Proposition \ref{3 h-function basic property}. The rest comes from the first claim and the definition.   
\end{proof}
\begin{remark}
The $F$-threshold is not easy to compute in general. However, we can find its upper bound. For example, if $R$ is local, $I$ is maximal and $f \in R$ is not a unit, then $c^I(f)\leq 1$.    
\end{remark}

\begin{proposition}\label{3 Dihe monotonicity and bound}
\begin{enumerate}
\item The functions $D_ih_e(\mathbf{t})$, $D_iH_e(\mathbf{t})$ are decreasing with respect to $t_i$ and increasing with respect to $t_j,j \neq i$ on the region $\{\mathbf{t}|\mathbf{t}+1/q\mathbf{v}_i>0\}$.
\item If $\dim R/f_i<\dim R$ and either $\mathbf{t}$ is bounded or $I$ is $\mathfrak{m}$-primary in $R/f_i$ for all $i$, then $D_ih_e$ is uniformly bounded in terms of $e$.
\end{enumerate}    
\end{proposition}
\begin{proof}
(1) It suffices to prove for $D_iH_e(\mathbf{t})$ since $D_ih_e(\mathbf{t})$ and $D_iH_e(\mathbf{t})$ only differ by a factor. We need to show: 
\begin{enumerate}[(a)]
\item If $j \neq i$, $D_iH_e(\mathbf{t}+t'\mathbf{v}_j)\geq D_iH_e(\mathbf{t})$ for $t'\geq 0$ such that $\mathbf{t}+1/q\mathbf{v}_i>0$.
\item $D_iH_e(\mathbf{t}+t'\mathbf{v}_i)\leq D_iH_e(\mathbf{t})$ for $t'\geq 0$ such that $\mathbf{t}+1/q\mathbf{v}_i>0$.
\end{enumerate}
From the definition of $H_e$ we see $H_e(\mathbf{t})=H_e(\frac{\lceil q\mathbf{t}\rceil}{q})$, so it suffices to show the case $q\mathbf{t}=\mathbf{r} \in \mathbb{Z}^s$ and $qt'=r' \in \mathbb{N}$ such that $r_i+1>0$, that is, $r_i \geq 0$. For (a), we need to prove
$$l(R/(I^{[q]},\underline{f}^{\mathbf{r+v}_i}))-l(R/(I^{[q]},\underline{f}^{\mathbf{r}}))\leq l(R/(I^{[q]},\underline{f}^{\mathbf{r+v}_i+r'\mathbf{v}_j}))-l(R/(I^{[q]},\underline{f}^{\mathbf{r}+r'\mathbf{v}_j})).$$
Equivalently,
$$l(\frac{I^{[q]},\underline{f}^{\mathbf{r}}}{I^{[q]},\underline{f}^{\mathbf{r+v}_i}})\leq l(\frac{I^{[q]},\underline{f}^{\mathbf{r}+r'\mathbf{v}_j}}{I^{[q]},\underline{f}^{\mathbf{r+v}_i+r'\mathbf{v}_j}}),$$
and this is also equivalent to
$$l(R/(I^{[q]},\underline{f}^{\mathbf{r+v}_i}):f_i^{r_i})\leq l(R/(I^{[q]},\underline{f}^{\mathbf{r+v}_i+r'\mathbf{v}_j}):f_i^{r_i}),$$
which is true by the inclusion $(I^{[q]},\underline{f}^{\mathbf{r+v}_i+r'\mathbf{v}_j}):f_i^{r_i} \subset (I^{[q]},\underline{f}^{\mathbf{r+v}_i}):f_i^{r_i}$. For (b), we need to prove
$$l(R/(I^{[q]},\underline{f}^{\mathbf{r+v}_i}))-l(R/(I^{[q]},\underline{f}^{\mathbf{r}}))\geq l(R/(I^{[q]},\underline{f}^{\mathbf{r+v}_i+r'\mathbf{v}_i}))-l(R/(I^{[q]},\underline{f}^{\mathbf{r}+r'\mathbf{v}_i})).$$
This is equivalent to
$$l(R/(I^{[q]},\underline{f}^{\mathbf{r+v}_i}):f_i^{r_i})\geq l(R/(I^{[q]},\underline{f}^{\mathbf{r+v}_i+r'\mathbf{v}_i}):f_i^{r'+r_i})$$
and follows from containment $(I^{[q]},\underline{f}^{\mathbf{r+v}_i}):f_i^{r_i} \subset (I^{[q]},\underline{f}^{\mathbf{r+v}_i+r'\mathbf{v}_i}):f_i^{r'+r_i}$ whenever $r_i\geq 0$.

(2) It suffices to prove when $q\mathbf{t}=\mathbf{r} \in \mathbb{Z}^s$. We need to show
$$l(\frac{I^{[q]},\underline{f}^{\mathbf{r}}}{I^{[q]},\underline{f}^{\mathbf{r+v}_i}})=l(R/(I^{[q]},\underline{f}^{\mathbf{r+v}_i}):f_i^{r_i})\leq Cq^{\dim R-1}.$$
There is a containment $(f_i)+(I^{[q]},\underline{f}^{\mathbf{r}})=(f_i)+(I^{[q]},\underline{f}^{\mathbf{r+v}_i}) \subset (I^{[q]},\underline{f}^{\mathbf{r+v}_i}):f_i^{r_i}$. 
If $\mathbf{t}$ is bounded, then there is an $\mathfrak{m}$-primary ideal $J$ such that $J^{[q]} \subset (I^{[q]},\underline{f}^{q\mathbf{t}})$ for any $q$, so $J^{[q]} \subset (I^{[q]},\underline{f}^{q\mathbf{t}})=(I^{[q]},\underline{f}^\mathbf{r})$. If $I$ is $\mathfrak{m}$-primary modulo $f_i$, we take $J=I+(f_i)$, then $J^{[q]}=I^{[q]}+(f_i^q) \subset (f_i)+I^{[q]}$. In both cases there is an $\mathfrak{m}$-primary ideal $J$ such that $J^{[q]} \subset (f_i)+(I^{[q]},\underline{f}^\mathbf{r})$. We have
$$\lim_{q \to \infty}\frac{l(R/(J^{[q]},f_i))}{q^{\dim R-1}}\leq e_{HK}(J,R/f_iR)$$
where equality holds if $\dim R/fR=\dim R-1$. Thus, there is constant $C$ such that $l(R/(J^{[q]},f_i))\leq Cq^{\dim R-1}$. Combining the chain of containment
$$(J^{[q]},f_i) \subset (f_i)+(I^{[q]},\underline{f}^{\mathbf{r}}) \subset (I^{[q]},\underline{f}^{\mathbf{r+v}_i}):f_i^{r_i},$$
we see $l(R/(I^{[q]},\underline{f}^{\mathbf{r+v}_i}):f_i^{r_i})\leq l(R/(J^{[q]},f_i))\leq Cq^{\dim R-1}$.  
\end{proof}
The concavity of $h$ in each variable implies that $\frac{\partial h}{\partial t_i^\pm}(\mathbf{t})$ exists for any $i$ and any $\mathbf{t}>\mathbf{0}$. We have the following result:
\begin{corollary}\label{3 Dihe approaches derivative}
For $\mathbf{t}$ such that $t_i>0$ for all $i$, we have:
\begin{enumerate}
\item $\frac{\partial h}{\partial t_i^+}(\mathbf{t}) \leq \underset{q \to \infty}{\underline{\lim}}D_ih_e(\mathbf{t})\leq \underset{q \to \infty}{\overline{\lim}}D_ih_e(\mathbf{t}) \leq \frac{\partial h}{\partial t_i^-}(\mathbf{t})$.
\item Whenever $\frac{\partial h}{\partial t_i}(\mathbf{t})$ exists, all $4$ terms above coincide.
\end{enumerate}    
\end{corollary}
\begin{proof}
(2) is a consequence of (1), so it suffices to prove (1). We first prove $\frac{\partial h}{\partial t_i^+}(\mathbf{t}) \leq \underset{q \to \infty}{\underline{\lim}}D_ih_e(\mathbf{t})$. Fix $\epsilon,\epsilon' \in \mathbb{Z}[1/p]$ such that $0<\epsilon<\epsilon'$ and choose $q$ large enough such that $q\epsilon,q\epsilon' \in \mathbb{N}$. By definition and decreasing property of $D_ih_e$ in $\mathbf{v}_i$-direction we have
$$(\epsilon'q-\epsilon q)D_ih_e(\mathbf{t})\geq\sum_{1 \leq j \leq \epsilon'q-\epsilon q}D_ih_e(\mathbf{t}+(\epsilon +(j-1)/q)\mathbf{v}_i)=q(h_e(\mathbf{t}+\epsilon'\mathbf{v}_i)-h_e(\mathbf{t}+\epsilon\mathbf{v}_i)).$$
Therefore,
$$D_ih_e(\mathbf{t})\geq\frac{(h_e(\mathbf{t}+\epsilon'\mathbf{v}_i)-h_e(\mathbf{t}+\epsilon\mathbf{v}_i))}{\epsilon'-\epsilon}.$$
We fix $\epsilon',\epsilon$ and let $e \to \infty$, then 
$$\underset{q \to \infty}{\underline{\lim}}D_ih_e(\mathbf{t})\geq \frac{(h(\mathbf{t}+\epsilon'\mathbf{v}_i)-h(\mathbf{t}+\epsilon\mathbf{v}_i))}{\epsilon'-\epsilon}$$
for any $\epsilon,\epsilon'$. Letting $\epsilon,\epsilon'\to 0$ and applying Lemma \ref{2.2 convex function property of derivative}, we get  $\frac{\partial h}{\partial t_i^+}(\mathbf{t}) \leq \underset{q \to \infty}{\underline{\lim}}D_ih_e(\mathbf{t})$. The inequality  $\frac{\partial h}{\partial t_i^-}(\mathbf{t}) \geq \underset{q \to \infty}{\overline{\lim}}D_ih_e(\mathbf{t})$ can be proved similarly by considering $\mathbf{t}-\epsilon\mathbf{v}_i$ and $\mathbf{t}-\epsilon'\mathbf{v}_i$.
\end{proof}
\begin{corollary}\label{3 hpartial monotonicity}
We fix $1 \leq i \leq s$. The following two functions on $(0,\infty)^s$
$$\mathbf{t} \to \frac{\partial h}{\partial t_i^+}(\mathbf{t}),\mathbf{t} \to \frac{\partial h}{\partial t_i^-}(\mathbf{t})$$
are decreasing in $t_i$-direction and increasing in $t_j$-direction for any $j \neq i$.
\end{corollary}
\begin{proof}
By Corollary \ref{3 Dihe approaches derivative} and Proposition \ref{2.2 convex function property}, we have
$$\frac{\partial h}{\partial t_i^+}(\mathbf{t})=\lim_{\epsilon \to 0^+}\frac{\partial h}{\partial t_i}(\mathbf{t}+\epsilon\mathbf{v_i})=\lim_{\epsilon\to 0^+}\lim_{q \to \infty}D_ih_e(\mathbf{t}+\epsilon\mathbf{v}_i),$$
where we use the fact that $\frac{\partial h}{\partial t_i}(\mathbf{t}+\epsilon\mathbf{v_i})$ exists for 
$\epsilon$ living outside countably many points in a neighbourhood of $0$. Therefore, the conclusion is true by Proposition \ref{3 Dihe monotonicity and bound}. The proof for $\frac{\partial h}{\partial t_i^-}(\mathbf{t})$ is the same.
\end{proof}
\begin{remark}
We remark here that $\frac{\partial h}{\partial t_i}(\mathbf{t})= \underset{q \to \infty}{\lim}D_ih_e(\mathbf{t})$ is equivalent to the commutativity of a double limit, which is not a trivial fact. It depends on the decreasing property of $D_ih_e$. 
\end{remark}

We introduce one particular case of $h$-function, which is essential in the latter part of the paper.
\begin{definition}\label{3 D-phi def}
Let $k$ be a field of characteristic $p$, $T_1,\ldots,T_s$ be $s$ variables, $A=k[T_1,\ldots,T_s]$, $0 \neq\phi \in (T_1,\ldots,T_s)A$. We define \textbf{the kernel function of $\phi$} to be the following function
$$D_\phi(\mathbf{t},x)=h_{A,0,(T_1,\ldots,T_s,\phi)}(\mathbf{t},x)=\lim_{q\to \infty}\frac{l(A/T_1^{\lceil t_1q \rceil},T_2^{\lceil t_2q \rceil},\ldots,T_s^{\lceil t_sq \rceil},\phi^{\lceil xq \rceil})}{q^s}.$$    
\end{definition}
Since $A$ is a domain and $(T_1,\ldots,T_s,\phi)=(T_1,\ldots,T_s)$ is maximal, $D_\phi:\mathbb{R}^{s+1}\to \mathbb{R}$ is a well-defined Lipschitz continuous function on any bounded set. It satisfies all the properties of the $h$-function.

At the end of this section, we derive a formula for the $h$-function of monomials. For a monomial in $T_i$'s of the form $T_1^{a_1}\ldots T_n^{a_n}$, we abbreviate it as $\mathbf{T}^\mathbf{a}$ where $\mathbf{a}=(a_1,\ldots,a_n)$.
\begin{theorem}\label{3 h-function of general monomial ideal}
Let $R=k[T_1,\ldots,T_n]$ be a polynomial ring, $\mathbf{a}_1,\ldots,\mathbf{a}_r,\mathbf{b}_1,\ldots,\mathbf{b}_s \in \mathbb{N}^n$. Let $I=(\mathbf{T}^{\mathbf{a}_1},\ldots,\mathbf{T}^{\mathbf{a}_r})$, $f_i=\mathbf{T}^{\mathbf{b}_i}$for $1 \leq i \leq s$, and assume $(I,\underline{f})$ is $(T_1,\ldots,T_n)$-primary.  Then for $\mathbf{t}=(t_1,\ldots,t_s)$, we have
$$h_{R,I,\underline{f}}(\mathbf{t})=\begin{cases}
0 & \exists t_i\leq 0\\
\operatorname{vol}(\mathbb{R}^n_+\backslash \cup_i (\mathbf{a}_i+\mathbb{R}^n_+)\cup\cup_j (t_j\mathbf{b}_j+\mathbb{R}^n_+)) & \textup{otherwise.}
\end{cases}$$
In particular, the $h$-function does not depend on the characteristic of the field.
\end{theorem}
\begin{proof}
If $t_i\leq 0$ for some $i$, then it is trivial. Now we assume $t_i>0$ for all $i$. Since $R$ is a domain, we may assume all $f_i \neq 0$ by dropping $0$'s. So by continuity it suffices to prove the equality when $\mathbf{t} \in \mathbb{Z}[1/p]^s$. Choose $q$ such that $q\mathbf{t}\in \mathbb{Z}^s$. Since $R$ is regular, we get
\begin{align*}
h_{R,I,\underline{f}}(\mathbf{t})=\frac{h_{e,R,I,\underline{f}}(\mathbf{t})}{q^n}=\frac{l(R/(\mathbf{T}^{q\mathbf{a}_1},\ldots,\mathbf{T}^{q\mathbf{a}_r},\mathbf{T}^{qt_1\mathbf{b}_1},\ldots,\mathbf{T}^{qt_s\mathbf{b}_s}))}{q^n}\\
=\frac{\operatorname{vol}(\mathbb{R}^n_+\backslash \cup_i (q\mathbf{a}_i+\mathbb{R}^n_+)\cup\cup_j (qt_j\mathbf{b}_j+\mathbb{R}^n_+))}{q^n}\\
=\operatorname{vol}(\mathbb{R}^n_+\backslash \cup_i (\mathbf{a}_i+\mathbb{R}^n_+)\cup\cup_j (t_j\mathbf{b}_j+\mathbb{R}^n_+)).
\end{align*}
\end{proof}

\section{Properties of $D_{T_1+T_2}$ in characteristic $p$ and limit function}
We are particularly interested in the kernel function of $\phi=T_1+T_2$ where $s=2$. In this section, we will study $D_{T_1+T_2}$ in detail. Some properties of this kernel function valued at integer points have been studied by Han in \cite{Han92}; we will introduce some results in \cite{Han92} and point out the limit form of these results, which will be useful in the latter sections.
\subsection{Value of $D_{T_1+T_2}$ and its limit}\label{section 4.1}

In this subsection, we fix a characteristic $p>0$. Let $k[T_1,T_2]$ be a polynomial ring over a field of characteristic $p$ and $D=D_{T_1+T_2}: \mathbb{R}^3 \to  \mathbb{R}, (t_1,t_2,t_3) \to \lim_{q\to\infty}\frac{l(k[T_1,T_2]/(T_1^{\lceil t_1q \rceil},T_2^{\lceil t_2q \rceil},(T_1+T_2)^{\lceil t_3q \rceil}))}{q^2}$.
\begin{proposition}
The function $D(t_1,t_2,t_3)$ only depends on the characteristic $p$ of the field $k$.    
\end{proposition}
\begin{proof}
Since the coefficient of $T_1+T_2$ lies in $\mathbb{Z}$, we have
$$l(k[T_1,T_2]/(T_1^{\lceil aq \rceil},T_2^{\lceil bq \rceil},(T_1+T_2)^{\lceil cq \rceil}))=l(\mathbb{F}_p[T_1,T_2]/(T_1^{\lceil aq \rceil},T_2^{\lceil bq \rceil},(T_1+T_2)^{\lceil cq \rceil}))$$
for any field $k$ of characteristic $p$.
\end{proof}
Whenever we want to emphasize the characteristic of the base field, we make the following definition:
\begin{definition}
For a prime number $p$, we define $D_p(t_1,t_2,t_3)=D(t_1,t_2,t_3)$ over any field of characteristic $p$.     
\end{definition}
Apart from all properties of $h$-function, the function $D(t_1,t_2,t_3)$ also satisfies the following properties.
\begin{proposition}[\cite{Han92}]\label{4.1 basic properties of Dp}
Assume $t_1,t_2,t_3 \geq 0$.
\begin{enumerate}
\item $D(t_1,t_2,t_3)$ is stable under permutation of $t_1,t_2,t_3$.
\item If $t_1+t_2 \leq t_3$, then $D(t_1,t_2,t_3)=t_1t_2$. If $t_1+t_3 \leq t_2$, then $D(t_1,t_2,t_3)=t_1t_3$. If $t_2+t_3 \leq t_1$, then $D(t_1,t_2,t_3)=t_2t_3$. Thus, if $(t_1,t_2,t_3)$ does not satisfy the triangle inequality, then $D(t_1,t_2,t_3)$ is the product of two smaller values of $t_1,t_2,t_3$.
\item If $t_1,t_2 \leq 1 \leq t_3$, then $D(t_1,t_2,t_3)=t_1t_2$.
\item (Rescaling) $D_p(t_1p,t_2p,t_3p)=p^2D_p(t_1,t_2,t_3)$.
\item (Deletion) If $t_1 \leq 1 \leq t_2,t_3$, then $D(t_1,t_2,t_3)=D(t_1,t_2-1,t_3-1)+t_1$.
\item (Reflection) If $0 \leq t_1,t_2,t_3 \leq 1$, then $D(t_1,t_2,t_3)=D(t_1,1-t_2,1-t_3)+t_1(t_2+t_3-1).$
\item For $t_1,t_2,t_3 \geq 0$ satisfying the triangle inequalities, $D(t_1,t_2,t_3)\geq \frac{2t_1t_2+2t_1t_3+2t_2t_3-t_1^2-t_2^2-t_3^2}{4}.$
\item If $t_1,t_2,t_3 \in \mathbb{Z}$, then $D(t_1,t_2,t_3)=l(k[T_1,T_2]/(T_1^{t_1},T_2^{t_2},(T_1+T_2)^{t_3}))$. In particular, it is an integer.
\end{enumerate}    
\end{proposition}
\begin{definition}\label{4.1 syzygy gap definition0}
Let $t_1,t_2,t_3 \geq 0$ satisfy the triangle inequalities. The \textbf{limit syzygy gap}, denoted by $[t_1,t_2,t_3]$, is the following nonnegative real number satisfying
$[t_1,t_2,t_3]^2=D(t_1,t_2,t_3)-\frac{2t_1t_2+2t_1t_3+2t_2t_3-t_1^2-t_2^2-t_3^2}{4}$. We write $[t_1,t_2,t_3]_p$ when we want to emphasize the characteristic.
\end{definition}
This definition is compatible with Definition 2.1, Definition 2.28, and Theorem 2.29 of \cite{Han92}.

We also recall the following notions:
\begin{definition}[\cite{Han92}, Definition 2.15, Definition 2.26]\label{4.1 geometry: faces F}
\begin{enumerate}
\item $F \subset \mathbb{R}^3$ is the union of planes $\sum_{1 \leq i \leq 3}a_1t_1+a_2t_2+a_3t_3=a_4$ where $a_1,a_2,a_3=\pm1$ and $a_4 \in 2\mathbb{Z}$.
\item A cell is a connected component of $\mathbb{R}^3\backslash F$. 
Let $d^*$ be the metric on $\mathbb{R}^3$ induced by $1$-norm, then $d^*$-balls are octahedra. There are two kinds of cells: one is an octahedron ball centered at $(t_1,t_2,t_3)$ with $t_1,t_2,t_3 \in \mathbb{Z},t_1+t_2+t_3 \in 2\mathbb{Z}+1$, whose radius is $1$; the other one is a tetrahedron centered at $(t_1+1/2,t_2+1/2,t_3+1/2)$ for $t_1,t_2,t_3 \in \mathbb{Z}$.
\item Let $\Theta$ be the set of all closed tetrahedron cells.
\end{enumerate}    
\end{definition}
The above notions describe the limit syzygy gap in a geometric way.
\begin{proposition}[\cite{Han92},Definition 2.28]\label{4.1 syzygy gap definition}
For $(t_1,t_2,t_3)$ satisfying the triangle inequality, $$[t_1,t_2,t_3]_p=1/2\cdot\max_{n \in \mathbb{Z}} d^*((t_1,t_2,t_3),1/p^n\Theta).$$    
\end{proposition}
We are particularly interested in the behavior of $D(t_1,t_2,t_3)$ in the unit cube $[0,1]^3$. We divide the cube into $5$ parts, namely $B_1,B_2,B_3,B_4,T_0$ as in Figure \ref{figure 0} and Figure \ref{figure 1}. Here $B_1\sim B_4$ correspond to $(t_1,t_2,t_3)\in [0,1]^3$ satisfying inequalities $t_2+t_3\leq t_1,t_1+t_3\leq t_2,t_1+t_2\leq t_3,t_1+t_2+t_3\geq 2$ respectively; the closure of their complement in $[0,1]^3$ is $T_0$, which is a tetrahedron with vertices $(0,0,0),(0,1,1),(1,0,1),(1,1,0)$. Also, when we reflect two components at a time, then points in $B_4$ get transformed to points in $B_i$ for some $1 \leq i \leq 3$. 

\begin{figure}[ht]
    \centering
    \begin{tikzpicture}[scale=2]
    \draw[thick,->] (0,0,0) -- (1.5,0,0) node[anchor=north east]{$t_1$};
    \draw[thick,->] (0,0,0) -- (0,1.5,0) node[anchor=south west]{$t_2$};
    \draw[thick,->] (0,0,0) -- (0,0,1.5) node[anchor=south]{$t_3$};

    \draw[black, thick] (0,0,0) -- (1,0,0) -- (1,1,0) -- (0,1,0) -- cycle;
    \draw[black, thick] (0,0,1) -- (1,0,1) -- (1,1,1) -- (0,1,1) -- cycle;
    \draw[black, thick] (0,0,0) -- (0,0,1);
    \draw[black, thick] (1,0,0) -- (1,0,1);
    \draw[black, thick] (0,1,0) -- (0,1,1);
    \draw[black, thick] (1,1,0) -- (1,1,1);

    \foreach \x in {0,1}
    \foreach \y in {0,1}
    \foreach \z in {0,1} {
        \filldraw (\x,\y,\z) circle (0.5pt);
    }

    \draw[red, thick] (0,0,0) -- (0,1,1) -- (1,0,1) -- cycle;
    \draw[red, thick] (0,0,0) -- (1,0,1) -- (1,1,0) -- cycle;
    \draw[red, thick] (0,0,0) -- (0,1,1) -- (1,1,0) -- cycle;
    \draw[red, thick] (0,1,1) -- (1,0,1) -- (1,1,0) -- cycle;

    \filldraw[blue] (0,0,0) circle (1pt) node[anchor=south east]{$(0,0,0)$};
    \filldraw[blue] (0,1,1) circle (1pt) node[anchor=south east]{$(0,1,1)$};
    \filldraw[blue] (1,0,1) circle (1pt) node[anchor=north west]{$(1,0,1)$};
    \filldraw[blue] (1,1,0) circle (1pt) node[anchor=north west]{$(1,1,0)$};

    \filldraw[black] (1,0,0) circle (1pt) node[anchor=south west]{$(1,0,0)$};
    \filldraw[black] (1,1,1) circle (1pt) node[anchor=north west]{$(1,1,1)$};
    \filldraw[black] (0,0,1) circle (1pt) node[anchor=south east]{$(0,0,1)$};
    \filldraw[black] (0,1,0) circle (1pt) node[anchor=south west]{$(0,1,0)$};
\end{tikzpicture}
        \caption{The unit cube as a union of $B_1\sim B_4,T_0$}
        \label{figure 0}
\end{figure}

\begin{figure}
    \centering
    \begin{tikzpicture}[scale=2]
    \draw[red, thick] (0,0,0) -- (0,1,1) -- (1,0,1) -- cycle;
    \draw[red, thick] (0,0,0) -- (1,0,1) -- (1,1,0) -- cycle;
    \draw[red, thick] (0,0,0) -- (0,1,1) -- (1,1,0) -- cycle;
    \draw[red, thick] (0,1,1) -- (1,0,1) -- (1,1,0) -- cycle;

    \coordinate (A) at (0,0,0.3); 
    \coordinate (B) at (1,0,1.3); 
    \coordinate (C) at (0,1,1.3); 
    \coordinate (D) at (0,0,1.3);

    \draw[black, thick] (A) -- (B);
    \draw[black, thick] (A) -- (C);
    \draw[black, thick] (A) -- (D);
    \draw[black, thick] (B) -- (C);
    \draw[black, thick] (B) -- (D);
    \draw[black, thick] (C) -- (D);

    \coordinate (A) at (0.3,0,0);
    \coordinate (B) at (1.3,0,0); 
    \coordinate (C) at (1.3,0,1); 
    \coordinate (D) at (1.3,1,0);

    \draw[black, thick] (A) -- (B); 
    \draw[black, thick] (A) -- (C); 
    \draw[black, thick] (A) -- (D); 
    \draw[black, thick] (B) -- (C); 
    \draw[black, thick] (B) -- (D); 
    \draw[black, thick] (C) -- (D); 

    \coordinate (A) at (0,0.3,0); 
    \coordinate (B) at (0,1.3,0); 
    \coordinate (C) at (1,1.3,0); 
    \coordinate (D) at (0,1.3,1);

    \draw[black, thick] (A) -- (B);
    \draw[black, thick] (A) -- (C); 
    \draw[black, thick] (A) -- (D);
    \draw[black, thick] (B) -- (C);
    \draw[black, thick] (B) -- (D);
    \draw[black, thick] (C) -- (D);

    \coordinate (A) at (1.3,1.3,1.3); 
    \coordinate (B) at (1.3,0.3,1.3); 
    \coordinate (C) at (0.3,1.3,1.3); 
    \coordinate (D) at (1.3,1.3,0.3);

    \draw[black, thick] (A) -- (B);
    \draw[black, thick] (A) -- (C); 
    \draw[black, thick] (A) -- (D);
    \draw[black, thick] (B) -- (C);
    \draw[black, thick] (B) -- (D);
    \draw[black, thick] (C) -- (D);
    \draw[black,thick] (1.8,0,0) node[anchor=east]{$B_1$};
    \draw[black,thick] (0,1.5,0) node[anchor=west]{$B_2$};
    \draw[black,thick] (-0.3,0,1) node[anchor=south]{$B_3$};
    \draw[black,thick] (1.2,1.2,0) node[anchor=west]{$\xleftarrow{}B_4$};
    \draw[red,thick] (0,1,1) node[anchor=east]{$T_0\to$};
\end{tikzpicture}
    \caption{Dividing the cube}
    \label{figure 1}
\end{figure}

\begin{theorem}\label{4.1 Dp on unit cube}
Let $(t_1,t_2,t_3) \in [0,1]^3$. Under the above notations we have:
\begin{enumerate}
\item If $(t_1,t_2,t_3) \notin T_0$, then $D(t_1,t_2,t_3)$ is a polynomial given by:
\begin{equation*}
D(t_1,t_2,t_3) = \left\{
        \begin{array}{ll}
            t_1t_2 & \quad t_1+t_2 \leq t_3 \\
            t_1t_3 & \quad t_1+t_3 \leq t_2 \\
            t_2t_3 & \quad t_2+t_3 \leq t_1 \\
            1-t_1-t_2-t_3+t_1t_2+t_1t_3+t_2t_3 & \quad t_1+t_2+t_3 \geq 2. 
        \end{array}
    \right.
\end{equation*}  
\item If $(t_1,t_2,t_3) \in T_0$, then $[t_1,t_2,t_3]_p=1/2\cdot \max_{n \geq 1} d^*((t_1,t_2,t_3),1/p^n\Theta)$.
\item We have $d^*(x,\Theta)\leq 1$ for any $x \in \mathbb{R}^3$.
\item If $(t_1,t_2,t_3) \in T_0$, then
$$D_p(t_1,t_2,t_3)=\frac{2t_1t_2+2t_1t_3+2t_2t_3-t_1^2-t_2^2-t_3^2}{4}+[t_1,t_2,t_3]_p^2$$
with $[t_1,t_2,t_3]_p^2\leq \frac{1}{4p^2}< \frac{1}{p^2}$.
\end{enumerate}
\end{theorem}
\begin{proof}
\begin{enumerate}
\item The first three equalities are proved in (2) of Proposition \ref{4.1 basic properties of Dp}, so we prove the last equality. Since $t_1+t_2+t_3\leq 2$, $(1-t_2)+(1-t_3)\leq t_1$, so by reflection
\begin{align*}
D(t_1,t_2,t_3)=D(t_1,1-t_2,1-t_3)+t_1(t_2+t_3-1)\\=(1-t_2)(1-t_3)+t_1(t_2+t_3-1)
\\=t_1t_2+t_1t_3+t_2t_3-t_1-t_2-t_3+1.  
\end{align*}
\item  Points in $T_0$ always satisfy the triangle inequalities, so $[t_1,t_2,t_3]_p$ is well-defined. From \cite[Definition 2.28]{Han92} we have
$[t_1,t_2,t_3]_p=1/2\cdot \max_{n \in \mathbb{Z}} d^*((t_1,t_2,t_3),1/p^n\Theta)$.
But $(t_1,t_2,t_3) \in T_0$, which is a tetrahedron. We see for any $n \geq 0$, $T_0 \subset p^nT_0$. This means $(t_1,t_2,t_3) \in \cap_{n \geq 0}p^nT_0 \subset \cap_{n \geq 0}p^n\Theta$. Thus
$[t_1,t_2,t_3]_p=1/2\cdot \max_{n\geq 1} d^*((t_1,t_2,t_3),1/p^n\Theta)$.
\item This is true since every connected component of $\mathbb{R}^3\backslash \Theta$ is a $d^*$-ball of radius $1$.
\item This comes from (3) and definition of $[t_1,t_2,t_3]_p$.
\end{enumerate}
\end{proof}
From the above theorem we see $D_p(t_1,t_2,t_3)$ converges uniformly to a piecewise polynomial on $[0,1]^3$. We make the following definition:
\begin{definition}\label{4.3 Dinftydef}
We define
$$D_\infty(t_1,t_2,t_3)=\lim_{p \to \infty}D_p(t_1,t_2,t_3)$$
whenever the limit exists at $(t_1,t_2,t_3) \in \mathbb{R}^3$.
\end{definition}
\begin{proposition}\label{4.1 Dinfty on cube and more}
We have    
\begin{equation*}
D_\infty(t_1,t_2,t_3)= \left\{
        \begin{array}{ll}
            t_1t_2 & \quad t_1+t_2 \leq t_3,0 \leq t_1,t_2,t_3 \leq 1 \\
            t_1t_3 & \quad t_1+t_3 \leq t_2,0 \leq t_1,t_2,t_3 \leq 1 \\
            t_2t_3 & \quad t_2+t_3 \leq t_1, 0 \leq t_1,t_2,t_3 \leq 1 \\
            1-t_1-t_2-t_3+t_1t_2+t_1t_3+t_2t_3 & \quad t_1+t_2+t_3 \geq 2,0 \leq t_1,t_2,t_3 \leq 1 \\
            \frac{2t_1t_2+2t_1t_3+2t_2t_3-t_1^2-t_2^2-t_3^2}{4} & \quad (t_1,t_2,t_3) \in T_0 \\
            t_1t_2 & \quad 0\leq t_1,t_2 \leq 1 \leq t_3 \\
            t_1t_3 & \quad 0\leq t_1,t_3 \leq 1 \leq t_2 \\
            t_2t_3 & \quad 0\leq t_2,t_3 \leq 1 \leq t_1.
        \end{array}
    \right.
\end{equation*}
\end{proposition}
\begin{proof}
In the first $4$ regions, the value of $D_p(t_1,t_2,t_3)$ is independent of $p$ by (1) of Theorem \ref{4.1 Dp on unit cube}, so $D_\infty(t_1,t_2,t_3)$ is equal to this value. On the fifth region, $D_p(t_1,t_2,t_3)=\frac{2t_1t_2+2t_1t_3+2t_2t_3-t_1^2-t_2^2-t_3^2}{4}+[t_1,t_2,t_3]^2_p$ and $[t_1,t_2,t_3]^2_p \to 0$ as $p \to \infty$, so $D_\infty(t_1,t_2,t_3)=\frac{2t_1t_2+2t_1t_3+2t_2t_3-t_1^2-t_2^2-t_3^2}{4}$. On the last $3$ regions, $D_p(t_1,t_2,t_3)$ is the product of minimum of the two by (1) and (3) of Proposition \ref{4.1 basic properties of Dp}, and this is independent of $p$, so $D_\infty(t_1,t_2,t_3)=D_p(t_1,t_2,t_3)$ on these regions.     
\end{proof}

\subsection{Existence of $D_\infty$}
In this subsection, we prove the existence of $D_\infty$ on $\mathbb{R}^3$. We first fix a characteristic $p$ and consider $D=D_{T_1+\ldots+T_s}:\mathbb{R}^{s+1} \to \mathbb{R}$. Let $\boldsymbol{\epsilon}=(\epsilon_1,\ldots,\epsilon_{s+1})\in \{0,1\}^{s+1}$.

\begin{definition}[\cite{Han92}, Definition 4.2]\label{4.3 Def of l,phi}
For $\mathbf{t} \in \mathbb{Z}^{s+1}$, we define $l(\mathbf{t}) \in \mathbb{Z}$, $\phi_\mathbf{t}(\mathbf{r}) \in M_0=\oplus_{\epsilon \in \{0,1\}^{s+1}\backslash \textbf{1}}\mathbb{Z}r_1^{\epsilon_1}\ldots r_{s+1}^{\epsilon_{s+1}}\subset\mathbb{Z}[r_1,\ldots,r_{s+1}]$ such that the following equations hold:
\begin{enumerate}
\item When $\sum_i t_i$ is even, $D(\mathbf{t}+\boldsymbol{\epsilon})=l(\mathbf{t})\epsilon_1\ldots\epsilon_{s+1}+\phi_{\mathbf{t}}(\boldsymbol{\epsilon})$.
\item When $\sum_i t_i$ is odd, $D(\mathbf{t}+\boldsymbol{\epsilon})=l(\mathbf{t})(1-\epsilon_1)\epsilon_2\ldots\epsilon_{s+1}+\phi_{\mathbf{t}}(\boldsymbol{\epsilon})$.
\end{enumerate}
\end{definition}
\begin{theorem}[\cite{Han92}, Definition 6.8 and Theorem 6.9]\label{4.3 HanIFSintegral}
For $\mathbf{t},\mathbf{r} \in \mathbb{Z}^{s+1}$, $q$ is a power of $p$ such that $\mathbf{0} \leq \mathbf{r} \leq \mathbf{q}$, we have:
\begin{enumerate}
\item If $\sum_i t_i$ is even, then
$$D(q\mathbf{t}+\mathbf{r})=l(\mathbf{t})D(\mathbf{r})+q^s\phi_\mathbf{t}(\mathbf{r}/q),$$
here $\phi_\mathbf{t}$ is defined as above which is a polynomial independent of $\mathbf{r}$.
\item If $\sum_i t_i$ is odd, then
$$D(q\mathbf{t}+\mathbf{r})=l(\mathbf{t})D(q-r_1,r_2,\ldots,r_{s+1})+q^s\phi_\mathbf{t}(\mathbf{r}/q),$$
here $\phi_\mathbf{t}$ is defined as above which is a polynomial independent of $\mathbf{r}$.
\end{enumerate}    
\end{theorem}
Replace $\mathbf{r}$ by $q\mathbf{r}$, divide by $q^s$ and take limits when $q \to \infty$, then we get:
\begin{theorem}\label{4.3 HanIFS continuous}
For $\mathbf{t} \in \mathbb{Z}^{s+1}$, $\mathbf{r} \in [\mathbf{0},\mathbf{1}]$, then:
\begin{enumerate}
\item If $\sum_i t_i$ is even, then $D(\mathbf{t}+\mathbf{r})=l(\mathbf{t})D(\mathbf{r})+\phi_\mathbf{t}(\mathbf{r})$.
\item If $\sum_i t_i$ is odd, then $D(\mathbf{t}+\mathbf{r})=l(\mathbf{t})D(1-r_1,r_2,\ldots,r_{s+1})+\phi_\mathbf{t}(\mathbf{r})$.
\end{enumerate}    
\end{theorem}

\begin{lemma}\label{4.3 HanIFS stablize}
For fixed $\mathbf{t} \in \mathbb{N}^{s+1}$, the value of $D_p(\mathbf{t})$ calculated over a field of characteristic $p$ is independent of the choice of $p$ for large $p$. This value can also be viewed as a length in characteristic $0$. As a consequence, $l(\mathbf{t})$ and $\phi_\mathbf{t}(\mathbf{r})$ are independent of $p$ for sufficiently large $p$.     
\end{lemma}
\begin{proof}
We see
$$D_p(\mathbf{t})=l_{\mathbb{F}_p}(\mathbb{Z}[T_1,\ldots,T_s]/(T_1^{t_1},\ldots,T_s^{t_s},(T_1+\ldots+T_s)^{t_{s+1}})\otimes_\mathbb{Z}\mathbb{F}_p).$$
Let $R=\mathbb{Z}[T_1,\ldots,T_s]/(T_1^{t_1},\ldots,T_s^{t_s},(T_1+\ldots+T_s)^{t_{s+1}})$, then $R$ is a module-finite $\mathbb{Z}$-algebra. Write $R=F\oplus T$ as $\mathbb{Z}$-module where $F$ is a free $\mathbb{Z}$-module and $T$ is a torsion $\mathbb{Z}$-module. Then for sufficiently large $p$, $T/pT=0$, so $$D_p(\mathbf{t})=l_{\mathbb{F}_p}(R\otimes_\mathbb{Z}\mathbb{F}_p)=\rank_\mathbb{Z}F=\rank_\mathbb{Z}R=l_\mathbb{Q}(R\otimes_\mathbb{Z}\mathbb{Q})$$
is independent of $p$ for $p$ sufficiently large and can be viewed as a length in characteristic $0$. The rest is true since $l(\mathbf{t})$ and $\phi_\mathbf{t}(\mathbf{r})$ only depend on $D(\mathbf{t}+\boldsymbol{\epsilon})$, and there are only finitely many choices of $\boldsymbol{\epsilon}$.
\end{proof}
\begin{proposition}\label{4.3 Dinftyexists}
Let $s=2$. Then $D_\infty$ exists on all of $\mathbb{R}^3$. Moreover, the convergence $D_p \to D_\infty$ is uniform on any bounded region of $\mathbb{R}^3$. 
\end{proposition}
\begin{proof}
Since the relation in Theorem \ref{4.3 HanIFS continuous} and Lemma \ref{4.3 HanIFS stablize} is independent of $p$ for $p \gg 0$, the functional relating $D|_{[\mathbf{r},\mathbf{r}+\mathbf{1}]}$ and $D|_{[\mathbf{0},\mathbf{1}]}$ is independent of $p$ for $p \gg 0$. Since $D_\infty$ exists on $[\mathbf{0},\mathbf{1}]$, it exists on $[\mathbf{r},\mathbf{r}+\mathbf{1}]$ for any $\mathbf{r}\in\mathbb{N}^s$, so it exists on all of $\mathbb{R}^3$. Inside any bounded region there are only finitely many choices of $\mathbf{r}$, so there exists $P \in \mathbb{N}$ such that the functional relating $D|_{[\mathbf{r},\mathbf{r}+\mathbf{1}]}$ and $D|_{[\mathbf{0},\mathbf{1}]}$ is independent of $p$ for any $p\geq P$ and any $\mathbf{r}$ lying in this region. In this case, the convergence of $D_p$ is reduced to the convergence on $[0,1]^3$, and we see $D_p \to D_\infty$ uniformly on $[0,1]^3$, so we are done.
\end{proof}
\subsection{Attached points and the geometry of $\Theta$}
In this subsection, we use $C$ to indicate the unit cube $[0,1]^3$ instead of a constant. Let $\phi=T_1+T_2$ and consider the kernel function $D_p$ in characteristic $p$ and the limit kernel function $D_\infty$ restricted to the unit cube. In Subsection \ref{section 4.1}, we have seen the following fact:
\begin{enumerate}
\item $D_p\geq D_\infty$;
\item $D_p \to D_\infty$ uniformly on the cube $[0,1]^3$;
\item $D_p=D_\infty$ for any $p$ on $B_1\sim B_4$.
\end{enumerate}
Note that (1) is saying $D_p$ is no less than $D_\infty$. Thus, we may expect that certain $h$-function in characteristic $p$ is no less than its limit. We can also check points in $T_0$ where the value of $D_p$ differs from $D_\infty$, which may lead to strict inequalities. We assume $p\geq 3$ throughout this subsection unless otherwise stated, since many properties here fail for $p=2$.
\begin{definition}\label{4.4 attacheddef}
We say a point $x \in C=[0,1]^3$ is an \textbf{attached point} in characteristic $p$ if $D_p(x)=D_\infty(x)$, otherwise it is an \textbf{unattached point}.    
\end{definition}
We use the notations defined in Definition \ref{4.1 syzygy gap definition0} and Definition \ref{4.1 geometry: faces F}, then by Proposition \ref{4.1 syzygy gap definition}, Theorem \ref{4.1 Dp on unit cube} and Proposition \ref{4.1 Dinfty on cube and more}, we have for $x \in T_0$, $[x]_p=1/2\cdot \max_{n\geq 1} d^*(x,1/p^n\Theta)=1/2\cdot \max_{n\geq 1} 1/p^nd^*(p^nx,\Theta)$ and $D_p(x)=D_\infty(x)+[x]_p^2$. Thus we have:
\begin{proposition}\label{4.4 attachedgeometry}
Let $x \in C$. If $x \in T_0$, then it is attached if and only if $[x]_p=0$, if and only if $p^nx\in \Theta$ lies in a tetrahedron for any $n \geq 1$. Also, points in the closure of $B_1\sim B_4$ and points in $\partial T_0$ are attached.    
\end{proposition}

\begin{example}
Let $p$ be an odd prime and $x \in [0,1]$. Then $(1/2,1/2,x)$ is a segment consisting of attached points. This is true since for any $a,b\in \mathbb{Z},x \in \mathbb{R}$, $(a+1/2,b+1/2,x) \in \Theta$.    
\end{example}

\begin{figure}
    \centering
    \begin{tikzpicture}[scale=0.8]

\draw[->, line width=0.8pt] (-1.2,0) -- (3.2,0) node[right] {$x$};
\draw[->, line width=0.8pt] (0,-1.2) -- (0,3.2) node[above] {$y$};
\foreach \x in {-1,0,1,2,3}
    \draw (\x,0.1) -- (\x,-0.1) node[below] {$\x$};
\foreach \y in {-1,0,1,2,3}
    \draw (0.1,\y) -- (-0.1,\y) node[left] {$\y$};

\foreach \x in {0,1,2,3}
  \foreach \y in {-1,0,1,2}
    \filldraw[fill=gray!30,opacity=0.5] 
    (\x-0.5,\y) -- (\x,\y+0.5) -- (\x-0.5,\y+1) -- (\x-1,\y+0.5) -- cycle;

\draw[->, line width=0.8pt] (-7.2,0) -- (-2.8,0) node[right] {$x$};
\draw[->, line width=0.8pt] (-6,-1.2) -- (-6,3.2) node[above] {$y$};
\foreach \x in {-1,0,1,2,3}
    \draw (\x-6,0.1) -- (\x-6,-0.1) node[below] {$\x$};
\foreach \y in {-1,0,1,2,3}
    \draw (-5.9,\y) -- (-6.1,\y) node[left] {$\y$};

\draw[line width=1.5pt, black] (-7,-1) -- (-3,3); 
\draw[line width=1.5pt, black] (-7.2,1.2) -- (-4.8,-1.2);      
\draw[line width=1.5pt, black] (-5.2,-1.2) -- (-2.8,1.2);  
\draw[line width=1.5pt, black] (-7.2,3.2) -- (-5,1) -- (-2.8,-1.2);    
\draw[line width=1.5pt, black] (-7.2,0.8) -- (-4.8,3.2);         
\draw[line width=1.5pt, black] (-5.2,3.2) -- (-2.8,0.8); 

\draw[->, line width=0.8pt] (4.6,0) -- (9.4,0) node[right] {$x$};
\draw[->, line width=0.8pt] (6,-1.4) -- (6,3.4) node[above] {$y$};
\foreach \x in {-1,0,1,2,3}
    \draw (\x+6,0.1) -- (\x+6,-0.1) node[below] {$\x$};
\foreach \y in {-1,0,1,2,3}
    \draw (6.1,\y) -- (5.9,\y) node[left] {$\y$};

\draw[line width=1.5pt, black] (4.8,-0.2) -- (8.2,3.2); 
\draw[line width=1.5pt, black] (4.8,1.8) -- (6.2,3.2); 
\draw[line width=1.5pt, black] (5.8,-1.2) -- (9.2,2.2);
\draw[line width=1.5pt, black] (7.8,-1.2) -- (9.2,0.2); 
\draw[line width=1.5pt, black] (4.8,0.2) -- (6.2,-1.2);
\draw[line width=1.5pt, black] (4.8,2.2) -- (8.2,-1.2);
\draw[line width=1.5pt, black] (5.8,3.2) -- (9.2,-0.2);
\draw[line width=1.5pt, black] (7.8,3.2) -- (9.2,1.8);
\end{tikzpicture}  
    \caption{Sections $\Theta \cap \{t_3=a\}$ for different $a$ mod $2$: left.$t_3=0$ or $2$, middle.$t_3=1/2$ or $3/2$, right. $t_3=1$}
    \label{fig:Theta intersect t3plane}
\end{figure}

\begin{example}
Figure \ref{fig:Theta intersect t3plane} shows the section $\Theta\cap \{t_3=a\}, a=0,1/2,1,3/2,2$. We see inside $[0,1]^3$, $\Theta\cap \{t_3=0\}$ consists of segment $\{(t,t),0\leq t \leq 1\}$, and $\Theta\cap \{t_3=1\}$ consists of segment $\{(t,1-t),0\leq t \leq 1\}$. This cycling pattern has period $2$; thus, inside $C$, any segment parallel to $t_3$-axis of length at least $2$ that falls into $\Theta$ is contained in the line $(1/2,1/2,t_3)$. Similarly, we see if a segment is parallel to $t_1$, $t_2$, or $t_3$-axis, has length at least $2$, and is contained in $\Theta$, then the other two coordinates must be $1/2+a$ and $1/2+b$ for $a,b \in \mathbb{Z}$.
\end{example}
\begin{definition}
Let $S \subset C$ be a segment. We say $S$ is \textbf{an attached segment}, if $S$ consists of attached points. Equivalently, either $S$ lies in the union of $B_1\sim B_4$, or $p^n(S\cap T_0) \subset \Theta$ for all $n\geq 1$. Otherwise, we say $S$ is unattached. We say $S$ is an eventually attached segment, if for large enough $n$ we have $p^n(S\cap T_0) \subset \Theta$, otherwise we say $S$ is eventually unattached. We say a line or a segment is upright if it is parallel to $t_1,t_2$ or $t_3$ coordinate, otherwise we say it is skew.
\end{definition}
\begin{remark}
From the definition we see $S \subset C$ is attached if and only if $S\cap T_0$ is attached, so we may assume $S \subset T_0$ when talking about the attaching property. Also, attached segments are eventually attached.
\end{remark}
\begin{proposition}\label{4.4 eventuallyattachedcriterion}
Suppose $S=\{(a,b,x)\}$ is an upright segment in $T_0$ with parameter $x$ for fixed $a,b$, then $S$ is eventually attached if and only if $ap^m,bp^m \in 1/2+\mathbb{Z}$ for some $m \in \mathbb{N}$.
\end{proposition}
\begin{proof}
We see that multiples of upright segments are still upright. We assume the length of $p^nS$ is at least $2$. By Figure \ref{fig:Theta intersect t3plane}, we see the only candidate for the other two coordinates of upright segments in $\Theta$, whose length in $t_3$-direction is at least $2$, are half integers. Thus $ap^n$, $bp^n$ are all half integers for large $n$. In particular, this holds for one integer $n=m$. The converse of the above also holds for $p$ odd, that is, if $ap^m,bp^m \in 1/2+\mathbb{Z}$, then $ap^n,bp^n \in 1/2+\mathbb{Z}$ for $n \geq m$, so $(ap^n,bp^n,x) \in \Theta$ for any $x$. So we are done.   
\end{proof}
By the above proposition, the attaching property for upright segments is clear. Now we consider whether the skew segments inside $T_0$ are attached. In general, if a skew segment $S$ satisfies $p^mS \subset F=\partial\Theta$ for some $m \in \mathbb{N}$, then for any $n \geq m$, $p^nS \subset p^{n-m}F \subset \Theta$. Therefore, it is eventually attached. We consider these cases of attached segments as \textit{trivial}. 
\begin{proposition}\label{4.4 density of eventually unattached segment}
Let $S \subset T_0$ be a segment. There are only $3$ possibilities:
\begin{enumerate}
\item $S$ is eventually unattached.
\item $S$ is eventually attached, upright, and the two fixed components multiplied by a $p$-power are half integers.
\item $S$ is a eventually attached skew segment which is trivial.
\end{enumerate}
Moreover, if $S \subset T_0$ is an eventually unattached segment, then the set of unattached points is dense in $S$. Also, the set of unattached point is dense in $T_0 \cap H$ for any plane $H$ which is not parallel to planes contained in $F$.
\end{proposition}
\begin{proof}
Any line has a parameter equation $\mathbf{r}=r\cdot(a_1,a_2,a_3)+\mathbf{r_0}$ where $(a_1,a_2,a_3)\neq 0$. Since $\Theta$ is symmetric, we may assume $|a_3|\geq |a_1|,|a_2|$ by permuting indices, and this does not change the attaching property. Suppose $S$ is a segment on this line from $\mathbf{u}=(u_1,u_2,u_3)$ to $\mathbf{v}=(v_1,v_2,v_3)$ which is eventually attached, then there is $n$ large enough such that $p^nd^*(\mathbf{u},\mathbf{v})>12$. We fix such $n$ and write $S'=p^nS$. We see $|u_3-v_3|\geq |u_1-v_1|, |u_2-v_2|$, thus $p^n|u_3-v_3|\geq 4$, so the segment $S'$ intersects with at least $4$ consecutive planes $t_3=a$ where $a \in \mathbb{Z}$. That is, $S'$ intersects with $t_3=a$, $t_3=a+1$, $t_3=a+2$, $t_3=a+3$ for some $a$. We claim that for four such planes, if $S'\cap\{a \leq t_3 \leq a+3\}$ lies in the region $\Theta\cap \{a \leq t_3 \leq a+3\}$, then $S'$ is either eventually attached upright as in case (2) or trivially skew as in case (3). 

We first check the point $\mathbf{w}=p^nS\cap \{t_3=a+1\}$. It lies in some cube whose vertices are lattice points, that is, $\mathbf{w} \in C'=[\lfloor\mathbf{w} \rfloor,\lfloor\mathbf{w} \rfloor+\mathbf{1}]$. The assumption in the claim says $S'\cap C'\subset\Theta\cap C'$ which is a translation of either $T_0$ or $-T_0$ depending on parity of $||\lfloor\mathbf{w} \rfloor||_1$.
\begin{figure}
    \centering
    \begin{tikzpicture}[scale=2]
    \draw[thick] (0,0,0) -- (1,0,0) node[anchor=north east]{};
    \draw[thick] (0,0,0) -- (0,1,0) node[anchor=south west]{};
    \draw[thick] (0,0,0) -- (0,0,1) node[anchor=south]{};

    \draw[black, thick] (0,0,0) -- (1,0,0) -- (1,1,0) -- (0,1,0) -- cycle;
    \draw[black, thick] (0,0,1) -- (1,0,1) -- (1,1,1) -- (0,1,1) -- cycle;
    \draw[black, thick] (0,0,0) -- (0,0,1);
    \draw[black, thick] (1,0,0) -- (1,0,1);
    \draw[black, thick] (0,1,0) -- (0,1,1);
    \draw[black, thick] (1,1,0) -- (1,1,1);

    \foreach \x in {0,1}
    \foreach \y in {0,1}
    \foreach \z in {0,1} {
        \filldraw (\x,\y,\z) circle (0.5pt);
    }

    \draw[red, thick] (0,0,0) -- (0,1,1) -- (1,0,1) -- cycle;
    \draw[red, thick] (0,0,0) -- (1,0,1) -- (1,1,0) -- cycle;
    \draw[red, thick] (0,0,0) -- (0,1,1) -- (1,1,0) -- cycle;
    \draw[red, thick] (0,1,1) -- (1,0,1) -- (1,1,0) -- cycle;

    \draw[blue,line width=1mm] (0.25,0.25,0)--(0.75,0.25,1);

    \filldraw[blue] (0.25,0.25,0) circle (1pt) node[anchor=north]{};
    \filldraw[blue] (0.75,0.25,1) circle (1pt) node[anchor=south]{};

    \filldraw[fill=gray!30,opacity=0.3] (0,0,0) -- (0,1,1) -- (1,0,1) -- cycle;
    \filldraw[fill=gray!30,opacity=0.3] (0,0,0) -- (1,0,1) -- (1,1,0) -- cycle;
    \filldraw[fill=gray!30,opacity=0.3] (0,0,0) -- (0,1,1) -- (1,1,0) -- cycle;
    \filldraw[fill=gray!30,opacity=0.3] (0,1,1) -- (1,0,1) -- (1,1,0) -- cycle;

    \filldraw[blue] (0,0,0) circle (1pt) node[anchor=south east]{};
    \filldraw[black] (0,1,1) circle (1pt) node[anchor=south west]{};
    \filldraw[black] (1,0,1) circle (1pt) node[anchor=south east]{};
    \filldraw[black] (1,1,0) circle (1pt) node[anchor=north west]{};

    \filldraw[black] (1,0,0) circle (1pt) node[anchor=south east]{};
    \filldraw[blue] (1,1,1) circle (1pt) node[anchor=north west]{};
    \filldraw[black] (0,0,1) circle (1pt) node[anchor=south east]{};
    \filldraw[black] (0,1,0) circle (1pt) node[anchor=north west]{};

    \node at (2,1,1) {\textcolor{blue}{$(r_1+1,r_2+1,r_3+1)$}};
    \node at (1.7,0,0.5) {\textcolor{blue}{$(r_1+1-v,r_2+v,r_3+1)$}};
    \node at (-0.6,0,0) {\textcolor{blue}{$(r_1,r_2,r_3)$}};
    \node at (0,1.2,0) {\textcolor{blue}{$(r_1+u,r_2+u,r_3)$}};

    \draw[->] (0.25,1.1,0) -- (0.25,0.35,0);
\end{tikzpicture}
        \caption{A demonstration of $\Theta \cap C'$ when $r_1+r_2+r_3$ is even. Here the gray shape represents for $C'\cap\Theta$ which is a tetrahedron, and the red segments form the $1$-skeleton of $T_0$. We see $\partial C'\cap\Theta=\partial C'\cap (C'\cap \Theta)$ is the $1$-skeleton. The endpoints of the thick blue segment fall on this $1$-skeleton.}    
    \label{fig:1 skeleton}
\end{figure}

\begin{figure}
    \centering
\begin{tikzpicture}[scale=2, thick]

    \draw[->] (-0.5,0) -- (1.5,0) node[right]{$t_1$};
    \draw[->] (0,-0.5) -- (0,1.5) node[above]{$t_2$};
    
    \draw[red] (0,0) -- (1,1);
    \draw[red] (0,1) -- (1,0);
    \draw (0,1) -- (1,1);
    \draw (1,0) -- (1,1);

    \filldraw (0,0) circle (1pt) node[below left]{$(r_1,r_2)$};
    \filldraw (1,0) circle (1pt);
    \filldraw (1,1) circle (1pt) node[above right]{$(r_1+1,r_2+1)$};
    \filldraw (0,1) circle (1pt);
    \filldraw[blue] (0.25,0.25) circle (1pt);
    \filldraw[blue] (0.75,0.25) circle (1pt);

    \draw[blue, thick] (0.25,0.25) -- (0.75,0.25);
    \draw[blue, thick,dashed] (0.25,0) -- (0.25,0.25);
    \draw[blue, thick,dashed] (0.75,0) -- (0.75,0.25);

    \node at (0.25,-0.15) {\tiny $r_1+u$};
    \node at (0.75,-0.15) {\tiny $r_1+1-v$};

\end{tikzpicture}
        \caption{Projection of Figure \ref{fig:1 skeleton} onto $t_1-t_2$ plane. We call the length of the left dashed segment $u$ and the length of the right $v$. We may assume $0<u,v\leq 1/2$ by symmetry.} 
    \label{fig:1 skeleton projection}
\end{figure}

\begin{figure}
    \centering
\begin{tikzpicture}[scale=2, thick]

    \draw[->] (-0.5,0) -- (1.5,0) node[right]{$t_1$};
    \draw[->] (0,-0.5) -- (0,1.5) node[above]{$t_2$};
    
    \draw[red,line width=1mm] (0,0) -- (1,1)--(2,0)--(1,-1)--cycle;
    \draw[red,thick] (0,1) -- (1,0)--(0,-1)--(-1,0)--cycle;
    \draw (0,1) -- (1,1);
    \draw (1,0) -- (1,1);

    \filldraw (0,0) circle (1pt) node[below left]{$(r_1,r_2)$};
    \filldraw (1,0) circle (1pt);
    \filldraw (1,1) circle (1pt) node[above right]{$(r_1+1,r_2+1)$};
    \filldraw (0,1) circle (1pt);
    \filldraw[blue] (0.25,0.25) circle (1.5pt);
    \filldraw[blue] (0.75,0.25) circle (1pt);
    \filldraw[blue] (1.25,0.25) circle (1.5pt);
    \filldraw[blue] (-0.25,0.25) circle (1pt);

    \draw[blue, thick] (-0.25,0.25) -- (1.25,0.25);

\end{tikzpicture}
        \caption{The four blue points represent the intersection of $S'$ with four consecutive $t_3$-planes. The second and fourth points must fall on the thick diamond, and the first and third points must fall on the thin diamond. Thus, if neither of the second and the third blue points coincides with a black point, then the second and third blue points must both be at the center of the black square.} 
    \label{fig:1 skeleton projection proof}
\end{figure}

Since $S'$ is a segment and $C'$ is convex, the two endpoints of $S'\cap C'$ are the unique two points lying in the intersection $S'\cap \partial C'$.  Especially, we see $S' \cap \partial C' \subset S' \cap (\partial C'\cap \Theta)$. From Figure \ref{fig:1 skeleton}, we see $
\partial C' \cap \Theta$ is just the $1$-skeleton of the tetrahedron $C'\cap \Theta$, so $S'$ is the segment adjoining two points on this $1$-skeleton. If the two points lie on two adjacent edges of the tetrahedron, then $S'$ lies in the faces of the tetrahedron. This is saying $S' \in F$ and $S$ is a trivial skew eventually attached line. Otherwise, the two endpoints of $S'$ must come from the interior of the two opposite edges. The blue segment in Figure \ref{fig:1 skeleton} is one such example. And also see Figure \ref{fig:1 skeleton projection} for the projection onto the $t_1-t_2$ plane, which gives more explanation.

We check the case where $C'=[\mathbf{r},\mathbf{r}+\mathbf{1}]$ and $T'=C'\cap \Theta$ is a translation of $T_0$; the case of $-T_0$ can be proved similarly using symmetry. In this case, the bottom edge of $T'$ connects $(r_1,r_2,r_3)$ and $(r_1,r_2+1,r_3+1)$, and the top edge connects $(r_1+1,r_2,r_3+1)$ and $(r_1,r_2+1,r_3+1)$. We assume $S'$ is the segment between $(r_1+u,r_2+u,r_3)$ and $(r_1+1-v,r_2+v,r_3+1)$, where $0<u,v<1$ are real numbers. Taking reflection if necessary, we may assume $u,v\leq 1/2$. One can refer to Figure \ref{fig:1 skeleton projection proof} for a demonstration of the above notations, projected onto $t_1-t_2$ plane. We see if $u=v=1/2$, it is an upright segment contained in $\Theta$, otherwise we have:
\begin{enumerate}
\item If $0<u,v<1/2$, then $S'\cap \{t_3=a\}\nsubseteq \Theta \cap \{t_3=a\}$ and $S'\cap \{t_3=a+3\}\nsubseteq \Theta \cap \{t_3=a+3\}$. We see in this case the fourth point does not lie on the thick diamond, and the first point does not lie on the thin diamond.
\item If $0<u<1/2,v=1/2$, then $S'\cap \{t_3=a\}\nsubseteq \Theta \cap \{t_3=a\}$, but $S'\cap \{t_3=a+3\}\subset \Theta \cap \{t_3=a+3\}$. We see in this case the fourth point lies on the thick diamond, but the first point does not lie on the thin diamond.
\item If $0<v<1/2,u=1/2$, then $S'\cap \{t_3=a\}\subset \Theta \cap \{t_3=a\}$, but $S'\cap \{t_3=a+3\}\nsubseteq \Theta \cap \{t_3=a+3\}$. We see in this case the fourth point does not lie on the thick diamond, but the first point lies on the thin diamond.
\end{enumerate}
Thus, if $S'\cap \{a \leq t_3 \leq a+3\} \subset \Theta$, then $S'$ must be trivially skew or upright. When $S'$ is upright, its fixed $t_1$ and $t_2$ coordinate must be half integers.

Finally, we deal with the density of unattached points. We assume $S \subset T_0$ is not trivially skew or upright and eventually attached. Then by the previous argument, for large enough $n$ and $a$ such that $p^nS$ intersects with $4$ consecutive planes $t_3=a,a+1,a+2,a+3$, either its intersection with one of the planes does not fall in $\Theta$, or there is a point $y$ on the boundary of a cube inside planes $t_3=a+1,a+2$ such that $y \in S'$ but $y \notin\Theta$. So for large $n$ and any $x \in S$, there exists $y \in p^nS\backslash \Theta$ such that the $t_3$-coordinate of $p^nx$ and $y$ differ by at most $2$. For such $p^nx$ and $y$, since the differences in $t_1$, $t_2$ coordinates are no larger than that in $t_3$-direction, $||p^nx-y||_1\leq 6$. So $||x-1/p^ny||_1\leq 6/p^n$. Since $y \notin \Theta$, $1/p^ny$ is unattached. When $n\to \infty$, $6/p^n\to0$, so there is an unattached point in an arbitrary small neighbourhood of $x$, that is, the set of unattached points is dense in $S$.

For a plane $H$ which is not parallel to planes in $F$, we can choose a direction $\mathbf{a}$ parallel to $H$, but is not upright or parallel to planes in $F$. Then $T_0\cap H$ is a union of disjoint segments in direction $\mathbf{a}$. Every such segment is eventually unattached, so the set of unattached points is dense in these segments. Therefore, the set of unattached points is dense in their union, that is, $T_0\cap H$.
\end{proof}

\subsection{Properties of $\frac{\partial}{\partial r^\pm}D_p(t_1,t_2,r)$}
In this subsection, we prove some properties of $\frac{\partial}{\partial r^\pm}D_p(t_1,t_2,r)$ and $\frac{\partial}{\partial r^\pm}D_\infty(t_1,t_2,r)$, whose existence is guaranteed by convexity.
\begin{lemma}[\cite{Han92}, Lemma 4.8]\label{4.5 Han's lemma on partial i partial j}
For any integer $t_1,t_2,t_3\geq 0$, any characteristic $p$ and $i \neq j$,
$$0 \leq D_p(\mathbf{t})-D_p(\mathbf{t}+\mathbf{v}_i)-D_p(\mathbf{t}+\mathbf{v}_j)+D_p(\mathbf{t}+\mathbf{v}_i+\mathbf{v}_j)\leq 1.$$
\end{lemma}
\begin{corollary}\label{4.5 bound on double difference}
Let $t_i,t_j \in \mathbb{N}$. Then, for any $\mathbf{t}\geq \mathbf{0}$, 
$$0 \leq D_p(\mathbf{t})-D_p(\mathbf{t}+t_i\mathbf{v}_i)-D_p(\mathbf{t}+t_j\mathbf{v}_j)+D_p(\mathbf{t}+t_i\mathbf{v}_i+t_j\mathbf{v}_j)\leq t_it_j.$$
\end{corollary}
\begin{proof}
We have for any $0\leq n_1<t_i$, $0 \leq n_2<t_j$,
\begin{align*}
0 \leq D_p(\mathbf{t}+n_1\mathbf{v}_i+n_2\mathbf{v}_j)-D_p(\mathbf{t}+(n_1+1)\mathbf{v}_i+n_2\mathbf{v}_j)\\
-D_p(\mathbf{t}+n_1\mathbf{v}_i+(n_2+1)\mathbf{v}_j)+D_p(\mathbf{t}+(n_1+1)\mathbf{v}_i+(n_2+1)\mathbf{v}_j)\leq 1.
\end{align*}
Taking sum over all $n_1,n_2$, we get the result.
\end{proof}
\begin{corollary}\label{4.5 continuity of partial derivative}
The functions $\frac{\partial}{\partial t_i^\pm}D_p(\mathbf{t}),\frac{\partial}{\partial t_i^\pm}D_\infty(\mathbf{t})$ are Lipschitz continuous with respect to all the coordinates except for the $i$-th coordinate. In particular, for fixed $r$, $(t_1,t_2)\to \frac{\partial}{\partial t_i^\pm}D_p(t_1,t_2,r),\frac{\partial}{\partial t_i^\pm}D_\infty(t_1,t_2,r)$ are continuous on $(0,\infty)^2$.
\end{corollary}
\begin{proof}
First, take any $\mathbf{t}\in \mathbb{Z}[1/p]\cap [0,\infty)^3$, any $i \neq j$, any $t_i,t_j \in \mathbb{Z}[1/p]\cap [0,\infty)$. For sufficiently large $q$, $q\mathbf{t} \in \mathbb{Z}^3, qt_i \in \mathbb{Z},qt_j \in \mathbb{Z}$, so
$$0 \leq D_p(q\mathbf{t})-D_p(q\mathbf{t}+qt_i\mathbf{v}_i)-D_p(q\mathbf{t}+qt_j\mathbf{v}_j)+D_p(q\mathbf{t}+qt_i\mathbf{v}_i+qt_j\mathbf{v}_j)\leq q^2t_it_j.$$
Dividing by $q^2$, we get
$$0 \leq D_p(\mathbf{t})-D_p(\mathbf{t}+t_i\mathbf{v}_i)-D_p(\mathbf{t}+t_j\mathbf{v}_j)+D_p(\mathbf{t}+t_i\mathbf{v}_i+t_j\mathbf{v}_j)\leq t_it_j.$$
Next, since $\mathbb{Z}[1/p]$ is dense in $\mathbb{R}$ and $D_p$ is continuous,
$$0 \leq D_p(\mathbf{t})-D_p(\mathbf{t}+t_i\mathbf{v}_i)-D_p(\mathbf{t}+t_j\mathbf{v}_j)+D_p(\mathbf{t}+t_i\mathbf{v}_i+t_j\mathbf{v}_j)\leq t_it_j$$
holds for any $\mathbf{t}>0,t_i>0,t_j>0$. We first fix $\mathbf{t}$ and rewrite the inequality as
$$0 \leq -\frac{D_p(\mathbf{t}+t_i\mathbf{v}_i)-D_p(\mathbf{t})}{t_i}+\frac{D_p(\mathbf{t}+t_i\mathbf{v}_i+t_j\mathbf{v}_j)-D_p(\mathbf{t}+t_j\mathbf{v}_j)}{t_i}\leq t_j.$$
Let $t_i \to 0^+$, we get
$$0 \leq \frac{\partial}{\partial t_i^+}D_p(\mathbf{t}+t_j\mathbf{v}_j)-\frac{\partial}{\partial t_i^+}D_p(\mathbf{t})\leq t_j.$$
Then, we take $t_i$ sufficiently small such that $\mathbf{t}-t_i\mathbf{v}_i>\mathbf{0}$. In this case, we replace $\mathbf{t}$ with $\mathbf{t}-t_i\mathbf{v}_i$ to get that
$$0 \leq D_p(\mathbf{t}-t_i\mathbf{v}_i)-D_p(\mathbf{t})-D_p(\mathbf{t}-t_i\mathbf{v}_i+t_j\mathbf{v}_j)+D_p(\mathbf{t}+t_j\mathbf{v}_j)\leq t_it_j$$
holds for any $\mathbf{t}>0,t_i>0,t_j>0$. We first fix $\mathbf{t}$ and rewrite the inequality as
$$0 \leq -\frac{D_p(\mathbf{t})-D_p(\mathbf{t}-t_i\mathbf{v}_i)}{t_i}+\frac{D_p(\mathbf{t}+t_j\mathbf{v}_j)-D_p(\mathbf{t}-t_i\mathbf{v}_i+t_j\mathbf{v}_j)}{t_i}\leq t_j.$$
Let $t_i \to 0^+$, we get
$$0 \leq \frac{\partial}{\partial t_i^-}D_p(\mathbf{t}+t_j\mathbf{v}_j)-\frac{\partial}{\partial t_i^-}D_p(\mathbf{t})\leq t_j.$$
Thus, both partial derivatives $\frac{\partial}{\partial t_i^\pm}D_p(\mathbf{t})$ are Lipschitz continuous with respect to all the coordinates except for the $i$-th coordinate. For the limit kernel function, note that
$$0 \leq D_p(\mathbf{t})-D_p(\mathbf{t}+t_i\mathbf{v}_i)-D_p(\mathbf{t}+t_j\mathbf{v}_j)+D_p(\mathbf{t}+t_i\mathbf{v}_i+t_j\mathbf{v}_j)\leq t_it_j$$
holds for any $\mathbf{t}>0,t_i>0,t_j>0$ and any $p$, so taking $p\to \infty$ yields
$$0 \leq D_\infty(\mathbf{t})-D_\infty(\mathbf{t}+t_i\mathbf{v}_i)-D_\infty(\mathbf{t}+t_j\mathbf{v}_j)+D_\infty(\mathbf{t}+t_i\mathbf{v}_i+t_j\mathbf{v}_j)\leq t_it_j.$$
The same argument for $D_p$ shows that $D_\infty$ is Lipschitz continuous with respect to all but the $i$-th coordinate.
\end{proof}

\begin{lemma}\label{4.5 lem: converge of concave leads to converge of derivative}
Let $\phi_i$ be a sequence of concave functions on $[a,b]$. Suppose $\phi_i \to \phi$ on $[a,b]$. Then for any $x \in (a,b)$,
$$\underline{\lim}_{i \to \infty}\phi'_{i,+}(x)\geq \phi'_+(x)$$
and
$$\overline{\lim}_{i \to \infty}\phi'_{i,-}(x)\leq \phi'_{-}(x).$$
In particular, if $\phi'(x)$ exists, then
$$\lim_{i \to \infty}\phi'_{i,+}(x)=\lim_{i \to \infty}\phi'_{i,-}(x)=\phi'(x).$$
\end{lemma}
\begin{proof}
Suppose the first inequality fails, that is, there is $\epsilon>0$ and a sequence $i_n \to \infty$ such that
$$\phi'_{i_n,+}(x)< \phi'_+(x)-\epsilon.$$
Since $\phi'_+(x)=\phi'(x^+)$ and $\phi'_{i_n,+}$ is decreasing, we may choose $\delta$ such that for any $y \in [x,x+\delta]$,
$$\phi'_{i_n,+}(y)\leq\phi'_{i_n,+}(x)< \phi'(y)-1/2\epsilon.$$
Since $\phi_{i_n},\phi$ are all concave, they are absolutely continuous. So
$$\phi_{i_n}(x+\delta)-\phi_{i_n}(x)=\int_x^{x+\delta}\phi'_{i_n}(y)dy$$
and
$$\phi(x+\delta)-\phi(x)=\int_x^{x+\delta}\phi'(y)dy.$$
So
\begin{align*}
(\phi(x+\delta)-\phi(x))-(\phi_{i_n}(x+\delta)-\phi_{i_n}(x))=\int_x^{x+\delta}(\phi'(y)-\phi'_{i_n}(y))dy\\
\geq \int_x^{x+\delta}1/2\epsilon dy=\delta\epsilon/2>0.
\end{align*}
Taking limit when $n \to 0$, we get a contradiction. So the first inequality holds. The second inequality can be proved similarly considering $[x-\delta,x]$. The last equality comes from the inequality
$$\overline{\lim}_{i \to \infty}\phi'_{i,+}(x)\leq \underline{\lim}_{i \to \infty}\phi'_{i,-}(x).$$
\end{proof}
\begin{remark}
We cannot expect
$$\lim_{i \to \infty}\phi'_{i,+}(x)= \phi'_+(x)$$
and
$$\lim_{i \to \infty}\phi'_{i,-}(x)= \phi'_{-}(x)$$
at $x$ where $\phi$ is not differentiable. For example, consider the sequence of concave functions
$$\phi_i(x)=\begin{cases}
x & x\leq -1/i\\
-1/i & -1/i \leq x \leq 1/i\\
-x & x \geq 1/i.
\end{cases}$$
Then
$$\phi(x)=\lim_{i \to \infty}\phi_i(x)=\begin{cases}
x & x\leq 0\\
-x & x \geq 0.
\end{cases}$$
And $\phi'_{i,+}(0)=\phi'_{i,-}(0)=0$ for any $i$, but $\phi'_+(0)=-1$ and $\phi'_-(0)=1$.
\end{remark}
\begin{lemma}\label{4.5 convergence of partial derivative}
For any $t_1,t_2\geq 0,r \geq 0$
$$\lim_{p \to \infty}\frac{\partial}{\partial r^+}D_p(t_1,t_2,r)=\frac{\partial}{\partial r^+}D_\infty(t_1,t_2,r)$$
and when $r>0$,
$$\lim_{p \to \infty}\frac{\partial}{\partial r^-}D_p(t_1,t_2,r)=\frac{\partial}{\partial r^-}D_\infty(t_1,t_2,r).$$
\end{lemma}
\begin{proof}
By Theorem \ref{4.3 HanIFS continuous} and Lemma \ref{4.3 HanIFS stablize} on the restriction of $D_p$ on different lattice cubes, it suffices to prove the equality in $[0,1]^3$, and we don't need to consider the case $\frac{\partial}{\partial r^+}$ at $r=1$ which translates to $\frac{\partial}{\partial r^+}$ at $r=0$ and $\frac{\partial}{\partial r^-}$ at $r=0$ which is always $0$.

Recall that in Proposition \ref{4.1 Dinfty on cube and more} we have proved
\begin{equation*}
D_\infty(t_1,t_2,t_3)= \left\{
        \begin{array}{ll}
            t_1t_2 & \quad t_1+t_2 \leq t_3,0 \leq t_1,t_2,t_3 \leq 1 \\
            t_1t_3 & \quad t_1+t_3 \leq t_2,0 \leq t_1,t_2,t_3 \leq 1 \\
            t_2t_3 & \quad t_2+t_3 \leq t_1, 0 \leq t_1,t_2,t_3 \leq 1 \\
            1-t_1-t_2-t_3+t_1t_2+t_1t_3+t_2t_3 & \quad t_1+t_2+t_3 \geq 2, \\
             &\quad 0 \leq t_1,t_2,t_3 \leq 1\\
            \frac{2t_1t_2+2t_1t_3+2t_2t_3-t_1^2-t_2^2-t_3^2}{4} & \quad (t_1,t_2,t_3) \in T_0.
        \end{array}
    \right.
\end{equation*}
A simple calculation yields
\begin{equation*}
\frac{\partial}{\partial t_3}D_\infty(t_1,t_2,t_3)= \left\{
        \begin{array}{ll}
            0 & \quad t_1+t_2 \leq t_3,0 \leq t_1,t_2,t_3 \leq 1 \\
            t_1 & \quad t_1+t_3 \leq t_2,0 \leq t_1,t_2,t_3 \leq 1 \\
            t_2 & \quad t_2+t_3 \leq t_1, 0 \leq t_1,t_2,t_3 \leq 1 \\
            -1+t_1+t_2 & \quad t_1+t_2+t_3 \geq 2,0 \leq t_1,t_2,t_3 \leq 1 \\
            \frac{t_1+t_2-t_3}{2} & \quad (t_1,t_2,t_3) \in T_0. 
        \end{array}
    \right.
\end{equation*}
We can check that $\frac{\partial}{\partial t_3^\pm}D_\infty(t_1,t_2,t_3)$ is continuous with respect to both $t_1,t_2,t_3$ at $\partial T_0$, so $\frac{\partial}{\partial t_3}D_\infty(t_1,t_2,t_3)$ is well-defined on $\partial T_0$. Also, the value of $D_p(t_1,t_2,t_3)$ in $B_1\sim B_4$ is independent of $p$:
\begin{equation*}
D_p(t_1,t_2,t_3)= \left\{
        \begin{array}{ll}
            t_1t_2 & \quad t_1+t_2 \leq t_3,0 \leq t_1,t_2,t_3 \leq 1 \\
            t_1t_3 & \quad t_1+t_3 \leq t_2,0 \leq t_1,t_2,t_3 \leq 1 \\
            t_2t_3 & \quad t_2+t_3 \leq t_1, 0 \leq t_1,t_2,t_3 \leq 1 \\
            1-t_1-t_2-t_3\\+t_1t_2+t_1t_3+t_2t_3 & \quad t_1+t_2+t_3 \geq 2,0 \leq t_1,t_2,t_3 \leq 1.
        \end{array}
    \right.
\end{equation*}
Thus $D_p=D_\infty$ and $\frac{\partial}{\partial r^\pm}D_p=\frac{\partial}{\partial r^\pm}D_\infty$ outside the closure of $T_0$.
\begin{figure}[ht]
    \centering
    \begin{tikzpicture}[scale=2]
    
    \draw[thick,->] (0,0,0) -- (1.5,0,0) node[anchor=north east]{$t_1$};
    \draw[thick,->] (0,0,0) -- (0,1.5,0) node[anchor=south west]{$t_2$};
    \draw[thick,->] (0,0,0) -- (0,0,1.5) node[anchor=south]{$t_3$};

    \draw[black, thick] (0,0,0) -- (1,0,0) -- (1,1,0) -- (0,1,0) -- cycle;
    \draw[black, thick] (0,0,1) -- (1,0,1) -- (1,1,1) -- (0,1,1) -- cycle;
    \draw[black, thick] (0,0,0) -- (0,0,1);
    \draw[black, thick] (1,0,0) -- (1,0,1);
    \draw[black, thick] (0,1,0) -- (0,1,1);
    \draw[black, thick] (1,1,0) -- (1,1,1);

    \foreach \x in {0,1}
    \foreach \y in {0,1}
    \foreach \z in {0,1} {
        \filldraw (\x,\y,\z) circle (0.5pt);
    }

    \draw[red, thick] (0,0,0) -- (0,1,1) -- (1,0,1) -- cycle;
    \draw[red, thick] (0,0,0) -- (1,0,1) -- (1,1,0) -- cycle;
    \draw[red, thick] (0,0,0) -- (0,1,1) -- (1,1,0) -- cycle;
    \draw[red, thick] (0,1,1) -- (1,0,1) -- (1,1,0) -- cycle;
    \draw[blue, thick] (1,0,1) -- (0,1,1);
    \draw[blue, thick] (0,0,0) -- (1,1,0);
    \draw[black, thick] (0.75,0.5,0) -- (0.75,0.5,1);
    \draw[black, thick] (0.75,0.25,0) -- (0.75,0.25,1);

    \filldraw[blue] (0,0,0) circle (1pt) node[anchor=north east]{};
    \filldraw[blue] (0,1,1) circle (1pt) node[anchor=south west]{};
    \filldraw[blue] (1,0,1) circle (1pt) node[anchor=south east]{};
    \filldraw[blue] (1,1,0) circle (1pt) node[anchor=north west]{};

    \filldraw[black] (1,0,0) circle (1pt) node[anchor=north east]{};
    \filldraw[black] (1,1,1) circle (1pt) node[anchor=north west]{};
    \filldraw[black] (0,0,1) circle (1pt) node[anchor=south east]{};
    \filldraw[black] (0,1,0) circle (1pt) node[anchor=north west]{};

    \filldraw[black] (0.75,0.5,0) circle (0.5pt);
    \filldraw[red] (0.75,0.5,0.25) circle (0.5pt);
    \filldraw[red] (0.75,0.5,0.75) circle (0.5pt);
    \filldraw[black] (0.75,0.5,1) circle (0.5pt);

    \filldraw[blue] (0.75,0.25,1) circle (0.5pt);
    \filldraw[red] (0.75,0.25,0.5) circle (0.5pt);
    \filldraw[black] (0.75,0.25,0) circle (0.5pt);
\end{tikzpicture}
    \caption{Two cases: interior case (upper vertical segment) and boundary case (lower vertical segment)}
    \label{fig: value of partial defivative-B}
\end{figure}

Take a segment joining $(t_1,t_2,0)$ and $(t_1,t_2,1)$, then exactly two points on the segment lie on $\partial T_0$. There are two cases: case 1 is that one such point happens to be the endpoint of the segment, and case 2 is that these two points both lie in the interior. For example, from Figure \ref{fig: value of partial defivative-B} we see that the segment joining $(3/4,1/2,0)$ and $(3/4,1/2,1)$ intersets $\partial T_0$ at $(3/4,1/2,1/4)$ and $(3/4,1/2,3/4)$ which both lie in the interior; the segment joining $(3/4,1/4,0)$ and $(3/4,1/4,1)$ intersets $\partial T_0$ at $(3/4,1/4,1)$ and $(3/4,1/4,1/2)$, and the first point is an endpoint. 

Case 1: Consider the right derivative at $t_1=t_2,r=0$ or left derivative at $t_1+t_2=1,r=1$. This corresponds to the blue segments in the front and back faces of the cube in Figure \ref{fig: value of partial defivative-B}. They are related with reflection, so it suffices to check the first case $t_1=t_2,r=0$. In this case, we see
$$\frac{\partial}{\partial r^+}D_p(t_1,t_2,0)=\frac{\partial}{\partial r^+}D_\infty(t_1,t_2,0)=\min\{t_1,t_2\}$$
for any $t_1\neq t_2$. Since $\frac{\partial}{\partial r^+}D_p(t_1,t_2,0)$ and $\frac{\partial}{\partial r^+}D_\infty(t_1,t_2,0)$ are both continuous with respect to $t_1,t_2$, we get
$$\frac{\partial}{\partial r^+}D_p(t,t,0)=\frac{\partial}{\partial r^+}D_\infty(t,t,0)=t.$$

Case 2: Suppose we are not in case 1. Then either $D_p(t_1,t_2,r)$ is independent of $p$ in a neighboorhood of $(t_1,t_2,r)$, or $0<r<1$ and $D_p(t_1,t_2,r)$ is independent of $p$ in a neighboorhood of $(t_1,t_2,0)$ and $(t_1,t_2,1)$. In the latter case, since $D_p(t_1,t_2,r) \to D_\infty(t_1,t_2,r)$ for $r \in [0,1]$ and $r \to D_\infty(r_1,r_2,r)$ is differentiable in $(0,1)$,  
$$\frac{\partial}{\partial r^+}D_p(t_1,t_2,r)\to\frac{\partial}{\partial r^+}D_\infty(t_1,t_2,r)$$
for $r \in (0,1)$ by Lemma \ref{4.5 lem: converge of concave leads to converge of derivative}. 
\end{proof}
\begin{remark}
In the above proof, we have proved
$$\frac{\partial}{\partial r^+}D_p(t_1,t_2,0)=\min\{t_1,t_2\}$$
for $0 \leq t_1,t_2 \leq 1$. By reflection, we see
\begin{align*}
\frac{\partial}{\partial r^-}D_p(t_1,t_2,1)=-\frac{\partial}{\partial r^+}D_p(t_1,1-t_2,0)+t_1\\
=-\min\{t_1,1-t_2\}+t_1=\max\{0,t_1+t_2-1\}    
\end{align*}
for $0 \leq t_1,t_2\leq 1$.
\end{remark}

\section{Integral formulas for $h$-function}
In this section, we derive the integral formulas for the $h$-function. Since we will deal with integrals in $\mathbb{R}^s$ frequently, we use the symbol $\mathbf{t}=(t_1,\ldots,t_s)$ throughout this section.

\subsection{Settings}\label{subsection 5.1}
We  first introduce several settings with which we work.
\begin{settings}\label{5.1 Tensor product setting}
Suppose we have $s$ groups of data consisting of the following objects: for $1 \leq i \leq s$, let $R_i$ be a Noetherian $k$-algebra, $\mathfrak{m}_i$ be a maximal ideal of $R_i$ such that $k \to R_i/\mathfrak{m}_i$ is an isomorphism, $I_i$ be an $\mathfrak{m}_i$-primary $R_i$ ideal, $f_i$ be an element of $\mathfrak{m}_i$, $q=p^e$ be a power of $p$. Let $\phi \in k[T_1,\ldots,T_s]$ be an element without constant term. Consider the ring $R=\otimes_k R_i$. Let $\mathfrak{m}=\sum_i \mathfrak{m}_iR$, which is a maximal ideal in $R$ with residue field $k$. Let $I=\sum_i I_iR$. We can define $f=\phi(\underline{f})$ as an element in $R$, which falls into $\mathfrak{m}$. Denote $\dim (R_i)_{\mathfrak{m}_i}=d_i,\dim R_{\mathfrak{m}}=d$.
\end{settings}
\begin{proposition}
Using the notation in Settings \ref{5.1 Tensor product setting} we have
$$\dim R_{\mathfrak{m}}=\sum_{1 \leq i \leq s}\dim (R_i)_{\mathfrak{m}_i}.$$
\end{proposition}
\begin{proof}
From Settings \ref{5.1 Tensor product setting}, we have the following isomorphism
$$R/I^{[q]}\cong \otimes_k R_i/I_i^{[q]}.$$
The length of the left side is approximately $cq^d$ for $c \neq 0$ and the length of the right side is approximately $c'q^{\sum_i d_i}$ for $c' \neq 0$, so $d=\sum_i d_i$.
\end{proof}
Now we introduce the settings for reduction modulo $p$ process.
\begin{settings}\label{5.1 Mod p setting}
Let $R$ be a finitely generated $\mathbb{Z}$-algebra, $I$ be an $R$-ideal, $f \in R$. Suppose $R/I$ is a finitely generated $\mathbb{Z}$-algebra. For prime number $p>0$, denote
$$R_p=R\otimes_\mathbb{Z}\mathbb{F}_p,I_p=IR_p,f_p=f\otimes_\mathbb{Z}\mathbb{F}_p \in R_p.$$
Suppose $I_p$ is an $\mathfrak{m}_p$-primary ideal where $\mathfrak{m}_p$ is a maximal $R_p$-ideal. We consider the following sequence of functions
$$p \to h_{R_p,I_p,f_p}(t).$$
For simplicity, denote $h_{R_p,I_p,f_p}(t)=h_{R,I,f,p}(t)$ and omit $R,I,f$ if they are clear from context. Denote
$$h_{R,I,f,\infty}(t)=\lim_{p \to \infty}h_{R,I,f,p}(t)$$
whenever the limit exists at $t$. We call this limit the \textit{limit $h$-function of the triple $(R,I,f)$}. We also say it is an $h$-function in limit characteristic.
\end{settings}
\begin{example}
We restate Definition \ref{4.3 Dinftydef} as follows:
$$D_\infty(t_1,t_2,t_3)=h_{\mathbb{Z}[T_1,T_2],0,(T_1,T_2,T_3),\infty}(t_1,t_2,t_3),$$
which exists on $\mathbb{R}^3$ by Proposition \ref{4.3 Dinftyexists}.
\end{example}
Now we work in a ``reduction modulo $p$" version of Settings \ref{5.1 Tensor product setting}. This is a combination of Settings \ref{5.1 Tensor product setting} and Settings \ref{5.1 Mod p setting}.
\begin{settings}\label{5.1 Tensor product mod p setting}
In this setting, we have data $R_i,\mathfrak{m}_i,I_i,f_i,R,I,\mathfrak{m},\phi,f=\phi(\underline{f})$ where $\phi \in \mathbb{Z}[T_1,\ldots,T_s]$, $R,R_i$ are $\mathbb{Z}$-algebras, $f \in R$, $f_i \in R_i$, $I_i,\mathfrak{m}_i$ are $R_i$-ideals, $I,\mathfrak{m}$ are $R$-ideals, and their reduction modulo $p$, $R_p$ and $R_{i,p}$, along with $f_{i,p} \in R_{i,p},I_{i,p},\mathfrak{m}_{i,p} \subset R_{i,p},f_p \in R_p,I_p,\mathfrak{m}_p \subset R_p$, $k_p=R_p/\mathfrak{m}_p=R_{i,p}/\mathfrak{m}_{i,p}$, $\phi_p \in \mathbb{Z}[T_1,\ldots,T_s]\otimes_\mathbb{Z}k_p,f_p=\phi_p(\underline{f}_p)$ satisfy Settings \ref{5.1 Tensor product setting} for large $p$. Suppose moreover there is constant $C_i$ such that $h_{e,R_{i,p},I_{i,p},f_{i,p}}$ are constant on $[C_i,\infty)$ for all $p$ and $e$. In particular, if we set $\lim_{t \to \infty}h_{e,R_{i,p},I_{i,p},f_{i,p}}(t)=e_{i,p}$, then $\lim_{p \to \infty}e_{i,p}=e_i$ exists. Denote $d_i=\dim (R_{i,p})_{\mathfrak{m}_{i,p}}$, $d=\dim (R_p)_{\mathfrak{m}_p}=\sum_i d_i$ for large $p$.    
\end{settings}

\subsection{$k$-objects and the representation ring $\Gamma$}
In this subsection, we fix a field $k$ of characteristic $p>0$. We introduce the concept of representation ring $\Gamma$ where we make computations. The concept of the representation ring first appears in \cite{HM93}, although the computational results are rooted from the results in \cite{Han92}. Apart from its additive structure and multiplication, we will also discuss multilinear maps on $\Gamma$.

\begin{definition}[\cite{HM93}]
We say a $k$-object $M$ with respect to $T$ is a finitely generated $k[T]$-module annihilated by a power of $T$. The direct sum of two $k$-objects is the direct sum as $k[T]$-module. Let $\Gamma$ be the quotient of the free abelian group over symbols $[M]$ where $M$ runs through isomorphic classes of $k$-objects by the relations $[{M\oplus N}]-[M]-[N]$.    
\end{definition}
We see $k[T]$ is a PID. From the structure theorem of PID, we deduce:
\begin{proposition}
\begin{enumerate}
\item Let $\delta_i=[k[T]/(T^i)]\in\Gamma$ for $i \geq 1$. Then $\Gamma$ is a free abelian group over $\delta_i$ with addition $\oplus$.
\item Every element in $\Gamma$ is a formal difference of two isomorphic classes of $k$-objects.
\end{enumerate}    
\end{proposition} 
For a $k$-object $M$, we still denote its class in $\Gamma$ by $[M]$, so in $\Gamma$ we have $[M\oplus N]=[M]+[N]$. We can write $[M]=\sum_{i \geq 1}e_{M,T}(i)\delta_i$, where $e_{M,T}(i)$ is the multiplicity of $k[T]/(T^i)$ in $M$. Denote $l_{M,T}(i): \mathbb{Z} \to \mathbb{Z}$ to be the following function
\begin{equation*}
l_{M,T}(i) = \left\{
        \begin{array}{ll}
            0 & \quad i \leq 0 \\
            l(M/T^iM) & \quad i \geq 1.
        \end{array}
    \right.
\end{equation*}
We omit $T$ from $e_{M,T}(i)$ or $l_{M,T}(i)$ if it is clear from context.

\begin{proposition}\label{5.2 e_M and l_M relation}
Let $M$ be a $k$-object. Then: 
\begin{enumerate}
\item \begin{equation*}
l_M(n) = \left\{
        \begin{array}{ll}
            0 & \quad n \leq 0 \\
            e_{M}(1)+2e_{M}(2)+\ldots\\+ne_{M}(n)+ne_{M}(n+1)+\ldots & \quad n \geq 1.
        \end{array}
    \right.
    \end{equation*}
\item \begin{equation*}
l_{M}(n)-l_M(n-1) = \left\{
        \begin{array}{ll}
            0 & \quad n \leq 0\\
            e_{M}(n)+e_{M}(n+1)+\ldots & \quad n \geq 1.
        \end{array}
    \right.
\end{equation*}
\item for any $n \geq 1$,
\begin{align*}
e_{M}(n)=l_M(n)-l_M(n-1)-(l_{M}(n+1)-l_M(n))\\
=2l_M(n)-l_M(n+1)-l_M(n-1).
\end{align*}
\end{enumerate}
\end{proposition}
\begin{proof}
(1) follows from the definition of $l_M(n)$ and $e_{M}(n)$; (2) and (3) follow from (1).
\end{proof}
Next, we consider multilinear maps on $\Gamma$. Since $\Gamma$ is the free abelian group with basis $\delta_i, i \geq 1$, to define a multilinear map $\Gamma^s \to \Gamma$, we only need to define it on tuples of basis element $\delta_i$'s. For any $\phi \in k[T_1,T_2,\ldots,T_s]$ without a constant term, we define the multilinear map $B_\phi: \Gamma^s \to \Gamma$ by specifying $B_\phi(\delta_{t_1},\ldots,\delta_{t_s})$ as follows: $$B_\phi(\delta_{t_1},\ldots,\delta_{t_s})=k[T_1,T_2,\ldots,T_s]/(T_1^{t_1},T_2^{t_2},\ldots,T_s^{t_s})$$ 
as a $k$-object with respect to $T=\phi \in k[T_1,\ldots,T_s]$. Here $T$ acts nilpotently since $T$ has no constant term. This gives the multilinear form $B_\phi$. The importance of this definition is reflected in the following proposition on $s$-fold tensor product:
\begin{proposition}\label{5.2 multilinear form property}
Let $M_i$ be a $k$-object with respect to $f_i$ for $1 \leq i \leq s$. Take $\phi \in k[T_1,T_2,\ldots,T_s]$ and define $B_\phi$ as above. Let $\otimes_k M_i=(\otimes_k)_{1 \leq i \leq s} M_i$ be the $s$-fold tensor product of all $M_i$'s. Then $\phi(\underline{f})$ acts on $ \otimes_k M_i$ which makes $\otimes_k M_i$ a $k$-object with respect to $\phi(\underline{f})$, and as a $k$-object, $[{\otimes_k M_i}]=B_\phi([{M_1}],\ldots,[{M_s}])$.  
\end{proposition}
\begin{proof}
Since $M_i$ is a $k[f_i]$-module, $\otimes_k M_i$ is a $k[\underline{f}]$-module. In other words, we let $f_i$ act on $\otimes_k M_i$ via $id_{M_1}\otimes\ldots\otimes f_i \otimes\ldots \otimes id_{M_s}$. So $\phi(\underline{f}) \in k[\underline{f}]$ also acts on $\otimes_k M_i$ and this action is $k$-linear. Since the actions of $f_i$'s are nilpotent and $\phi$ has no constant term, the action of $\phi(\underline{f})$ is also nilpotent. The last sentence can be proved by decomposing $M_i$ into cyclic submodules over $k[f_i]$. 
\end{proof}
Denote the coefficient of the bilinear form 
$B_\phi$ by $$B_\phi(\mathbf{t},r)=e_{k[T_1,T_2,\ldots,T_s]/(T_1^{t_1},T_2^{t_2},\ldots,T_s^{t_s}),\phi}(r).$$ 
That is, we also view $B_\phi$ as a map $\mathbb{Z}^{s+1} \to \mathbb{Z}$ by abusing the notation. We have:
\begin{proposition}
For any $\mathbf{t} \in \mathbb{Z}^s,r \in \mathbb{Z}$, we have:
\begin{enumerate}
\item $B_\phi(\delta_{t_1},\ldots,\delta_{t_s})=\sum_{r \geq 1} B_\phi(\mathbf{t},r)\delta_r$.
\item $l_{k[T_1,T_2,\ldots,T_s]/(T_1^{t_1},T_2^{t_2},\ldots,T_s^{t_s}),\phi}(r)=l(k[T_1,T_2,\ldots,T_s]/(T_1^{t_1},T_2^{t_2},\ldots,T_s^{t_s},\phi^r))=D_\phi(\mathbf{t},r).$
\item if 
$r \geq 1$, $B_\phi(\mathbf{t},r)=2D_\phi(\mathbf{t},r)-D_\phi(\mathbf{t},r+1)-D_\phi(\mathbf{t},r-1)$.
\end{enumerate}
\end{proposition}
\begin{corollary}\label{5.2 Multilinear form coefficient}
Let $M_i$ be a $k$-object with respect to $f_i$ for $1 \leq i \leq s$. Take $\phi \in k[T_1,\ldots,T_s]$. Then
\begin{enumerate}
\item $e_{B_\phi([{M_1}],\ldots,[{M_s}]),\phi(\underline{f})}(r)=\sum_{\mathbf{t} \geq 1}\prod_{1 \leq i \leq s}e_{M_i,f_i}(t_i) B_\phi(\mathbf{t},r)$.
\item $l_{B_\phi([{M_1}],\ldots,[{M_s}]),\phi(\underline{f})}(r)=\sum_{\mathbf{t} \geq 1}\prod_{1 \leq i \leq s}e_{M_i,f_i}(t_i) D_\phi(\mathbf{t},r)$.
\end{enumerate}
\end{corollary}

We give an additional remark on the multiplicative structure of the representation ring. In the setting of \cite{HM93}, we work under the assumption $s=2,\phi=T_1+T_2\in k[T_1,T_2]$. The specialty of this setting leads to the following fact:
\begin{proposition}
Let $s=2$, $\phi=T_1+T_2$, then $B_\phi:\Gamma\times\Gamma \to \Gamma$ is a bilinear map which satisfies associativity and commutativity. Thus, $(\Gamma,\oplus,B_\phi)$ has a commutative ring structure, called the \textbf{representation ring}. It is a unital ring with unit $\delta_1$.    
\end{proposition}

\subsection{A discrete multilinear formula for $h$-function}
We can apply the results on $\Gamma$ in last subsection to $h$-functions. 

Let $R$ be a Noetherian $k$-algebra, $\mathfrak{m}$ be a maximal ideal of $R$ such that $k=R/\mathfrak{m}$, $d=\dim R_\mathfrak{m}$, $I$ be an $\mathfrak{m}$-primary ideal of $R$, $f \in \mathfrak{m}$ be an element in $R$, $q=p^n$ be a power of $n$. Then $R/I^{[q]}$ is a module of finite length annihilated by some power of $f$, so it is a $k$-object with respect to $f$. In this sense we have:
\begin{proposition}\label{5.2 h-function and eM,lM}
\begin{enumerate}
\item If $r \in \mathbb{Z}$, $l_{R/I^{[q]},f}(r)=l(R/I^{[q]},f^r)=H_{e,R,I,f}(r/q)$.
\item There is constant $C$ such that for any $e$, whenever $r\geq Cq, l_{R/I^{[q]},f}(r)$ is a constant.
\item If $r$ is a positive integer, $e_{R/I^{[q]},f}(r)=2H_{e,R,I,f}(r/p^e)-H_{e,R,I,f}((r+1)/p^e)-H_{e,R,I,f}((r-1)/p^e))$.
\item There is constant $C$ such that for any $e$, whenever $r\geq Cq, e_{R/I^{[q]},f}(r)=0$.
\end{enumerate}
\end{proposition}
\begin{proof}
(1) is just definition. For (2), choose $C$ such that $f^C \in I$, then $f^{Cq} \subset I^{[q]}$, so $l(R/I^{[q]},f^r)=l(R/I^{[q]})$ is independent of $r$ for $r \geq Cq$. (3) comes from (1) and Proposition \ref{5.2 e_M and l_M relation} (3). (4) comes from (2) and Proposition \ref{5.2 e_M and l_M relation} (3).   
\end{proof}
Now we derive a result under Settings \ref{5.1 Tensor product setting}.
\begin{proposition}[Discrete multilinear formula]\label{5.2 Discrete multilinear formula}
Under Settings \ref{5.1 Tensor product setting} we have for $r \in \mathbb{N}$:
\begin{enumerate}
\item 
\begin{align*}
e_{R/I^{[q]},f}(r)\\
=\sum_{t_i\geq 1}B_\phi(\mathbf{t},r)\prod_{1 \leq i \leq s}(2H_{e,R_i,I_i,f_i}(\frac{t_i}{q})-H_{e,R_i,I_i,f_i}(\frac{t_i+1}{q})-H_{e,R_i,I_i,f_i}(\frac{t_i-1}{q}))\\
=\sum_{t_i\in \mathbb{Z}}B_\phi(\mathbf{t},r)\prod_{1 \leq i \leq s}(2H_{e,R_i,I_i,f_i}(\frac{t_i}{q})-H_{e,R_i,I_i,f_i}(\frac{t_i+1}{q})-H_{e,R_i,I_i,f_i}(\frac{t_i-1}{q}))\\
=\frac{1}{q^s}\sum_{t_i\in \mathbb{Z}}B_\phi(\mathbf{t},r)\prod_{1 \leq i \leq s}(DH_{e,R_i,I_i,f_i}(\frac{t_i-1}{q})-DH_{e,R_i,I_i,f_i}(\frac{t_i}{q})).
\end{align*}
\item 
\begin{align*}
H_{e,R,I,f}(\frac{r}{q})=l_{R/I^{[q]},f}(r)\\
=\sum_{t_i\geq 1}D_\phi(\mathbf{t},r)\prod_{1 \leq i \leq s}(2H_{e,R_i,I_i,f_i}(\frac{t_i}{q})-H_{e,R_i,I_i,f_i}(\frac{t_i+1}{q})-H_{e,R_i,I_i,f_i}(\frac{t_i-1}{q}))\\
=\sum_{t_i\in \mathbb{Z}}D_\phi(\mathbf{t},r)\prod_{1 \leq i \leq s}(2H_{e,R_i,I_i,f_i}(\frac{t_i}{q})-H_{e,R_i,I_i,f_i}(\frac{t_i+1}{q})-H_{e,R_i,I_i,f_i}(\frac{t_i-1}{q}))\\
=\frac{1}{q^s}\sum_{t_i\in \mathbb{Z}}D_\phi(\mathbf{t},r)\prod_{1 \leq i \leq s}(DH_{e,R_i,I_i,f_i}(\frac{t_i-1}{q})-DH_{e,R_i,I_i,f_i}(\frac{t_i}{q})).
\end{align*}
\end{enumerate}
\end{proposition}
\begin{proof}
We have
$$R/I^{[q]}\cong \tilde{R}/\sum_i I_i^{[q]}\tilde{R}\cong \otimes_k (R_i/I_i^{[q]}).$$
We prove (1) and (2) simutaneously. The first equality on either side is a consequence of Proposition \ref{5.2 multilinear form property} and Corollary \ref{5.2 Multilinear form coefficient}; the second equality holds as $B_\phi(\mathbf{t},r)=0$ whenever $t_i \leq 0$ for some $i$; the third equality is a reformulation by definition of $DH_e$.    
\end{proof}

\subsection{The integral formula in fixed characteristic}
In this subsection, we fix the residue field $k$, therefore fixing its characteristic $p$. 

We first pick a triple of data $(R,I,f)$ where $R$ is a Noetherian ring, $I$ is an $R$-ideal, $f\in R$ such that $\sqrt{(I,\underline{f})}=\mathfrak{m}$ is a maximal $R$-ideal. Let $H_e=H_{e,R,I,f},h_e=h_{e,R,I,f},h=h_{R,I,f}$. Denote $Dh_e(x)=q(h_e(x+1/q)-h_e(x))$ and $DH_e(x)=q(H_e(x+1/q)-H_e(x))$.

\begin{definition}
Let $e\in\mathbb{N}$. We define the $1/q$-linearization of $h_e$ to be the piecewise linear function $\hat{h}_e$ which coincides with $h_e$ on $1/q\mathbb{Z}$ and is linear on $[t/q,(t+1)/q]$ for all $t \in \mathbb{Z}$.     
\end{definition}
\begin{lemma}\label{5.3 properties of hhat}
We take $h_e=h_{e,R',I',f'}$ for some $R',I',f'$. Then
\begin{enumerate}
\item $\hat{h}_e'(x)=
Dh_e(\lfloor xq\rfloor/q)$ for $x \notin 1/q\mathbb{Z}$.
\item for $x \in 1/q\mathbb{Z}$, $\hat{h}_{e,+}'(x)=
Dh_e(x)$ and $\hat{h}_{e,-}'(x)=
Dh_e(x-1/q)$.
\item $\lim_{q \to \infty}\hat{h}'_e(x)=h'(x)$ for all but countably many $x \in \mathbb{R}$.
\end{enumerate}    
\end{lemma}
\begin{proof}
(1) and (2) are trivial by definition, and (3) is a consequence of Corollary \ref{3 Dihe approaches derivative}.    
\end{proof}
We will use the same notation as Settings \ref{5.1 Tensor product setting}, including $R_i,\mathfrak{m}_i,k,I_i,f_i,\phi,R,\mathfrak{m}$, and let $f=\phi(\underline{f})$, $d_i=\dim (R_i)_{\mathfrak{m}_i}$, $d=\dim R_\mathfrak{m}=\sum_i d_i$. By boundness result, we may choose $C$ sufficiently large such that for all $i$, $H_{e,R_i,I_i,f_i}(t)$ is constant for $t \geq Cq$, so $DH_{e,R_i,I_i,f_i}(t)=Dh_{e,R_i,I_i,f_i}(t)=0$ for $t \geq Cq$.

\begin{proposition}\label{5.3 integral formula lemma}Let $P$ be the partition of the interval $\mathcal{I}=[\mathbf{-C},\mathbf{C}]^s=\cup_{\mathbf{-Cq}\leq\mathbf{i}\leq\mathbf{Cq-1}}\mathcal{I}_{e,\mathbf{i}}$ which divides $\mathcal{I}$ into $(2Cq)^s$ cubes of length $1/q$ for a sufficient large integer $C$. Then $h_{e,R,I,f}(r)$ is a Riemann sum $RS(P,\xi,D_\phi(\cdot,\frac{\lceil rq\rceil}{q}),\{-\hat{h}'_{e,R_i,I_i,f_i,+}\}_{1 \leq i \leq s})$ of the following Riemann Stieltjes integral
\begin{align*}
\int_{[\mathbf{-C},\mathbf{C}]}D_\phi(\mathbf{t},\frac{\lceil rq \rceil}{q}) \prod_{1 \leq i \leq s}d(-\hat{h}'_{e,R_i,I_i,f_i,+}(t_i)),
\end{align*}
where $\mathbf{t}=(t_1,\ldots,t_s)$. Moreover, there is constant $C'$ independent of $e$ such that
$$|h_{e,R,I,f}(r)-\int_{[\mathbf{-C},\mathbf{C}]}D_\phi(\mathbf{t},r) \prod_{1 \leq i \leq s}d(-\hat{h}'_{e,R_i,I_i,f_i,+}(t_i))|\leq C'/q.$$
\end{proposition}
\begin{proof}
For simplicity we make the following convention and omit some lower indices:
$$h_e=h_{e,R,I,f},H_e=H_{e,R,I,f},h_{e,i}=h_{e,R_i,I_i,f_i},H_{e,i}=H_{e,R_i,I_i,f_i}.$$
By definition $h_{e}(r)=h_{e}(\frac{\lceil rq \rceil}{q})$. So to prove $h_{e}(r)=RS(P,\xi,D_\phi(\cdot,\frac{\lceil rq\rceil}{q}),\{-\hat{h}'_{e,i,+}\}_{1 \leq i \leq s})$, we may assume $r \in \frac{1}{q}\mathbb{Z}$. By Proposition \ref{5.2 Discrete multilinear formula} we have for $r \in \mathbb{N}$,
\begin{align*}
H_{e}(\tfrac{r}{q})=\tfrac{1}{q^s}\sum_{t_i\in \mathbb{Z}}D_\phi(\mathbf{t},r)\prod_{1 \leq i \leq s}(DH_{e,i}(\tfrac{t_i-1}{q})-DH_{e,i}(\tfrac{t_i}{q})).
\end{align*}
We replace $r$ by $rq$ and $t_i$ by $t_iq$. By rescaling we get
\begin{align*}
H_{e}(r)=\tfrac{1}{q^s}\sum_{t_i\in \mathbb{Z}}D_\phi(\mathbf{t},qr)\prod_{1 \leq i \leq s}(DH_{e,i}(\tfrac{t_i-1}{q})-DH_{e,i}(\tfrac{t_i}{q}))\\
=\tfrac{1}{q^s}\sum_{t_i\in \tfrac{1}{q}\mathbb{Z}}D_\phi(q\mathbf{t},qr)\prod_{1 \leq i \leq s}(DH_{e,i}(\tfrac{qt_i-1}{q})-DH_{e,i}(\tfrac{qt_i}{q}))\\
=\sum_{t_i\in \tfrac{1}{q}\mathbb{Z}}D_\phi(\mathbf{t},r)\prod_{1 \leq i \leq s}(DH_{e,i}(t_i-\tfrac{1}{q})-DH_{e,i}(t_i))\\
=\sum_{t_i\in \tfrac{1}{q}\mathbb{Z}}D_\phi(\mathbf{t},r)\prod_{1 \leq i \leq s}q^{d_i}(Dh_{e,i}(t_i-\tfrac{1}{q})-Dh_{e,i}(t_i))\\
=q^d\sum_{t_i\in \tfrac{1}{q}\mathbb{Z}}D_\phi(\mathbf{t},r)\prod_{1 \leq i \leq s}(Dh_{e,i}(t_i-\tfrac{1}{q})-Dh_{e,i}(t_i))\\
=q^d\sum_{t_i\in \tfrac{1}{q}\mathbb{Z}}D_\phi(\mathbf{t},r)\prod_{1 \leq i \leq s}(\hat{h}'_{e,i,+}(t_i-\tfrac{1}{q})-\hat{h}'_{e,i,+}(t_i))\\
=q^d\sum_{t_i\in \tfrac{1}{q}\mathbb{Z}}D_\phi(\mathbf{t}+\mathbf{\tfrac{1}{q}},r)\prod_{1 \leq i \leq s}(\hat{h}'_{e,i,+}(t_i)-\hat{h}'_{e,i,+}(t_i+\tfrac{1}{q}))\\
=q^d\sum_{t_i\in \tfrac{1}{q}\mathbb{Z}}D_\phi(\mathbf{t}+\mathbf{\tfrac{1}{q}},r)\prod_{1 \leq i \leq s}(-1)(\hat{h}'_{e,i,+}(t_i+\tfrac{1}{q})-\hat{h}'_{e,i,+}(t_i))\\
=\sum_{t_i\in \tfrac{1}{q}\mathbb{Z},-C \leq t_i \leq C}q^dD_\phi(\mathbf{t}+\mathbf{\tfrac{1}{q}},r)\prod_{1 \leq i \leq s}(-1)(\hat{h}'_{e,i,+}(t_i+\tfrac{1}{q})-\hat{h}'_{e,i,+}(t_i)).
\end{align*}
The last equation holds since $DH_{e}(x)=0$ for $x<0$ and $x\geq Cq$. Dividing by $q^d$, we get
\begin{align*}
h_{e}(r)=\sum_{t_i\in \tfrac{1}{q}\mathbb{Z},-C \leq t_i \leq C}D_\phi(\mathbf{t}+\mathbf{\tfrac{1}{q}},r)\prod_{1 \leq i \leq s}(-1)(\tfrac{\hat{h}'_{e,i,+}(t_i+\tfrac{1}{q})}{q^{d_i}}-\tfrac{\hat{h}'_{e,i,+}(t_i)}{q^{d_i}}).
\end{align*}
The above is equal to the Riemann sum $RS(P,\xi,D_\phi,\{-\hat{h}'_{e,i,+}\}_{1 \leq i \leq s})$ for the choice $\xi_{\mathbf{i}}=(\mathbf{i}+\mathbf{1})/q \in \mathcal{I}_{e,\mathbf{i}}=[\mathbf{i}/q,(\mathbf{i+1})/q]$. Now we claim 
\begin{align*}
|RS(P,\xi,D_\phi,\{-\hat{h}'_{e,i,+}\}_{1 \leq i \leq s})-\int_{[\mathbf{-C},\mathbf{C}]}D_\phi(\mathbf{t},r) \prod_{1 \leq i \leq s}d(-\hat{h}'_{e,i,+}(t_i))|\leq C/q.
\end{align*}
By Lemma \ref{2.1 Intermediatevalue}, there is another choice of $\xi'$ such that
$$\int_{[\mathbf{-C},\mathbf{C}]}D_\phi(\mathbf{t},r) \prod_{1 \leq i \leq s}d(-\hat{h}'_{e,i,+}(t_i))=RS(P,\xi',D_\phi,\{-\hat{h}'_{e,i,+}\}_{1 \leq i \leq s}).$$
Note that $D_\phi$ is Lipschitz continuous. We assume 
$|D_\phi(\mathbf{t},r)-D_\phi(\mathbf{t}',r')|\leq C_1||(\mathbf{t}-\mathbf{t}',r-r')||_1$. Now for any choice of $\xi,\xi'$ and multiindex $\mathbf{i}$, $||\xi_{\mathbf{i}}-\xi'_{\mathbf{i}}||_1\leq sd(P)=s/q$ and $|\tfrac{\lceil rq \rceil}{q}-r|\leq 1/q$, so $|D_\phi(\xi_{\mathbf{i}},\tfrac{\lceil rq \rceil}{q})-D_\phi(\xi'_{\mathbf{i}},r)|\leq C_1(s+1)/q$. Also $D_\phi(\xi_{\mathbf{i}})\neq D_\phi(\xi'_{\mathbf{i}})$ only when $\mathbf{i}\geq \mathbf{0}$, otherwise both values at $D_\phi$ are $0$. So
\begin{align*}
|RS(P,\xi,D_\phi(\cdot,\tfrac{\lceil rq\rceil}{q}),\{-\hat{h}'_{e,i,+}\}_{1 \leq i \leq s})-RS(P,\xi',D_\phi(\cdot,r),\{-\hat{h}'_{e,i,+}\}_{1 \leq i \leq s})|\\
\leq \sum_{t_i\in 1/q\mathbb{Z}\cap [\mathbf{0},\mathbf{C}]}(C_1(s+1)/q)\prod_{1 \leq i \leq s}(-1)(\tfrac{\hat{h}'_{e,i,+}(t_i+1/q)}{q^{d_i}}-\tfrac{\hat{h}'_{e,i,+}(t_i)}{q^{d_i}})\\
=C_1(s+1)/q\prod \hat{h}'_{e,i,+}(0).
\end{align*}
By definition of Hilbert-Kunz multiplicity, we have
$$\lim_{e \to \infty}\hat{h}'_{e,i,+}(0)=\lim_{e \to \infty}Dh_{e,i,+}(0)=e_{HK}(I_i,R/f_i)<\infty$$
Thus we have
\begin{align*}
\overline{\lim}_{e\to \infty}|RS(P,\xi,D_\phi,\{-\hat{h}'_{e,i,+}\}_{1 \leq i \leq s})-RS(P,\xi',D_\phi,\{-\hat{h}'_{e,i,+}\}_{1 \leq i \leq s})|\\
\leq 1/q*C_1(s+1)*\prod_{1 \leq i \leq s}e_{HK}(R_i/f_i,I_i).
\end{align*}
Thus an appropriate constant $C>C_1(s+1)\prod_{1 \leq i \leq s}e_{HK}(R_i/f_i,I_i)$ will satisfy the desired inequality.
\end{proof}
\begin{lemma}\label{5.3 integral formula lemma2}
We have
\begin{align*}
\lim_{q \to \infty}\int_{[\mathbf{-C},\mathbf{C}]}D_\phi(\mathbf{t},r) \prod_{1 \leq i \leq s}(-d\hat{h}'_{e,R_i,I_i,f_i,+}(t_i))
=\int_{[\mathbf{-C},\mathbf{C}]}D_\phi(\mathbf{t},r) \prod_{1 \leq i \leq s}d(-h'_{R_i,I_i,f_i}(t_i)).
\end{align*}
\end{lemma}
\begin{proof}
By Corollary \ref{3 Dihe approaches derivative} and Lemma \ref{5.3 properties of hhat}, $\lim_{e\to \infty}\hat{h}'_e=\lim_{e \to \infty}Dh_e=h$ for all but countably many points. Now the result follows from Theorem \ref{2.4 convergence of RS integral on f and alpha side}.  
\end{proof}
Proposition \ref{5.3 integral formula lemma} and Lemma \ref{5.3 integral formula lemma2} yield the following result:
\begin{theorem}\label{5.3 integral formula C}
Under Settings \ref{5.1 Tensor product setting}, we have
$$h_{R,I,f}(r)=\int_{[\mathbf{-C},\mathbf{C}]}D_\phi(\mathbf{t},r) \prod_{1 \leq i \leq s}d(-h'_{R_i,I_i,f_i}(t_i))$$
for $C$ large enough depending on $R_i,I_i,f_i$.
\end{theorem}
Since this is true for sufficiently large $C$, we can also view it as an integral over $\mathbb{R}^s$. However, we would like to specify an interval to integrate instead of an implicit large $C$. This is done in the following theorem.
\begin{theorem}[Integral formula for $h$-function]\label{5.3 integral formula threshold}
Let $C_i=c^{I_i}(f_i)$ be the $F$-threshold of $f_i$ with respect to $I_i$. Then for $C$ large enough, 
\begin{align*}
h_{R,I,f}(r)=\int_{[\mathbf{-C},\mathbf{C}]}D_\phi(\mathbf{t},r) \prod_{1 \leq i \leq s}d(-h'_{R_i,I_i,f_i}(t_i))\\
=\int_{\prod_{1 \leq i \leq s}[0^-,C^+_i]}D_\phi(\mathbf{t},r) \prod_{1 \leq i \leq s}d(-h'_{R_i,I_i,f_i}(t_i))\\
=\int_{\prod_{1 \leq i \leq s}[0^+,C^+_i]}D_\phi(\mathbf{t},r) \prod_{1 \leq i \leq s}d(-h'_{R_i,I_i,f_i}(t_i)).
\end{align*}

\end{theorem}
\begin{proof}
For any $C_i'>C_i$ and $C_i''<0$,  $D_\phi(\cdot,r)=0$ on $[\mathbf{-C},\mathbf{C}]\backslash \prod_{1 \leq i \leq s}[C''_i,C]$ and $h'_{R_i,I_i,f_i}(t_i)=0$ for $t_i \geq C'_i>C_i$ by Proposition \ref{3 F-threshold property}, so the integral only depends on the region $\prod_{1 \leq i \leq s}[-C''_i,C'_i]$. So we have
\begin{align*}
h_{R,I,f}(r)=\int_{[\mathbf{-C},\mathbf{C}]}D_\phi(\mathbf{t},r) \prod_{1 \leq i \leq s}d(-h'_{R_i,I_i,f_i}(t_i))\\
=\int_{\prod_{1 \leq i \leq s}[-C''_i,C'_i]}D_\phi(\mathbf{t},r) \prod_{1 \leq i \leq s}d(-h'_{R_i,I_i,f_i}(t_i)).
\end{align*}  
Letting $C''_i \to 0^-$ and $C'_i \to C^+_i$ yields the second equality in the statement. Now $D_\phi=0$ on coordinate planes $t_i=0$, so the value of $h'$ at $t_i=0$ does not affect the value of the integral. So the third equality holds.
\end{proof}
The most commonly used case of the integral formula is $s=2$. In this case, the formula specializes to the following integral formula
\begin{align*}
h_{R_1\otimes R_2,I_1+I_2,f=\phi(f_1,f_2)}(r)=\int_{[0,\infty)^2}D_\phi(t_1,t_2,r)d(-h'_{R_1,I_1,f_1}(t_1))d(-h'_{R_2,I_2,f_2}(t_2))\\
=\int_{[0,C_1^+]\times[0,C_2^+]}D_\phi(t_1,t_2,r)d(-h'_{R_1,I_1,f_1}(t_1))d(-h'_{R_2,I_2,f_2}(t_2))
\end{align*}
where $C_i$ is the threshold of $f_i$ with respect to $I_i$ for $i=1,2$. We omit the sign of $0$ here.

Now we prove a multivariate version of the integral formula for multivariate $h$-function.

We will use the same notations as Settings \ref{5.1 Tensor product setting}, including $R_i,\mathfrak{m}_i,k,I_i,\phi,R,\mathfrak{m}$, , $d_i=\dim (R_i)_{\mathfrak{m}_i}$, $d=\dim R_\mathfrak{m}=\sum_i d_i$, and the assumptions on these notations are the same. Instead of picking a single $f_i$ from each $R_i$, we will choose $s$-many sequences $g_{ij} \in R_i,1 \leq i \leq s,1 \leq j \leq r_i$ such that $I_i+(\underline{g}_i)$ is $\mathfrak{m}_i$-primary and a sequence of $s$ elements $f_i \in R_i,1\leq i \leq s$. Let $f=\phi(\underline{f})$. There are $s$ many multivariate $h$-functions:
$$h_i=h_{R_i,I_i,(\underline{g}_i,f)}(\mathbf{t}_i,r):\mathbb{R}^{r_i+1} \to \mathbb{R},$$
here $\mathbf{t}_i \in \mathbb{R}^{r_i}$. By boundness result, we may choose $C$ sufficiently large such that for all $i$, $h_i(\mathbf{t},r)$ is constant for $\mathbf{t} \geq C\mathbf{q}$ or $r \geq Cq$, so $D_rh_i(\mathbf{t},r)=0$ for $\mathbf{t} \geq C\mathbf{q}$ or $r \geq Cq$.

The corresponding result for Theorem \ref{5.3 integral formula C} is
\begin{theorem}\label{5.3 integral formula multivariate}
$$h_{R,I,\underline{g_1},\ldots,\underline{g_s},\phi(\underline{f})}(\mathbf{t}_1,\ldots,\mathbf{t}_s,r)=\int_{[\mathbf{-C},\mathbf{C}]}D_\phi(r_1,\ldots,r_s,r) \prod_{1 \leq i \leq s}(-d_{r_i}\frac{\partial}{\partial r_i}h_i(\mathbf{t}_i,r_i)).$$    
\end{theorem}
\begin{proof}
The left side is continuous with respect to $\mathbf{t}_i$ by Proposition \ref{3 h-function basic property} and the right side is continuous with respect to $\mathbf{t}_i$ by Corollary \ref{2.4 convergence of RS integral continuous version}. So we may assume $\mathbf{t}_i \in \mathbb{Z}[1/p]^{r_i}$. For every sufficiently large $q$, we view $$R_i/(I_i^{[q]},\underline{g}_i^{q\mathbf{t}_i})$$
as a $k$-object with respect to $f_i$, and view their tensor product
$$R/(I_i^{[q]},\underline{g}_i^{q\mathbf{t}_i},1 \leq i \leq s)$$
as a $k$-object with respect to $f=\phi(\underline{f})$. We can apply Corollary \ref{5.2 Multilinear form coefficient}, proceed as in Proposition \ref{5.2 Discrete multilinear formula}, Proposition \ref{5.3 integral formula lemma}, and Lemma \ref{5.3 integral formula lemma2} to get the result. Note that Corollary \ref{3 Dihe approaches derivative} still holds in multivariate case.
\end{proof}

\subsection{The integral formula for limit characteristic and convergence of $h$-function}
In this subsection, we will let $p \to \infty$ and let all the assumptions vary with $p$.

We first prove a uniform bound result independent of $p,e$.
\begin{lemma}
Under Settings \ref{5.1 Mod p setting}, the sequence of functions
$$p \to h'_{e,R_{p},I_{p},f_{p},\pm}(t)$$
is uniformly bounded.
\end{lemma}
\begin{proof}
For each fixed $p$ we have
$$h'_{e,R_p,I_p,f_p,\pm}(t)\leq h'_{e,R_p,I_p,f_p,+}(0)=e_{HK}(R_p/f_p,I_p).$$
But
$$e_{HK}(R_p/f_p,I_p)\leq e(R_p/f_p,I_p),$$
and the right side of this inequality can be bounded uniformly with respect to $p$. Actually, since $R$ is a finitely generated $\mathbb{Z}$-algebra, by generic flatness there is $c \in \mathbb{N}, c \neq 1$ such that $R/(f+I^n)$ is flat over $\mathbb{Z}$ on $\operatorname{Spec}(\mathbb{Z})\backslash V(c)$, so $e(R_p/f_p,I_p)$ is independent of $p$ on this set, and there are only finitely many $p \in V(c)$.
\end{proof}

\begin{theorem}\label{5.4 integral formula limit char}
Under Settings \ref{5.1 Tensor product mod p setting}, suppose
$$h_{R_i,I_i,f_i,\infty}(t_i)=\lim_{p \to \infty}h_{R_i,I_i,f_i,p}(t_i)$$
exists for all $i$, $h'_{R_i,I_i,f_i,p,\pm}(t_i)$ is uniformly bounded, and
$$D_\infty(\mathbf{t},r)=\lim_{p \to \infty}D_p(\mathbf{t},r)$$
exists. Then:
\begin{enumerate}
\item $$\lim_{p \to \infty}h'_{R_i,I_i,f_i,p,+}(t_i)=\lim_{p \to \infty}h'_{R_i,I_i,f_i,p,-}(t_i)=h'_{R_i,I_i,f_i,\infty}(t_i)$$
for all but countably many $t_i$.
\item 
\begin{align*}
h_{R,I,f,\infty}(r)
=\int_{[\mathbf{0},\mathbf{\infty})^s}D_\infty(\mathbf{t},r) \prod_{1 \leq i \leq s}d(-h'_{R_i,I_i,f_i,\infty}(t_i))\\
=\int_{\prod_{1\leq i \leq s}[0,C_i^+]}D_\infty(\mathbf{t},r) \prod_{1 \leq i \leq s}d(-h'_{R_i,I_i,f_i,\infty}(t_i)).
\end{align*}
\item for all but countably many $r \in \mathbb{R}$ where $h_{R,I,f,\infty}(r)$ is not differentiable,
$$h'_{R,I,f,\infty,\pm}(r)=d/dr^\pm\int_{[\mathbf{0},\mathbf{\infty})^s}D_\infty(\mathbf{t},r) \prod_{1 \leq i \leq s}d(-h'_{R_i,I_i,f_i,\infty}(t_i)).$$
\end{enumerate}
\end{theorem}
\begin{proof}
(1) is a consequence of Lemma \ref{4.5 lem: converge of concave leads to converge of derivative}. We prove (2) using (1). We have
\begin{enumerate}[(a)]
\item $\lim_{p \to \infty}D_p(\mathbf{t},r)=D_\infty(\mathbf{t},r)$ for every $\mathbf{t},r$;
\item $D_p(\mathbf{t},r)$ is increasing for every $p$;
\item $h'_{R_i,I_i,f_i,p,\pm}(t_i)$ is uniformly bounded;
\item $h'_{R_i,I_i,f_i,p}(C_i^+)=h'_{R_i,I_i,f_i,\infty}(C_i^+)=0$ and $h'_{R_i,I_i,f_i,p}(0^-)=h'_{R_i,I_i,f_i,\infty}(0^-)=0$;
\item (1) says $$\lim_{p \to \infty}h'_{R_i,I_i,f_i,p,+}(t_i)=\lim_{p \to \infty}h'_{R_i,I_i,f_i,p,-}(t_i)=h'_{R_i,I_i,f_i,\infty}(t_i)$$
for all but countably many $t_i$.
\end{enumerate}
Also, since $D_p$, $D_\infty$ are nonzero only on $[\mathbf{0},\infty)^{s+1}$ and $h'_p$, $h'_\infty$ are zero outside $\prod_{1\leq i \leq s}[0,C_i]$, the integral is not affected if we replace $\prod_{1\leq i \leq s}[0,C_i]$ by any interval containing it. From (a)-(e) and Theorem \ref{2.4 convergence of RS integral on f and alpha side} we deduce
\begin{align*}
h_{R,I,f,\infty}(r)=\lim_{p \to \infty}h_{R,I,f,p}(r)
=\lim_{p \to \infty}\int_{\prod_{1\leq i \leq s}[0,C_i^+]}D_p(\mathbf{t},r) \prod_{1 \leq i \leq s}d(-h'_{R_i,I_i,f_i,p}(t_i))\\
=\int_{\prod_{1\leq i \leq s}[0,C_i^+]}D_\infty(\mathbf{t},r) \prod_{1 \leq i \leq s}d(-h'_{R_i,I_i,f_i,\infty}(t_i))
\\
=\int_{[\mathbf{0},\infty)^s}D_\infty(\mathbf{t},r) \prod_{1 \leq i \leq s}d(-h'_{R_i,I_i,f_i,\infty}(t_i)).
\end{align*}
So (2) is proved. Note that for any $p$, $h_{R,I,f,p}$ is convex, so we deduce (3) from (2) and Lemma \ref{4.5 lem: converge of concave leads to converge of derivative}.
\end{proof}
\subsection{Pointwise convergence of derivatives}
In last subsection, we prove a convergence result for $h$-function and an ``almost everywhere'' convergence result for derivatives. However, in some applications we need to look at the convergence of left or right derivative at a certain point. For example, we recognize the Hilbert-Kunz multiplicity from right derivative at $0$ and the $F$-signature from left derivative at $1$. In this subsection, we describe stronger conditions on $h$-function that lead to the pointwise convergence of left and right derivatives.

The two following conditions are essential and will appear in the statements later:
\begin{enumerate}[(a)]
\item $\frac{\partial D_{\phi,p}}{\partial r^\pm}(\mathbf{t},r)$ and $\frac{\partial D_{\phi,\infty}}{\partial r^\pm}(\mathbf{t},r)$ are Riemann-Stieltjes integrable with respect to some functions in terms of $\mathbf{t}$ for a fixed $r$.
\item $\frac{\partial D_{\phi,p}}{\partial r^\pm}(\mathbf{t},r) \to \frac{\partial D_{\phi,\infty}}{\partial r^\pm}(\mathbf{t},r)$ pointwise.
\end{enumerate}
Here we do not require that $\frac{\partial D_{\phi,p}}{\partial r^\pm}(\mathbf{t},r)$ or $\frac{\partial D_{\phi,\infty}}{\partial r^\pm}(\mathbf{t},r)$ is continuous in $r$-direction.

We have seen in Corollary \ref{3 hpartial monotonicity} that for any fixed $r$, $\mathbf{t} \to \frac{\partial D_{\phi,p}}{\partial r^\pm}(\mathbf{t},r)$ is increasing in terms of $\mathbf{t}$. We prove the analogous result for $D_{\phi,\infty}$.
\begin{lemma} Whenever $D_{\phi,\infty}(\mathbf{t},r)$ exists, $\mathbf{t} \to \frac{\partial D_{\phi,\infty}}{\partial r^\pm}(\mathbf{t},r)$ is increasing in terms of $\mathbf{t}$.    
\end{lemma}
\begin{proof}
For fixed $\mathbf{t}$, $D_{\phi,p}(\mathbf{t},r)$ is convex with respect to $r$. In particular, it is absolutely continuous, and for any $\epsilon'>\epsilon>0$,
$$D_{\phi,p}(\mathbf{t},r+\epsilon')-D_{\phi,p}(\mathbf{t},r+\epsilon)=\int_{r+\epsilon}^{r+\epsilon'}\frac{\partial}{\partial r}D_{\phi,p}(\mathbf{t},x)dx.$$
Letting $p \to \infty$, we see
$$D_{\phi,\infty}(\mathbf{t},r+\epsilon')-D_{\phi,\infty}(\mathbf{t},r+\epsilon)=\lim_{p \to \infty}\int_{r+\epsilon}^{r+\epsilon'}\frac{\partial}{\partial r}D_{\phi,p}(\mathbf{t},x)dx.$$
Since the limit function of concave functions is still concave, we see
\begin{align*}
\frac{\partial}{\partial r^+}D_{\phi,\infty}(\mathbf{t},r)=\lim_{0<\epsilon<\epsilon' \to 0}\frac{1}{\epsilon'-\epsilon}(D_{\phi,\infty}(\mathbf{t},r+\epsilon')-D_{\phi,\infty}(\mathbf{t},r+\epsilon))\\
=\lim_{0<\epsilon<\epsilon' \to 0}\lim_{p \to \infty}\frac{1}{\epsilon'-\epsilon}\int_{r+\epsilon}^{r+\epsilon'}\frac{\partial}{\partial r}D_{\phi,p}(\mathbf{t},x)dx.  
\end{align*}
So the increasing property of $\frac{\partial}{\partial r}D_{\phi,p}(\mathbf{t},x)$ will lead to the increasing property of $\frac{\partial}{\partial r^+}D_{\phi,\infty}(\mathbf{t},r)$ with respect to $\mathbf{t}$.
\end{proof}

Lemma \ref{5.5 lem: partial commutes with integration} leads to the following two results on partial derivatives in fixed and limit characteristic.
\begin{theorem}\label{5.5 integral formula derivative}
We work under Settings \ref{5.1 Tensor product setting}. Let $C_i$ be the $F$-threshold of $f_i$ with respect to $I_i$. Take $\epsilon_i>0$. Suppose for $\phi \in k[T_1,\ldots,T_s]$ and $r=r_0$, $\frac{\partial D_{\phi,p}}{\partial r^\pm}(\mathbf{t},r)$ is Riemann-Stieltjes integrable with respect to $d(-h'_1)\ldots d(-h'_s)$. Then for $r=r_0$ we have
\begin{align*}
h'_{\phi(\underline{f}),p,+}(r)=\int_{\prod_{1\leq i \leq s}[0,C_i+\epsilon_i]}\frac{\partial}{\partial r^+}D_{\phi,p}(\mathbf{t},r)\prod_{1 \leq i \leq s}d(-h'_{f_i,p}(t_i)).
\end{align*}
\end{theorem}
\begin{proof}
By Corollary \ref{3 hpartial monotonicity},
$\frac{\partial D_{\phi,p}}{\partial r^\pm}(\mathbf{t},r)$ is increasing with respect to $\mathbf{t}$. So we get the result using Lemma \ref{5.5 lem: partial commutes with integration}.
\end{proof}
\begin{theorem}\label{5.5 integral formula derivative limit char}
We work under Settings \ref{5.1 Tensor product mod p setting}. Assume $h_{R_i,I_i,f_i,\infty}$ exists for any $i$, and choose $C_i$ as in Settings \ref{5.1 Tensor product mod p setting}. Take $\epsilon_i>0$. Suppose for $\phi \in \mathbb{Z}[T_1,\ldots,T_s]$ and $r=r_0$, we have 
\begin{enumerate}[(a)]
\item $\frac{\partial D_{\phi,p}}{\partial r^\pm}(\mathbf{t},r)$ is Riemann-Stieltjes integrable with respect to $d(-h'_{1,p})\ldots d(-h'_{s,p})$ and $\frac{\partial D_{\phi,\infty}}{\partial r^\pm}(\mathbf{t},r)$ is Riemann-Stieltjes integrable with respect to $d(-h'_{1,\infty})\ldots d(-h'_{s,\infty})$ on $\prod_{1\leq i \leq s}[0,C_i+\epsilon_i]$.
\item $\frac{\partial D_{\phi,p}}{\partial r^\pm}(\mathbf{t},r) \to \frac{\partial D_{\phi,\infty}}{\partial r^\pm}(\mathbf{t},r)$ for every $\mathbf{t}\in\prod_{1\leq i \leq s}[0,C_i+\epsilon_i]$.
\end{enumerate}
Denote $h_{\phi(\underline{f}),\infty}=\lim_{p \to \infty}h_{\phi_p(\underline{f}_p)}$ whose existence is guaranteed by Theorem \ref{5.4 integral formula limit char}. Then for $r=r_0$,
$$h'_{\phi(\underline{f}),\infty,\pm}(r)=\lim_{p \to \infty}h'_{\phi_p(\underline{f}_p),\pm}(r).$$
\end{theorem}
\begin{proof}
We prove for right derivative, and the proof for left derivative is the same. We have
\begin{align*}
\lim_{p \to \infty}h'_{\phi(\underline{f}),p,+}(r)=\lim_{p \to \infty}\frac{d}{d r^+}\int_{\prod_{1\leq i \leq s}[0,C_i+\epsilon_i]}D_{\phi,p}(\mathbf{t},r)\prod_{1 \leq i \leq s}d(-h'_{f_i,p}(t_i))\\
=\lim_{p \to \infty}\int_{\prod_{1\leq i \leq s}[0,C_i+\epsilon_i]}\frac{\partial}{\partial r^+}D_{\phi,p}(\mathbf{t},r)\prod_{1 \leq i \leq s}d(-h'_{f_i,p}(t_i))\\
=\int_{\prod_{1\leq i \leq s}[0,C_i+\epsilon_i]}\frac{\partial}{\partial r^+}D_{\phi,\infty}(\mathbf{t},r)\prod_{1 \leq i \leq s}d(-h'_{f_i,\infty}(t_i))\\
=\frac{d}{d r^+}\int_{\prod_{1\leq i \leq s}[0,C_i+\epsilon_i]}D_{\phi,\infty}(\mathbf{t},r)\prod_{1 \leq i \leq s}d(-h'_{f_i,\infty}(t_i))=h'_{\phi(\underline{f}),\infty,+}(r).
\end{align*}
Here the first and last equalities come from the integral formula, the second and fourth equalities come from Lemma \ref{5.5 lem: partial commutes with integration} which uses condition (a), the third equality comes from condition (b) and Theorem \ref{2.4 convergence of RS integral on f and alpha side}.
\end{proof}

\section{The applications of the integral formulas to computations}
In this section we will mention some applications of the integral formulas.
\subsection{General recursive iteration principal}
Let $l \in \mathbb{N}$. The previous section allows us to do $l$-fold iterations on elements with convergent $h$-function when reduced to characteristic $p$. We consider the following scenario.

\begin{settings}\label{6.1 Iteration setting char p}
Consider the following sequence of elements in some ambient polynomial ring of characteristic $p$:
$$f_{ij}, 1 \leq i \leq s_j, 1\leq j \leq l$$
and the following sequence of elements in some chosen polynomial ring $$\phi_{ij},1 \leq i \leq s_j, 2\leq j \leq l$$
such that for each fixed $j$ and $i \neq i'$, $f_{ij},f_{i'j}$ involve different sets of variables. Assume $f_{ij}=\phi_{ij}(f_{1,j-1},\ldots,f_{s_j,j-1})$. 
\end{settings}
Under Settings \ref{6.1 Iteration setting char p}, the integral formula gives a sequence of equations
$$h_{f_{ij}}(r)=\int_{[0,\infty)^{s_{j-1}}}D_{\phi_{ij}}(t_1,\ldots,t_{s_j},r)\prod_{1 \leq i \leq s_j}d(-h'_{f_{i,j-1}}(t_i)),$$
which allows us to compute $h_{f_{ij}}$ for any $i,j$.

Now we can consider a ``reduction modulo $p$" version of the above settings. To be precise, we have:
\begin{settings}\label{6.1 Iteration setting limit characteristc}
Consider sequence of elements consisting of elements $f_{ij}$ and $\phi_{ij}$ whose indices are the same as Settings \ref{6.1 Iteration setting char p} and lie in some ambient ring, which is a polynomial ring over $\mathbb{Z}$. Assume for large enough $p$, their reduction modulo $p$ counterparts, say $f_{ij,p}$ and $\phi_{ij,p}$ satisfy Settings \ref{6.1 Iteration setting char p}.    
\end{settings}
We have the following propositions.
\begin{proposition}\label{6.1 general iteration}
In Settings \ref{6.1 Iteration setting limit characteristc}, assume the ambient ring is a finitely generated $\mathbb{Z}$-algebra. Assume the following limits of reduction modulo $p$ exist:
$$h_{f_{i1,\infty}}(t),1 \leq i \leq s_1,D_{\phi_{ij},\infty}(\mathbf{t},r),1 \leq i \leq s_j,2 \leq j \leq l.$$
Then the following limits
$$h_{f_{ij,\infty}}(t),1 \leq i \leq s_j,2 \leq j \leq l$$
exist, and
$$h'_{f_{ij,\infty}}(t)=\lim_{p \to \infty}h'_{f_{ij,p,+}}(t)=\lim_{p \to \infty}h'_{f_{ij,p,-}}(t)$$
for all but countably many $t$, where $1 \leq i \leq s_j,2 \leq j \leq l$.
\end{proposition}
\begin{proof}
By Lemma \ref{4.5 lem: converge of concave leads to converge of derivative}, the convergence of $h$ leads to the convergence of $h'_{\pm}$ outside countably many points; $D_\phi$ is always increasing, so the convergence of $D_\phi$ pointwise and the convergence of $h'_{\pm}$ outside countably many points lead to the convergence of $h$ in the next level. So we are done by induction. 
\end{proof}
\begin{proposition}\label{6.1 general iteration derivative}
Under Settings \ref{6.1 Iteration setting limit characteristc}, assume the ambient ring is a finitely generated $\mathbb{Z}$-algebra. Assume the following limits of reduction modulo $p$ exist:
$$h_{f_{i1,\infty}}(t),1 \leq i \leq s_1,D_{\phi_{ij},\infty}(\mathbf{t},r),1 \leq i \leq s_j,2 \leq j \leq l.$$
Assume moreover there are open sets $\Omega_j$ containing the support of $h'_{f_{i,j-1,p}}(t)$ such that
\begin{enumerate}[(a)]
\item $\frac{\partial D_{\phi_{ij},p}}{\partial r^{\pm}}(\mathbf{t},r),\frac{\partial D_{\phi_{ij},\infty}}{\partial r^{\pm}}(\mathbf{t},r)$ are continuous with respect to $\mathbf{t}$ on $\Omega=\prod_{1 \leq i \leq s_j}\Omega_i$;
\item $\frac{\partial D_{\phi_{ij},p}}{\partial r^{\pm}}(\mathbf{t},r)\to\frac{\partial D_{\phi_{ij},\infty}}{\partial r^{\pm}}(\mathbf{t},r)$.
\end{enumerate}
Then we have
$$h'_{f_{ij,p},\pm}(t) \to h'_{f_{ij,\infty},\pm}(t),1 \leq i \leq s_j,2 \leq j \leq l$$
for any $1 \leq i \leq s_j,2 \leq j \leq l$.
\end{proposition}
\begin{proof}
This is true by Theorem \ref{5.5 integral formula derivative limit char} and induction.    
\end{proof}

\subsection{The kernel function of addition of $s$ elements}
\begin{proposition}\label{6.2 h-function of pure power}
Let $R=k[x]$, $I=(x)$, $f=x^n$ for some integer $n$. Then
$$h(t)=\begin{cases}
0 & t \leq 0\\
nt & 0 \leq t \leq 1/n\\
1 & t \geq 1/n
\end{cases}$$
and $h''(t)=-n(\delta_{1/n}-\delta_0)$.
\end{proposition}
\begin{proof}
This is a particular case of Theorem \ref{3 h-function of general monomial ideal}.    
\end{proof}
\begin{theorem}\label{6.2 h-function of diagonal}
Let $\phi=T_1+\ldots+T_s$, $D=D_{\phi,p}(t_1,\ldots,t_s,r)$ be the kernel function of $\phi$ over a field of characteristic $p$. Let $R=k[x_1,\ldots,x_s],I=(x_0,\ldots,x_n),f=\sum_{1 \leq i \leq s}x_i^{n_i}$. Then
$$h_{R,I,f}(r)=n_1\ldots n_sD(1/n_1,\ldots,1/n_s,r).$$
\end{theorem}
\begin{proof}
By integral formula, it suffices to check
$$-h''_{x^n}(t)=n(\delta_{1/n}-\delta_0)$$
and
\begin{align*}
\int_{[0,\infty)^s}D(t_1,\ldots,t_s,r)(n_1\delta_{1/n_1}(t_1)dt_1)\ldots(n_s\delta_{1/n_s}(t_s)dt_s)\\
=n_1\ldots n_sD(1/n_1,\ldots,1/n_s,r).    
\end{align*}
\end{proof}
Thus we recover the $h$-function of the diagonal hypersurface as a multiple of the restriction of the kernel function on certain lines parallel to the last coordinate.

We also show that the kernel function of three variables recovers the kernel function of more variables.
\begin{theorem}\label{6.2 recursive iteration of kernel of s addition}
In any characteristic $p>0$ we have
$$D_{T_1+\ldots+T_j}(t_1,\ldots,t_j,r)=\int_{[0,\infty)}D_{T_1+T_2}(r_1,1,r)d_{r_1}(-\frac{\partial}{\partial r_1}D_{T_1+\ldots+T_{j-1}}(t_1,\ldots,t_j,r_1)).$$
For limit characteristic, we also have
\begin{align*}
D_{T_1+\ldots+T_j,\infty}(t_1,\ldots,t_j,r)\\
=\int_{[0,\infty)}D_{T_1+T_2,\infty}(r_1,1,r)d_{r_1}(-\frac{\partial}{\partial r_1}D_{T_1+\ldots+T_{j-1},\infty}(t_1,\ldots,t_j,r_1)).    
\end{align*}
\end{theorem}
\begin{proof}
We work under Settings \ref{6.1 Iteration setting char p} in characteristic $p$ and Settings \ref{6.1 Iteration setting limit characteristc} for limit characteristic, where all the $f_{ij}$'s and $\phi_{ij}$'s are given in the following chain:
$$(T_1,\ldots,T_s)\to (T_1+T_2,T_3,\ldots,T_s) \to (T_1+T_2+T_3,T_4,\ldots,T_s)\to \ldots \to T_1+\ldots+T_s$$
For each step in the chain, the multivariate integral formula Theorem \ref{5.3 integral formula multivariate} yields
\begin{align*}
D_{T_1+\ldots+T_j}(t_1,\ldots,t_j,r)\\
=\int_{[0,\infty)^2}D_{T_1+T_2}(r_1,r_2,r)d_{r_1}\frac{\partial}{\partial r_1}D_{T_1+\ldots+T_{j-1}}(t_1,\ldots,t_j,r_1)dh'_{T_j}(r_2),    
\end{align*}
we see $-dh'_{T_j}(r_2)=\delta_1-\delta_0$, so the result for characteristic $p>0$ holds. The result for limit characteristic comes from Theorem \ref{2.4 convergence of RS integral on f and alpha side} and Theorem \ref{5.4 integral formula limit char}.
\end{proof}

\begin{corollary}\label{6.2 convergence of kernel function of s addition}
Let $\phi=T_1+\ldots+T_s \in \mathbb{Z}[T_1,\ldots,T_s]$. Then
$$D_{\phi,p}(\mathbf{t},r)\to D_{\phi,\infty}(\mathbf{t},r)$$
exists, and
$$\frac{\partial}{\partial r^\pm}D_{\phi,p}(\mathbf{t},r)\to \frac{\partial}{\partial r^\pm}D_{\phi,\infty}(\mathbf{t},r)$$
for any $\mathbf{t},r$.
\end{corollary}
\begin{proof}
We see the kernel function of the addition $\phi=T_1+T_2$ has a limit kernel function, and satisfies (a) and (b) of Proposition \ref{6.1 general iteration derivative} by Proposition \ref{4.3 Dinftyexists} and Lemma \ref{4.5 convergence of partial derivative}. The partial derivative of $D_\phi$ is continuous by Corollary \ref{4.5 continuity of partial derivative}, so is Riemann-Stieltjes integrable with respect to any monotone function. So the result follows from Proposition \ref{6.1 general iteration} and Proposition \ref{6.1 general iteration derivative} since addition of $s$-elements is the $s-1$-fold iteration of addition of two elements.   
\end{proof}
\begin{remark}\label{6.2 remark on CSTZ result}
The convergence part of \cite[Theorem 15, Theorem 16]{shidelerthesis} is a consequence of Corollary \ref{6.2 convergence of kernel function of s addition} in the case $\mathbf{t}=(\frac{1}{d_1},\ldots,\frac{1}{d_s})$.    
\end{remark}

\subsection{$h$-function of binomials}In this subsection, we calculate the $h$-function of binomials. Let $x_1,\ldots,x_{s_1},y_1,\ldots,y_{s_2}$ be two sets of variables. Let $R=k[x_1,\ldots,x_{s_1},y_1,\ldots,y_{s_2}]$, $I=(x_1,\ldots,x_{s_1},y_1,\ldots,y_{s_2})$, $f=x_1^{a_1}\ldots x_{s_1}^{a_{s_1}}y_1^{b_1}\ldots y_{s_2}^{b_{s_2}}(x_1^{c_1}\ldots x_{s_1}^{c_{s_1}}+y_1^{d_2}\ldots y_{s_2}^{d_{s_1}})$ be a binomial. Let $f_1=x_1^{c_1}\ldots x_{s_1}^{c_{s_1}}$ and $f_2=y_1^{d_1}\ldots y_{s_2}^{d_{s_2}}$. We see $\dim R=d=s_1+s_2$. Consider the following functions
$$H_{e,R,I,f}(r)=l(R/I^{[q]},f^{rq})$$
and
$$h_{R,I,f}(r)=\lim_{q \to \infty}\frac{H_{e,R,I,f}(r)}{q^d}.$$
For simplicity, we first work with the case $rq \in \mathbb{Z}$. In this case,
\begin{align*}
H_{e,R,I,f}(r)=l(R/I^{[q]},f^{rq})=\\
l(k[\underline{x},\underline{y}]/(\underline{x}^q,\underline{y}^q,x_1^{a_1rq}\ldots x_{s_1}^{a_{s_1}rq}y_1^{b_1rq}\ldots y_{s_2}^{b_{s_2}rq}(x_1^{c_1}\ldots x_{s_1}^{c_{s_1}}+y_1^{d_1}\ldots y_{s_2}^{d_{s_2}})^{rq}).
\end{align*}
There are two cases.

Case 1: there exists $a_ir \geq 1$ or $b_ir \geq 1$. Then $H_{e,R,I,f}(r)=q^d$ for any $q$ and $h_{e,R,I,f}(r)=1$.

Case 2: $a_ir\leq 1$ for any $1 \leq i \leq s_1$ and $b_ir\leq 1$ for any $1 \leq i \leq s_2$. Then we have
\begin{align*}
q^d-H_{e,R,I,f}(r)
=l(k[\underline{x},\underline{y}]/(\underline{x}^q,\underline{y}^q))
\\
-l(k[\underline{x},\underline{y}]/(\underline{x}^q,\underline{y}^q,x_1^{a_1rq}\ldots x_{s_1}^{a_{s_1}rq}y_1^{b_1rq}\ldots y_{s_2}^{b_{s_2}rq}(x_1^{c_1}\ldots x_{s_1}^{c_{s_1}}+y_1^{d_1}\ldots y_{s_2}^{d_{s_2}})^{rq}))\\
=l(k[\underline{x},\underline{y}]/((\underline{x}^q,\underline{y}^q):x_1^{a_1rq}\ldots x_{s_1}^{a_{s_1}rq}y_1^{b_1rq}\ldots y_{s_2}^{b_{s_2}rq}(x_1^{c_1}\ldots x_{s_1}^{c_{s_1}}+y_1^{d_1}\ldots y_{s_2}^{d_{s_2}})^{rq})).
\end{align*}
Since $\underline{x},\underline{y}$ is a regular sequence, we can cancel terms in the colon ideal:
\begin{align*}
(\underline{x}^q,\underline{y}^q):x_1^{a_1rq}\ldots x_{s_1}^{a_{s_1}rq}y_1^{b_1rq}\ldots y_{s_2}^{b_{s_2}rq}(x_1^{c_1}\ldots x_{s_1}^{c_{s_1}}+y_1^{d_1}\ldots y_{s_2}^{d_{s_2}})^{rq}\\
=(x_1^{q-a_1rq},\ldots ,x_{s_1}^{q-a_{s_1}rq},y_1^{q-b_1rq},\ldots ,y_{s_2}^{q-b_{s_2}rq}):(x_1^{c_1}\ldots x_{s_1}^{c_{s_1}}+y_1^{d_1}\ldots y_{s_2}^{d_{s_2}})^{rq}.
\end{align*}
We also have
\begin{align*}
l(k[\underline{x},\underline{y}]/((x_1^{q-a_1rq},\ldots ,x_{s_1}^{q-a_{s_1}rq},y_1^{q-b_1rq},\ldots ,y_{s_2}^{q-b_{s_2}rq}):(x_1^{c_1}\ldots x_{s_1}^{c_{s_1}}+y_1^{d_1}\ldots y_{s_2}^{d_{s_2}})^{rq})\\
=l(k[\underline{x},\underline{y}]/((x_1^{q-a_1rq},\ldots ,x_{s_1}^{q-a_{s_1}rq},y_1^{q-b_1rq},\ldots ,y_{s_2}^{q-b_{s_2}rq}))\\
-l(k[\underline{x},\underline{y}]/((x_1^{q-a_1rq},\ldots ,x_{s_1}^{q-a_{s_1}rq},y_1^{q-b_1rq},\ldots ,y_{s_2}^{q-b_{s_2}rq},(x_1^{c_1}\ldots x_{s_1}^{c_{s_1}}+y_1^{d_1}\ldots y_{s_2}^{d_{s_2}})^{rq})\\
=q^d\prod_{1 \leq i \leq s_1}(1-a_ir)\prod_{1 \leq i \leq s_2}(1-b_ir)\\
-l(k[\underline{x},\underline{y}]/((x_1^{q-a_1rq},\ldots ,x_{s_1}^{q-a_{s_1}rq},y_1^{q-b_1rq},\ldots ,y_{s_2}^{q-b_{s_2}rq},(x_1^{c_1}\ldots x_{s_1}^{c_{s_1}}+y_1^{d_1}\ldots y_{s_2}^{d_{s_2}})^{rq}).
\end{align*}
Dividing by $q^d$ and taking the limit, we get
$$1-h_{R,I,f}(r)=\prod_{1 \leq i \leq s_1}(1-a_ir)\prod_{1 \leq i \leq s_2}(1-b_ir)-h_{R,0,(\underline{x},\underline{y},f_1+f_2)}(\mathbf{1}-r\mathbf{a},\mathbf{1}-r\mathbf{b},r).$$
Finally, the multivariate integral formula Theorem \ref{5.3 integral formula multivariate} yields
\begin{align*}
h_{R,0,(\underline{x},\underline{y},f_1+f_2)}(\mathbf{1}-r\mathbf{a},\mathbf{1}-r\mathbf{b},r)\\
=\int_{[0,\infty)^2}D_{T_1+T_2}(r_1,r_2,r)d(-h'_{k[\underline{x}],0,(\underline{x},f_1)}(\mathbf{1}-r\mathbf{a},r_1))d(-h'_{k[\underline{y}],0,(\underline{y},f_2)}(\mathbf{1}-r\mathbf{b},r_2)).
\end{align*}
The left side of the equation is continuous with respect to $r$ by continuity of $h$, and the right side is also continuous by Corollary \ref{2.4 convergence of RS integral continuous version}. Now the equation on both sides are continuous with respect to $r$ and it holds on $\mathbb{Z}[\frac{1}{p}]$, so it holds for all $r \in \mathbb{R}$.

In the above formula, $h_{k[\underline{x}],0,(\underline{x},f_1)}$ and $h_{k[\underline{y}],0,(\underline{y},f_2)}$ are $h$-function of monomials, so they are computable using Theorem \ref{3 h-function of general monomial ideal} and independent of the characteristic. Therefore, the above formula for $h$-function holds in both characteristic $p$ and limit characteristic. In sum, we have:
\begin{theorem}\label{6.2 h-function of binomial}
Under the notations above, in characteristic $p$ we have: if $r \geq 0$ such that $a_ir\leq 1$ and $b_ir\leq 1$ for any $i$, then
\begin{align*}
h_{R,I,f,p}(r)=1-\prod_{1 \leq i \leq s_1}(1-a_ir)\prod_{1 \leq i \leq s_2}(1-b_ir)+\\
\int_{[0,\infty)^2}D_{T_1+T_2,p}(r_1,r_2,r)d(-h'_{k[\underline{x}],0,(\underline{x},f_1)}(\mathbf{1}-r\mathbf{a},r_1))d(-h'_{k[\underline{y}],0,(\underline{y},f_2)}(\mathbf{1}-r\mathbf{b},r_2)).
\end{align*}
In particular, when $p \to \infty$, $h_{R,I,f,p}(r) \to h_{R,I,f,\infty}(r)$ where
\begin{align*}
h_{R,I,f,\infty}(r)=1-\prod_{1 \leq i \leq s_1}(1-a_ir)\prod_{1 \leq i \leq s_2}(1-b_ir)+\\
\int_{[0,\infty)^2}D_{T_1+T_2,\infty}(r_1,r_2,r)d(-h'_{k[\underline{x}],0,(\underline{x},f_1)}(\mathbf{1}-r\mathbf{a},r_1))d(-h'_{k[\underline{y}],0,(\underline{y},f_2)}(\mathbf{1}-r\mathbf{b},r_2)).
\end{align*}
\end{theorem}
\begin{remark}
We remark that this calculates certain $F$-signature of pairs. Let $f=x^ay^b(x^u+y^v)$, and we compute the $h$-function of $f$ as follows. We have
$$h_{k[x],0,(x,x^u)}(t_1,t_2)=\begin{cases}
0 & t_1\leq 0,t_2 \leq 0\\
t_1 & t_1,t_2 \geq 0,t_1\leq ut_2\\
ut_2 & t_1,t_2 \geq 0,t_1\geq ut_2.
\end{cases}$$
Suppose $0 \leq 1-ra \leq 1$, then $h''_{r_1}(1-ra,r_1)=u\delta_{\frac{1-ra}{u}}-u\delta_0$. Similarly
$$h_{k[y],0,(y,y^v)}(t_1,t_2)=\begin{cases}
0 & t_1\leq 0,t_2 \leq 0\\
t_1 & t_1,t_2 \geq 0,t_1\leq vt_2\\
vt_2 & t_1,t_2 \geq 0,t_1\geq vt_2
\end{cases}$$
and $h''_{r_2}(1-rb,r_2)=v\delta_{\frac{1-rb}{v}}-v\delta_0$. By integral formula we see for $r$ such that $1-ra\geq 0$,$1-rb\geq 0$,
$$h_{R,I,f,p}(r)=1-(1-ra)(1-rb)+uvD_p(\frac{1-ra}{u},\frac{1-rb}{v},r).$$
We can apply a similar argument to show: if the $h$-function of $x^ay^b(x^u+y^v)$ is $\phi_{a,b,u,v}(r)=1-(1-ra)(1-rb)+uvD_p(\frac{1-ra}{u},\frac{1-rb}{v},r)$, then the $h$-function of $x^ay^b(x^u+y^v)^c$ is 
$$\phi_{a,b,c,u,v}(r)=\phi_{a/c,b/c,u,v}(cr)=1-(1-ra)(1-rb)+uvD_p(\frac{1-ra}{u},\frac{1-rb}{v},cr).$$   
Letting $p\to \infty$ yields the concrete expression of $h_{x^ay^b(x^u+y^v)^c}$, which is a piecewise polynomial since $D_\infty$ is a piecewise polynomial. However, its concrete expression in characteristic $p$ allows us to analyze attached points and prove positivity results. In the above expression of $\phi_p(r)=\phi_{a,b,c,u,v}(r)$, the eventual behavior of this $\phi_p(r)$ depends on the segment $S'$, starting from $(\frac{1}{u},\frac{1}{v},0)$, pointing at direction $(-\frac{a}{u},-\frac{b}{v},c)$ until it hits the boundary of the unit cube $[0,1]^3$.
We apply Proposition \ref{4.4 density of eventually unattached segment} to get that if this segment is not eventually attached, then the set of unattached point is dense on this segment. Given fixed $a,b,u,v$, we can always judge whether the segment is attached or eventually attached. Note that even if the segment is eventually unattached, it may happen that certain point on this segment is an attached point.
\end{remark}
\begin{remark}\label{6.2 BCPT rmk6.5 disproved}
We answer a question raised by Brosowsky, Coskun, Pande and Tucker in an unpublished work, which is contrary to the expectation. Assume $u=v=1$, $c \leq b \leq a<b+c$, $a+b+c$ is odd, $abc \neq 0$. Then $S'$ is the segment connecting $(1,1,0)$ with some other point on the boundary of the unit cube. Let the interval $I$ represent the smaller segment $S'\cap T_0$, which is nonempty since $a,b,c$ satisfy the strict triangle inequality. Since $abc \neq 0$, $S'$ is not upright. Since $a+b+c$ is odd, any linear combination $\pm a\pm b \pm c \neq 0$. So the segment is not parallel to the planes in $F$ because its direction is not perpendicular to the normal vector of these planes. Thus, the segment $S'$ is eventually unattached, and by Proposition \ref{4.4 density of eventually unattached segment}, there is a dense subset inside $I$ consisting of unattached points in characteristic $p$. 

On the other hand, a suitable choice of $a,b,c,r$ gives us an attached point in large characteristic $p$. We set $b=c$, and let $a$ be any odd integer such that $b\leq a<2b$, $u=v=1$, $r=\frac{1}{2b}$. Then
$$\phi_p(r)=1-\frac{1}{2}(1-ra)+D_p(1-ra,\frac{1}{2},\frac{1}{2}).$$
If $p$ is odd, then by Proposition \ref{4.4 eventuallyattachedcriterion}, the segment $\{(t_1,\frac{1}{2},\frac{1}{2})|0 \leq t_1 \leq 1\}$ is attached for characteristic $p$. Thus $D_p(1-ra,\frac{1}{2},\frac{1}{2})=D_\infty(1-ra,\frac{1}{2},\frac{1}{2})$. This is saying that $\phi_p(r)$ stabilizes for $p\geq 3$. Similarly, we can check that for $p\geq 5$, $(\frac{1}{3},\frac{1}{3},\frac{1}{2})$ and its reflections are attached in characteristic $p$, so if $(1-ra,1-rb,rc)$ is equal to this point or its image under reflections and permutations, $\phi_p(r)$ also stabilizes. This is true when $(a,b,c)=\frac{1}{6r}(4,4,3)$ where $\frac{1}{6r}$ is an odd integer. That is to say, although the $h$-function and the $F$-signature function do not stablize on all of $(0,1)$, they may stablize at some points inside $(0,1)$.
\end{remark}

\section{Nonnegativity and positivity-on inequalities conjectured by Watanabe-Yoshida}
In the convergence of $D_{\phi,p}\to D_{\phi,\infty}$ for $\phi=T_1+T_2$, we notice that this convergence comes from above, that is, $D_p \geq D_\infty$ for each fixed $p$. This leads to the following proof of the second part of Watanabe-Yoshida's conjecture. Recall that $S_{p,n,d}=\mathbb{F}_p[[x_0,\ldots,x_n]]/\sum_i x_i^d$ is the Fermat hypersurface of degree $d$. We set $\phi_{n,p}(r)=h_{\mathbb{F}_p[[x_0,\ldots,x_n]],(x_0,\ldots,x_n),\sum_i x_i^2}(r)$ and $\phi_{n,\infty}(r)=h_{\mathbb{Z}[[x_0,\ldots,x_n]],(x_0,\ldots,x_n),\sum_i x_i^2,\infty}(r)$.

\begin{theorem}\label{7.1 WY inequality proof}
For any fixed characteristic $p\geq 3$,
\begin{enumerate}
\item $e_{HK}(S_{p,n,2}) \geq \lim_{p \to \infty}e_{HK}(S_{p,n,2}).$
\item the inequality is strict if and only if $n\geq 4$, otherwise it is an equality.
\end{enumerate}
\end{theorem}
\begin{proof}
(1) By \cite[Theorem 16]{shidelerthesis} or Remark \ref{6.2 remark on CSTZ result}, $\lim_{p \to \infty}e_{HK}(S_{p,n,2})=\phi'_{n,\infty,+}(0)$. So it suffices to prove $\phi'_{n,p,+}(0)\geq \phi'_{n,\infty,+}(0)$, and it suffices to prove for any $r>0$, $\phi_{n,p}(r)\geq \phi_{n,\infty}(r)$. For $n=0$ we see $x_0^2$ is a monomial and $h$-functions of monomials are independent of characteristic, so $\phi_{0,p}=\phi_{0,\infty}$. By definition we see for any $n \in \mathbb{N}$,
$$\phi_{n,p}(t)=\phi_{n,\infty}(t)=0,t\leq 0$$
and
$$\phi_{n,p}(t)=\phi_{n,\infty}(t)=1,t\geq 1.$$
In particular, $\phi'_{n,p,+}(1)=\phi'_{n,\infty,+}(1)=0$. Now we prove by induction. We have proved the case $n=0$; suppose we have proved $\phi_{n,p}(t)\geq \phi_{n,\infty}(t)$ for some $n$. Now for $n+1$, the integral formula yields
$$\phi_{n+1,p}(t)=\int_0^{1^+}\int_0^{1^+}D_p(t_1,t_2,t)d(-\phi'_{n,p}(t_1))d(-\phi'_{0,p}(t_2)).$$
But $\phi''_{0,p}=-2\delta_{1/2}+2\delta_0$ and $D_p(t_1,0,t)=0$, so
$$\phi_{n+1,p}(t)=\int_0^{1^+}2D_p(t_1,1/2,t)d(-\phi'_{n,p}(t_1)).$$
Similarly,
$$\phi_{n+1,\infty}(t)=\int_0^{1^+}2D_\infty(t_1,1/2,t)d(-\phi'_{n,\infty}(t_1)).$$
Note that for any fixed $t_1$, $t \to 2D_p(t_1,1/2,t)$ is a concave function. Thus by Lemma \ref{7.1 nonnegativity lemma} and induction hypothesis
$$\int_0^{1^+}2D_p(t_1,1/2,t)d(-\phi'_{n,p}(t_1))\geq \int_0^{1^+}2D_p(t_1,1/2,t)d(-\phi'_{n,\infty}(t_1)).$$
Now $\phi_{n,\infty}(t_1)$ is concave, so $-\phi'_{n,\infty}$ is increasing and $D_p\geq D_\infty$ for any $p$ on $[0,1]^3$, hence
$$\int_0^{1^+}2D_p(t_1,1/2,t)d(-\phi'_{n,\infty}(t_1))\geq \int_0^{1^+}2D_\infty(t_1,1/2,t)d(-\phi'_{n,\infty}(t_1)).$$
Combining all these inequalities above, we get $\phi_{n+1,p}(t)\geq \phi_{n+1,\infty}(t)$. So we are done by induction.

(2) We first check for small $n$. We see $h_{x^2,p}$ is independent of characteristic. $h_{x^2+y^2,p}$ is independent of characteristic for $p\geq 3$, since in this case $p$ is odd and after a linear transformation $x^2+y^2 \to xy$, and $h_{xy,p}$ is independent of characteristic. We also see this from the fact that $(1/2,1/2,x)$ is attached when $p \geq 3$. This confirms the statement for the case $n=0,1$.

A direct computation from Theorem \ref{3 h-function of general monomial ideal} yields
$$h_{xy,p}(t)=\begin{cases}
0 & t \leq 0\\
2t-t^2 & 0 \leq t \leq 1\\
1 & t \geq 1.
\end{cases}$$
Thus $-h''_{xy,p}(t)=2\chi_{(0,1)}(t)$ is supported on $[0,1]$. 

We now claim that for $n \geq 2$ and $p \geq 3$, the set $\{t:\phi_{n,p}(t)>\phi_{n,\infty}(t)\}$ is dense in $(0,1)$. By the integral formula,
\begin{align*}
\phi_{n,p}(t)=\int_{0}^{1^+}\int_0^{1^+}D_p(t_1,t_2,r)(-\phi''_{n-1,p}(t_1))2\delta_{1/2}(t_2)dt_2dt_1\\
=\int_{0}^{1^+}D_p(t_1,1/2,r)(-\phi''_{n-1,p}(t_1))dt_1.
\end{align*}
Similarly,
\begin{align*}
\phi_{n,\infty}(t)=\int_{0}^{1^+}D_\infty(t_1,1/2,r)(-\phi''_{n-1,\infty}(t_1))dt_1.
\end{align*}
By \cite[Theorem 24]{shidelerthesis}, $\phi_{n-1,\infty}(t)$ is either a polynomial of degree $n$, or two pieces of two polynomials of degree $n$. When $n \geq 2$, $-\phi''_{n-1,\infty}(t_1)$ is one or two pieces of a polynomial of degree $n-2 \geq 0$, so it has only finitely many zeros. In particular, the support of $-\phi''_{n-1,\infty}(t_1)$ is $[0,1]$. For any $r \in (0,1)$, there is some $t_1 \in (0,1)$ such that $(t_1,1/2,r) \in T_0^{int}$. Suppose we have $r \notin 1/2\mathbb{Z}[1/p]$, then $\{(t_1,1/2,r)|t_1 \in \mathbb{R}\}$ is an eventually unattached segment by Proposition \ref{4.4 eventuallyattachedcriterion}. So the set of unattached points is also dense in $\{(t_1,1/2,r)|t_1 \in \mathbb{R}\}\cap T_0$.Thus
$$\int_{0}^{1^+}D_\infty(t_1,1/2,r)(-\phi''_{n-1,\infty}(t_1))dt_1< \int_{0}^{1^+}D_p(t_1,1/2,r)(-\phi''_{n-1,\infty}(t_1))dt_1.$$
By Lemma \ref{7.1 nonnegativity lemma}, we also have
$$\int_{0}^{1^+}D_p(t_1,1/2,r)(-\phi''_{n-1,\infty}(t_1))dt_1\leq \int_{0}^{1^+}D_p(t_1,1/2,r)(-\phi''_{n-1,p}(t_1))dt_1.$$
Thus $\phi_{n,\infty}(r)<\phi_{n,p}(r)$. We finish the proof of the claim by observing that $(0,1)\backslash 1/2\mathbb{Z}[1/p]$ is dense in $(0,1)$, so the claim is proved.

Now we come back to the proof of (2). We need to prove $\phi'_{n,p,+}(0)>\phi'_{n,\infty,+}(0)$ if $n \geq 4$ and $\phi'_{n,p,+}(0)=\phi'_{n,\infty,+}(0)$ if $n \leq 3$. If $n=0$ or $n=1$, then the $h$-function $h_{x^2}$ or $h_{x^2+y^2}$ is independent of $p$, so the equality holds. Now we assume $n\geq 2$. By integral formula and commutativity of partial derivative,
$$\phi'_{n,p,+}(0)=\int_0^{1^+}\int_0^{1^+}\frac{\partial}{\partial r^+}D_p(t_1,t_2,0)(-\phi''_{n-1,p}(t_1))2\delta_{1/2}(t_2)dt_2dt_1.$$
But we have seen $\frac{\partial}{\partial r^+}D_p(t_1,t_2,0)=\min\{t_1,t_2\}$, thus the above equation is equal to
\begin{align*}
2\int_0^{1^+}\min\{t_1,1/2\}(-\phi''_{n-1,p}(t_1))dt_1\\
=2\min\{t_1,1/2\}(-\phi'_{n-1,p}(t_1))|_0^{1^+}+2\int_0^{1^+}\chi_{(0,1/2)}\phi'_{n-1,p}(t_1)dt_1\\
=2\phi_{n-1,p}(1/2)-2\phi_{n-1,p}(0)=2\phi_{n-1,p}(1/2).
\end{align*}
We apply the integral formula again:
$$\phi_{n-1,p}(1/2)=2\int_0^{1^+}D_p(t_1,1/2,1/2)(-\phi''_{n-2,p}(t_1))dt_1.$$
Here we see $(t_1,1/2,1/2)$ are all attached points, which means
$$D_p(t_1,1/2,1/2)=D_\infty(t_1,1/2,1/2)=t_1/2-t_1^2/4.$$
and we see $D_p(t_1,1/2,1/2)=1/4$ for $t_1 \geq 1$. Using integration by parts and noticing that all boundary condition vanishes, we see
$$\phi_{n-1,p}(1/2)=2\int_0^{1^+}(-\frac{\partial^2}{\partial t_1^2}D_p(t_1,1/2,1/2))\phi_{n-2,p}(t_1)dt_1=\int_0^1\phi_{n-2,p}(t_1)dt_1.$$
Suppose $2 \leq n\leq 3$, then $0 \leq n-2 \leq 1$, thus
$$\phi'_{n,p,+}(0)=\int_0^1\phi_{n-2,p}(t)dt$$
is independent of $p$, and the equality holds. Otherwise $n-2\geq 2$, so there is a dense subset of $[0,1]$ such that $\phi_{n-2,p}(t)>\phi_{n-2,\infty}(t)$ on this dense subset, and both are continuous functions. Hence 
$$\int_0^1\phi_{n-2,p}(t)dt>\int_0^1\phi_{n-2,\infty}(t)dt,$$
which means $\phi'_{n,p,+}(0)>\phi'_{n,\infty,+}(0)$.
\end{proof}
\begin{remark}
One might ask whether a similar property for the $F$-signature $s(S_{p,n,2})$ holds. However, this is trivial. Just observe that $S_{p,n,2}$ is a hypersurface ring of multiplicity $2$ and for such rings $e_{HK}(S_{p,n,2})+s(S_{p,n,2})=2$ by \cite[Example 2.3]{WY04}, so for fixed $p$, $s(S_{p,n,2})\leq \lim_{p \to \infty}s(S_{p,n,2})$ and the equality holds if and only if $n \leq 3$.
\end{remark}

In the same manner, we strengthen the following result by Caminata-Shideler-Tucker-Zerman.
\begin{proposition}[Caminata-Shideler-Tucker-Zerman]\label{7.2 WY strict inequality on F-sig CSTZ}
Assume $d$ is odd and $p=d^2-d-1$ is a prime number, then
$$s(S_{p,d,d})<\lim_{p \to \infty}s(S_{p,d,d})=\frac{1}{2^{d-1}(d-1)!}.$$
\end{proposition}
Note that if $d=2$, then $s(S_{p,2,2})=2-e_{HK}(S_{p,2,2})$ by \cite[Example 2.3]{WY04} and is independent of $p$, so the equality holds. In $d=3$, the strict inequality has already been proved by Caminata-Shideler-Tucker-Zerman. We also point out that only the case $p>d$ is worth studying here, since when $p \leq d$, $S_{p,d,d}$ is not even $F$-pure by Fedder's criterion, so the $F$-signature is $0$.

We prove the following stronger statement.
\begin{proposition}\label{7.2 WY strict inequality on F-sig}
Assume $p$ is a prime number, $d$ is an integer such that $p>d\geq 3$. Then
$$s(S_{p,d,d})<\lim_{p \to \infty}s(S_{p,d,d})=\frac{1}{2^{d-1}(d-1)!}.$$
\end{proposition}
\begin{proof}
For $n \in \mathbb{Z}$, denote $\psi_{n,p}=h_{x_0^d+\ldots+x_{n-1}^d,p}(t)$. It suffices to prove $\psi'_{d+1,p,-}(1)<\psi'_{d+1,\infty,-}(1).$ We make the following statements:
\begin{enumerate}
\item $\psi'_{d+1,p,-}(1)<\psi'_{d+1,\infty,-}(1)$
\item $\psi_{d,\infty}(1-1/d) < \psi_{d,p}(1-1/d)$
\item $\psi'_{d-1,\infty,+}(1-2/d)>\psi'_{d-1,\infty,-}(1)$
\end{enumerate}
We prove (3) $\Rightarrow{}$ (2) $\Rightarrow{}$ (1) and (3) is true for $d \geq 3$.

(2) $\Rightarrow{}$ (1): by integral formula we see
$$\psi'_{d+1,p,-}(1)=\int_0^{1^+}\int_0^{1^+}\frac{\partial}{\partial r^-}D_p(t_1,t_2,1)(-\psi''_{d,p}(t_1))(-\psi''_{1,p}(t_2))dt_2dt_1.$$
We see $-\psi''_{1,p}(t_2)=d\delta_{1/d}(t_2)-d\delta_0(t_2)$, and for  $0 \leq t_1,t_2 \leq 1$,
$$\frac{\partial}{\partial r^-}D_p(t_1,t_2,1)=\max\{0,t_1+t_2-1\}.$$
For $t_1\geq 1$, $t_2,r\leq 1$,
$$D_p(t_1,t_2,r)=t_2r, \frac{\partial}{\partial r^\pm}D_p(t_1,t_2,1)=t_2.$$
Thus
$$\frac{\partial}{\partial r^-}D_p(t_1,1/d,1)=f(t_1)=\begin{cases}
0 & t_1 \leq 1-1/d\\
t_1+1/d-1 & 1-1/d \leq t_1 \leq 1\\
1/d & t_1 \geq 1.
\end{cases}$$
And $-f''(t_1)=\delta_1-\delta_{1-1/d}$. Plugging in the expression of $\psi'_{d+1,p,-}(1)$, we see
\begin{align*}
\psi'_{d+1,p,-}(1)=d\int_0^{1^+}\frac{\partial}{\partial r^-}D_p(t_1,1/d,1)(-\psi''_{d,p}(t_1))dt_1\\
=d\int_0^{1^+}f(t_1)(-\psi''_{d,p}(t_1))dt_1\\
=d\int_0^{1^+}(-f''(t_1))\psi_{d,p}(t_1)dt_1=d(\psi_{d,p}(1)-\psi_{d,p}(1-1/d)).
\end{align*}
Similarly, $\psi'_{d+1,\infty,-}(1)=d(\psi_{d,\infty}(1)-\psi_{d,\infty}(1-1/d))$. Since $\psi_{d,p}(1)=\psi_{d,\infty}(1)=1$, we see $\psi_{d,\infty}(1-1/d) < \psi_{d,p}(1-1/d)$ implies $\psi'_{d+1,p,-}(1)<\psi'_{d+1,\infty,-}(1)$.

(3) $\Rightarrow{}$ (2): We have
$$\psi_{d,p}(1-1/d)=d\int_0^1D_p(t_1,1/d,1-1/d)(-\psi''_{d-1,p}(t_1))dt_1$$
and
$$\psi_{d,\infty}(1-1/d)=d\int_0^1D_\infty(t_1,1/d,1-1/d)(-\psi''_{d-1,\infty}(t_1))dt_1.$$
There is a chain of two inequalities
\begin{align*}
\int_0^1D_p(t_1,1/d,1-1/d)(-\psi''_{d-1,p}(t_1))dt_1\\
\geq \int_0^1D_p(t_1,1/d,1-1/d)(-\psi''_{d-1,\infty}(t_1))dt_1\\
\geq \int_0^1D_\infty(t_1,1/d,1-1/d)(-\psi''_{d-1,\infty}(t_1))dt_1.
\end{align*}
So if one of the two inequalities is strict, we will get strict inequality $\psi_{d,p}(1-1/d)>\psi_{d,\infty}(1-1/d)$. We check the second inequality:
$$\int_0^1D_p(t_1,1/d,1-1/d)(-\psi''_{d-1,\infty}(t_1))dt_1\geq \int_0^1D_\infty(t_1,1/d,1-1/d)(-\psi''_{d-1,\infty}(t_1))dt_1.$$
Consider the segment inside $T_0^{int}$: $\{(t_1,1/d,1-1/d),0 \leq t_1 \leq 1\}\cap T_0^{int}=\{(t_1,1/d,1-1/d),1-2/d \leq t_1 \leq 1\}$. Since $p>d\geq 3$, $d$ does not divide $2p^m$ for any $m$. So for any $n$, $p^n/d,p^n(1-1/d)$ are not half integers, and this segment is not on an eventually attached segment by Proposition \ref{4.4 eventuallyattachedcriterion}. Thus the set of unattached points on this segment is dense. So the inequality would be strict if $\Supp(\psi''_{d-1,\infty}(t_1))\cap (1-2/d,1)\neq \emptyset$, and it suffices to prove $\psi'_{d-1,\infty,+}(1-2/d)>\psi'_{d-1,\infty,-}(1)$.

(3) is true for $d \geq 3$: it is well-known that the log canonical threshold $\textup{lct}(x_0^d+\ldots+x_{d-2}^d)=1-1/d$. Thus $\psi_{d-1,\infty}(t)=h_{x_0^d+\ldots+x_{d-2}^d}(t)$ is constant on $[1-1/d,\infty)$, and since it is concave, it cannot have zero left or right derivative on $(0,1-1/d)$. Thus $\psi'_{d-1,\infty,+}(1-2/d)>0=\psi'_{d-1,\infty,-}(1)$. So we are done.
\end{proof}

\section{Limit $h$-functions of Fermat hypersurfaces in different dimensions}
In this section, we reprove Gessel-Monsky's result that $\lim_{p \to \infty}e_{HK}(S_{p,n,2})=1+m_n$ such that $\sum_{n \geq 0}m_nx^n=\sec(x)+\tan(x)$ using the integral formula, and find the corresponding result for $S_{p,n,3}$.
\subsection{Restriction of $D_\infty$ to $\{t_3=1/2\}$}\label{subsection 4.6}
Here we record the value of $K(x,t)=D_\infty(x,t,1/2)$, which will appear in the differential equation in the next subsection. Since $D_\infty$ is a piecewise polynomial on $B_1\sim B_4,T_0$, $K(x,t)$ is a piecewise polynomial on $B_i\cap \{t_3=1/2\}=\Delta_i,1 \leq i \leq 4$, $T_0\cap \{t_3=1/2\}=\Delta_0$. We also record the value of $K$ on $\Delta_5=\{(x,t)|0\leq x\leq 1,t\geq 1\}$ since we need $K(x,1^+)$ in our computation. See Figure \ref{fig: def of delta regions2} and Figure \ref{fig: def of delta regions} for these regions.

\begin{figure}[ht]
\centering
\begin{tikzpicture}[scale=2]
    
    \draw[thick,->] (0,0,0) -- (1.5,0,0) node[anchor=north east]{$t_1$};
    \draw[thick,->] (0,0,0) -- (0,1.5,0) node[anchor=south west]{$t_2$};
    \draw[thick,->] (0,0,0) -- (0,0,1.5) node[anchor=south]{$t_3$};

    \draw[black, thick] (0,0,0) -- (1,0,0) -- (1,1,0) -- (0,1,0) -- cycle;
    \draw[black, thick] (0,0,1) -- (1,0,1) -- (1,1,1) -- (0,1,1) -- cycle;
    \draw[black, thick] (0,0,0) -- (0,0,1);
    \draw[black, thick] (1,0,0) -- (1,0,1);
    \draw[black, thick] (0,1,0) -- (0,1,1);
    \draw[black, thick] (1,1,0) -- (1,1,1);

    \foreach \x in {0,1}
    \foreach \y in {0,1}
    \foreach \z in {0,1} {
        \filldraw (\x,\y,\z) circle (0.5pt);
    }

    \draw[red, thick] (0,0,0) -- (0,1,1) -- (1,0,1) -- cycle;
    \draw[red, thick] (0,0,0) -- (1,0,1) -- (1,1,0) -- cycle;
    \draw[red, thick] (0,0,0) -- (0,1,1) -- (1,1,0) -- cycle;
    \draw[red, thick] (0,1,1) -- (1,0,1) -- (1,1,0) -- cycle;

    \filldraw[blue] (0,0,0) circle (1pt) node[anchor=south east]{};
    \filldraw[blue] (0,1,1) circle (1pt) node[anchor=south west]{};
    \filldraw[blue] (1,0,1) circle (1pt) node[anchor=south east]{};
    \filldraw[blue] (1,1,0) circle (1pt) node[anchor=north west]{};

    \filldraw[black] (1,0,0) circle (1pt) node[anchor=south east]{};
    \filldraw[black] (1,1,1) circle (1pt) node[anchor=north west]{};
    \filldraw[black] (0,0,1) circle (1pt) node[anchor=south east]{};
    \filldraw[black] (0,1,0) circle (1pt) node[anchor=north west]{};

    \filldraw[black] (1,0,1.2) node[anchor=north west]{$t_3=1/2$};

    \coordinate (A) at (1.1,-0.1,0.5);
    \coordinate (B) at (1.1,1.1,0.5);
    \coordinate (C) at (-0.1,1.1,0.5);
    \coordinate (D) at (-0.1,-0.1,0.5);
    \filldraw[fill=gray!50, opacity=0.5] (A) -- (B) -- (C) -- (D) -- cycle;

    \draw[red, thick] (0.5,0,0.5) -- (1,0.5,0.5) -- (0.5,1,0.5) -- (0,0.5,0.5) -- cycle;
    \draw[black, thick] (0,0,0.5) -- (1,0,0.5) -- (1,1,0.5) -- (0,1,0.5) -- cycle;
\end{tikzpicture}
\caption{Section of the unit cube with $t_3=1/2$}
\label{fig: def of delta regions2}
\end{figure}

\begin{figure}
\centering
\begin{tikzpicture}[scale=3, thick]

    \draw[->] (-0.5,0) -- (1.5,0) node[right]{$x$};
    \draw[->] (0,-0.5) -- (0,1.5) node[above]{$t$};

    \draw (0,0) rectangle (1,1);

    \draw (0,0.5) -- (0.5,1) -- (1,0.5) -- (0.5,0) -- cycle;

    \node[anchor=north east] at (0.25,0.25) {$\Delta_1$};
    \node[anchor=north west] at (0.75,0.25) {$\Delta_2$};
    \node[anchor=south west] at (0.75,0.75) {$\Delta_4$};
    \node[anchor=south east] at (0.25,0.75) {$\Delta_3$};
    \node at (0.5,0.5) {$\Delta_0$};
    \node at (0.5,1.2) {$\Delta_5$};

    \draw (1,1) -- (1,1.5);
    \draw (1,1) -- (1.5,1);

    \filldraw (0,0) circle (1pt) node[below left]{$(0,0)$};
    \filldraw (1,0) circle (1pt) node[below right]{$(1,0)$};
    \filldraw (1,1) circle (1pt) node[above right]{$(1,1)$};
    \filldraw (0,1) circle (1pt) node[above left]{$(0,1)$};
    \filldraw (0.5,0) circle (1pt);
    \filldraw (1,0.5) circle (1pt);
    \filldraw (0.5,1) circle (1pt);
    \filldraw (0,0.5) circle (1pt);

\end{tikzpicture}
\caption{$K(x,t)$ is a single polynomial on each of these regions}
\label{fig: def of delta regions}
\end{figure}

\begin{figure}
    \centering
    \begin{tikzpicture}[scale=1.5, thick]
    \draw[->] (-3.0,0) -- (-1.0,0) node[right]{$x$};
    \draw[->] (-2.5,-0.5) -- (-2.5,1.5) node[above]{$t$};

    \draw (-2.5,0) rectangle (-1.5,1);

    \draw (-2.5,0.5) -- (-2.0,1) -- (-1.5,0.5) -- (-2.0,0) -- cycle;

    \node[anchor=north east] at (-2.1,0.25) {\small $xt$};
    \node[anchor=north west] at (-1.9,0.3) {\scriptsize $t/2$};
    \node[anchor=south west] at (-2.55,0.7) {\scriptsize $x/2$};
    \node[anchor=south east] at (-0.6,0.78) {\tiny $xt-x/2$};
    \node[anchor=south east] at (-0.6,0.63) {\tiny $-t/2+1/2$};
    \node at (-1.5,0.9) {$\longleftarrow$};
    \node at (-3.3,0.5) {\tiny $xt/2+x/4+t/4$};
    \node at (-3.3,0.3) {\tiny $-x^2/4-t^2/4-1/16$};
    \node at (-2.4,0.4) {$\longrightarrow$};
    \node at (-2.0,1.2) {\small $x/2$};

    \draw (-1.5,1) -- (-1.5,1.5);
    \draw (-1.5,1) -- (-1.0,1);

    \filldraw (-2.5,0) circle (1pt) node[below left]{};
    \filldraw (-1.5,0) circle (1pt) node[below right]{$1$};
    \filldraw (-1.5,1) circle (1pt) node[above right]{};
    \filldraw (-2.5,1) circle (1pt) node[above left]{$1$};
    \filldraw (-2.0,0) circle (1pt);
    \filldraw (-1.5,0.5) circle (1pt);
    \filldraw (-2.0,1) circle (1pt);
    \filldraw (-2.5,0.5) circle (1pt);
    \draw[->] (-0.5,0) -- (1.5,0) node[right]{$x$};
    \draw[->] (0,-0.5) -- (0,1.5) node[above]{$t$};

    \draw (0,0) rectangle (1,1);

    \draw[blue, thick] (0,0.5) -- (0.5,1) -- (1,0.5) -- (0.5,0) -- cycle;

    \node[anchor=north east] at (0.25,0.25) {\small $x$};
    \node[anchor=north west] at (0.6,0.3) {\scriptsize $1/2$};
    \node[anchor=south west] at (0,0.7) {\scriptsize $0$};
    \node[anchor=south east] at (1.85,0.65) {\tiny $x-1/2$};
    \node at (1,0.8) {$\longleftarrow$};
    \node at (-0.7,0.4) {\tiny $x/2-t/2+1/4$};
    \node at (0.1,0.4) {$\longrightarrow$};
    \node at (0.5,1.2) {\small $0$};

    \draw (1,1) -- (1,1.5);
    \draw (1,1) -- (1.5,1);

    \filldraw (0,0) circle (1pt) node[below left]{};
    \filldraw (1,0) circle (1pt) node[below right]{};
    \filldraw (1,1) circle (1pt) node[above right]{};
    \filldraw (0,1) circle (1pt) node[above left]{};
    \filldraw (0.5,0) circle (1pt);
    \filldraw (1,0.5) circle (1pt);
    \filldraw (0.5,1) circle (1pt);
    \filldraw (0,0.5) circle (1pt);

    \draw[red,thick] (0.5,1) -- (1,1) node[above]{};
    \draw[->] (1.9,0) -- (3.5,0) node[right]{$x$};
    \draw[->] (2,-0.5) -- (2,1.5) node[above]{$t$};

    \draw (2,0) rectangle (3,1);

    \draw[blue, thick] (2,0.5) -- (2.5,1) -- (3,0.5) -- (2.5,0) -- cycle;

    \node[anchor=north east] at (2.25,0.25) {\small $0$};
    \node[anchor=north west] at (2.7,0.3) {\small $0$};
    \node[anchor=south west] at (2,0.7) {\small $0$};
    \node[anchor=south east] at (2.9,0.7) {\small $0$};
    \node at (2.5,0.5) {\small $-1/2$};
    \node at (2.5,1.2) {\small $0$};
    \node at (2.9,1.6) {\tiny $(1/2-x)\delta_1(t)$};
    \draw[red,->] (2.8,1.5) -- (2.8,1.1);

    \draw (3,1) -- (3,1.5);
    \draw (3,1) -- (3.5,1);

    \filldraw (2,0) circle (1pt) node[below left]{};
    \filldraw (3,0) circle (1pt) node[below right]{};
    \filldraw (3,1) circle (1pt) node[above right]{};
    \filldraw (2,1) circle (1pt) node[above left]{};
    \filldraw (2.5,0) circle (1pt);
    \filldraw (3,0.5) circle (1pt);
    \filldraw (2.5,1) circle (1pt);
    \filldraw (2,0.5) circle (1pt);

    \draw[red,thick] (2.5,1) -- (3,1) node[above]{};
    \end{tikzpicture}
    \caption{Evaluating $K,\frac{\partial K(x,t)}{\partial t},\frac{\partial^2 K(x,t)}{\partial t^2}$}
    \label{fig: Kxt value}
\end{figure}

The restriction of $D_\infty$ onto the plane $\{t_3=1/2\}$ is

$$ K(x, t) = D_\infty\left(x, t, \frac{1}{2}\right) = \begin{cases} 
x t & \text{in } \Delta_1 \\ 
\frac{1}{2} t & \text{in } \Delta_2 \\ 
\frac{1}{2} x & \text{in } \Delta_3\cup\Delta_5 \\ 
x t - \frac{1}{2} x - \frac{1}{2} t + \frac{1}{2} & \text{in } \Delta_4\\
\frac{xt}{2}+\frac{x}{4}+\frac{t}{4}-\frac{x^2}{4}-\frac{t^2}{4}-\frac{1}{16} & \text{in }\Delta_0.
\end{cases} $$
We can check that $K(x,t)$ is continuous everywhere, including $\partial \Delta_i,0 \leq i \leq 4$.
$$ \frac{\partial K(x,t)}{\partial t} = \begin{cases} 
x & \text{in } \Delta_1 \\ 
\frac{1}{2} & \text{in } \Delta_2 \\ 
0 & \text{in } \Delta_3\cup\Delta_5 \\ 
x- \frac{1}{2}& \text{in } \Delta_4\\
\frac{x}{2}+\frac{1}{4}-\frac{t}{2}& \text{in }\Delta_0.
\end{cases} $$
We see $\frac{\partial K(x,t)}{\partial t}$ is continuous on $\partial \Delta_0$ (the $4$ blue segments in Figure \ref{fig: Kxt value}) by definition. For example, consider the segment joining $(0,0.5)$ and $(0.5,0)$, then it satisfies the equation $x+t=1/2$, and when this equation holds, the equation $x=x/2-t/2+1/4$ also holds. We can check the other three edges of $\Delta_0$ similarly. However, on the segment joining $(0.5,1)$ and $(1,1)$ (the red segment in Figure \ref{fig: Kxt value}), $\frac{\partial K(x,t)}{\partial t}$ is not continuous in $t$-direction; so taking derivative again would produce a delta distribution.
$$ \frac{\partial^2 K(x,t)}{\partial t^2} = \begin{cases} 
-1/2 & \text{in } \Delta_0 \\ 
(1/2-x)\delta_1(t) & 1/2\leq x \leq 1,t=1 \\ 
0 & \text{otherwise. } 
\end{cases} $$
These values are recorded in Figure \ref{fig: Kxt value}.

\subsection{Proof of Gessel-Monsky's result}
\begin{notation}
We set $\phi_n=h_{\sum_{0 \leq i \leq n}x_i^2}(x)$. Let $\alpha$ be a small real number, and let $\Phi(\alpha,x)=\sum_{i \geq 0}\alpha^i\phi_i(x)$. Set $K(x,t)=D_\infty(x,t,1/2).$
\end{notation}
\begin{lemma}For $n \geq 0$,
$$\phi_{n+1}(x)=\int_0^{1^+}2K(x,t)d(-\phi'_n(t)).$$    
\end{lemma}
\begin{proof}
This is the special case of the integral formula. We have
$$\phi_{n+1}(x)=h_{\sum_{0 \leq i \leq n+1}x_i^2}(x)=\int_0^{1^+}\int_0^{1^+}D_\infty(x,t,t_1)d(-h'_{\sum_{0 \leq i \leq n}x_i^2}(t))d(-h'_{x_{n+1}^2}(t)).$$
And by Proposition \ref{6.2 h-function of pure power}, $-h''_{x_{n+1}^2}(t_1)=2\delta_{1/2}(t_1)-2\delta_0(t_1)$. Also, $D_\infty(x,t,0)=0$, thus
$$\phi_{n+1}(x)=\int_0^{1^+}2D_\infty(x,t,1/2)d(-h'_{\sum_{0 \leq i \leq n}x_i^2}(t))=\int_0^{1^+}2K(x,t)d(-\phi'_n(t)).$$
\end{proof}
\begin{theorem}\label{8 integral equation of phi d=2}
For small $\alpha$, $\Phi(\alpha,x)$ satisfies the following integral equation
$$\Phi(\alpha,x)-\int_0^{1^+}2\alpha K(x,t)d_t((-\frac{\partial}{\partial t})\Phi(\alpha,t))=\phi_0(x).$$
\end{theorem}
\begin{proof}
We see $h_{\sum_{0 \leq i \leq n}x_i^2}(x)$ and $h'_{\sum_{0 \leq i \leq n}x_i^2,\pm}(x)$ are uniformly bounded. Thus for $\alpha$ sufficiently small,
$$\sum_{0 \leq i \leq m}\alpha^i\phi_i(x) \to \Phi(\alpha,x),\sum_{0 \leq i \leq m}\alpha^i\phi'_{i,\pm}(x) \to \frac{\partial}{\partial x^{\pm}}\Phi(\alpha,x)$$
are both uniformly bounded and uniformly convergent. For fixed $m \in \mathbb{N}$ we have
$$\sum_{0 \leq i \leq m}\alpha^i\phi_i(x)-\int_0^{1^+}2\alpha K(x,t)d_t((-\frac{\partial}{\partial t})\sum_{0 \leq i \leq m}\alpha^i\phi_i(t))=\phi_0(x)-\alpha^{m+1}\phi_{m+1}(x).$$
Take $\alpha<1$ and let $m \to \infty$, we get the equality.
\end{proof}
\begin{theorem}\label{8 solution d=2}
The solution to the integral equation in Theorem \ref{8 integral equation of phi d=2} is 
$$\begin{cases}
0 \leq x \leq 1/2 & \Phi(\alpha,x)=\frac{1}{1-\alpha}x-\frac{1}{2\alpha}+\frac{1}{2\alpha}\cos(2\alpha x)+\frac{\tan\alpha+\sec\alpha}{2\alpha}\sin(2\alpha x)\\
1/2 \leq x \leq 1 &\Phi(\alpha,x)=\frac{1}{1-\alpha}(x-\frac{1}{2})+\frac{2\alpha-1}{2\alpha(1-\alpha)}\\
&-\frac{1}{2\alpha}\sin(2\alpha (x-\frac{1}{2}))+\frac{\tan\alpha+\sec\alpha}{2\alpha}\cos(2\alpha (x-\frac{1}{2})).
\end{cases}$$
\end{theorem}
\begin{proof}
We will solve this integral equation in the following steps.

\textbf{Step 1} We check the following boundary conditions: $\phi_i(x)=1, x \geq 1$, $\phi_i(x)=0, x \leq 0$, $\phi'_{i,+}(1)=0$, $\phi'_{i,-}(0)=0$. Thus $\Phi(\alpha,x)=1/(1-\alpha), x \geq 1$, $\Phi(\alpha,x)=0, x \leq 0$, $\frac{\partial}{\partial x^+}\Phi(\alpha,1)=0$, $\frac{\partial}{\partial x^-}\Phi(\alpha,0)=0$. We have $K(x,0)=0$, $K(x,t)=x/2$ is independent of $t$ for $t \geq 1$, so $\frac{\partial}{\partial t^+}K(x,1)=0$.

\textbf{Step 2} We move the derivatives under integration from $\Phi$ to $K$ using integration by parts. This is possible since $K$ is a continuous concave function with bounded partial derivative. We have
\begin{align*}
\int_0^{1^+}K(x,t)d_t((-\frac{\partial}{\partial t})\Phi(\alpha,t))\\
=K(x,t)(-\frac{\partial}{\partial t})\Phi(\alpha,t)|_{0^-}^{1^+}-(\frac{\partial}{\partial t}K(x,t)(-\Phi(\alpha,t)))|_{0^-}^{1^+}+\int_0^{1^+}\Phi(\alpha,t)d_t(-\frac{\partial}{\partial t})K(x,t))\\
=K(x,t)(-\frac{\partial}{\partial t})\Phi(\alpha,t)|_{0^-}^{1^+}-(\frac{\partial}{\partial t}K(x,t)(-\Phi(\alpha,t)))|_{0^-}^{1^+}+\int_0^{1^+}\Phi(\alpha,t)(-\frac{\partial^2}{\partial t^2})K(x,t))dt.
\end{align*}
Since $K(x,0)=0,\frac{\partial}{\partial x^+}\Phi(\alpha,1)=0,\frac{\partial}{\partial t^+}K(x,1)=0,\frac{\partial}{\partial t^-}K(x,0)=0$, so all the boundary conditions vanish, and
$$\int_0^{1^+}K(x,t)d_t((-\frac{\partial}{\partial t})\Phi(\alpha,t))=\int_0^{1^+}\Phi(\alpha,t)((-\frac{\partial^2}{\partial t^2})K(x,t))dt.$$
Here $\frac{\partial^2}{\partial t^2}K(x,t)$ is the second partial derivative of $K(x,t)$ in the distribution sense, which is equal to the pointwise derivative plus some $\delta$-distribution since $K(x,t)$ is a continuous piecewise polynomial. So we can rewrite the integral equation as
$$\Phi(\alpha,x)-\int_0^{1^+}2\alpha\Phi(\alpha,t)((-\frac{\partial^2}{\partial t^2})K(x,t))dt=\phi_0(x).$$

\textbf{Step 3} We plug in the value of $\frac{\partial^2 K(x,t)}{\partial t^2}$ computed in Section \ref{subsection 4.6} into the integral equation above. We have
$$ \frac{\partial^2 K(x,t)}{\partial t^2} = \begin{cases} 
-1/2 & \text{in } \Delta_0 \\ 
(1/2-x)\delta_1(t) & 1/2\leq x \leq 1,t=1 \\ 
0 & \text{otherwise.} 
\end{cases} $$
The resulting equation is 

$$
\begin{cases} 
\text{If } 0 \leq x \leq \frac{1}{2}, & \Phi(\alpha,x) -\alpha \int_{\frac{1}{2}-x}^{\frac{1}{2}+x} \Phi(\alpha,t) \, dt = 2x \\ 
\text{If } \frac{1}{2} \leq x \leq 1, & \Phi(\alpha,x) -\alpha \int_{x-\frac{1}{2}}^{\frac{3}{2}-x} \Phi(\alpha,t) \, dt + \left(\frac{1}{2} - x\right) 2\alpha\Phi(\alpha,1)  = 1.
\end{cases}
$$
The boundary condition $\Phi(\alpha,1)=1/(1-\alpha)$ is known, so we rewrite the above equation as
$$
\begin{cases} 
\text{If } 0 \leq x \leq \frac{1}{2}, & \Phi(\alpha,x) -\alpha \int_{\frac{1}{2}-x}^{\frac{1}{2}+x} \Phi(\alpha,t) \, dt = 2x \\ 
\text{If } \frac{1}{2} \leq x \leq 1, & \Phi(\alpha,x) -\alpha \int_{x-\frac{1}{2}}^{\frac{3}{2}-x} \Phi(\alpha,t) \, dt + \left(\frac{1}{2} - x\right) \frac{2\alpha}{1-\alpha}  = 1.
\end{cases}
$$

After differentiating, it yields

$$
\begin{cases}
0 \leq x \leq \frac{1}{2}, &\frac{\partial\Phi(\alpha,x) }{\partial x}=2+\alpha\Phi(\alpha,\frac{1}{2}+x)+\alpha\Phi(\alpha,\frac{1}{2}-x)\\
\frac{1}{2} \leq x \leq 1, &\frac{\partial\Phi(\alpha,x) }{\partial x}=\frac{2\alpha}{1-\alpha}-\alpha\Phi(\alpha,\frac{3}{2}-x)-\alpha\Phi(\alpha,x-\frac{1}{2}).
\end{cases}
$$

\textbf{Step 4} Now we fix $\alpha$ and solve this equation with parameter $\alpha$. We make substitutions $F(x)=\Phi(\alpha,x)|_{[0,1/2]}:[0,1/2] \to \mathbb{R}$, $G(x)=\Phi(\alpha,x+1/2)|_{[0,1/2]}:[0,1/2] \to \mathbb{R}$. The above two equations become
$$
\begin{cases}
F'(x)=2+\alpha G(x)+\alpha F(1/2-x)\\
G'(x)=\frac{2\alpha}{1-\alpha}-\alpha G(1/2-x)-\alpha F(x).
\end{cases}
$$
Replacing $x$ by $1/2-x$ yields
$$
\begin{cases}
F'(1/2-x)=2+\alpha G(1/2-x)+\alpha F(x)\\
G'(1/2-x)=\frac{2\alpha}{1-\alpha}-\alpha G(x)-\alpha F(1/2-x).
\end{cases}
$$
Set $F_2(x)=F(1/2-x)$, $G_2(x)=G(1/2-x)$. Then $F'_2(x)=-F'(1/2-x)$, $G'_2(x)=-G'(1/2-x)$. The above two systems of equations become one single system of equations in terms of just $x$:

$$
\begin{cases}
F'(x)=2+\alpha G(x)+\alpha F_2(x)\\
G'(x)=\frac{2\alpha}{1-\alpha}-\alpha G_2(x)-\alpha F(x)\\
F'_2(x)=-2-\alpha G_2(x)-\alpha F(x)\\
G'_2(x)=-\frac{2\alpha}{1-\alpha}+\alpha G(x)+\alpha F_2(x).
\end{cases}
$$
We can already solve this differential equation because it is of the form $\frac{d\mathbf{f}(x)}{dx}=\mathbf{A}\mathbf{f}(x)+b(x)$ where $\mathbf{A}$ is a $4*4$ matrix. However, we make one more observation that $\mathbf{A}$ does not have full rank, which further simplifies the equation. From the equation we see
$$G'(x)-F'_2(x)=2+2\alpha/(1-\alpha)=2/(1-\alpha)$$
and using the boundary condition $G(0)=F(1/2)=F_2(0)$, we get
$$G(x)-F_2(x)=2x/(1-\alpha).$$
Replacing $x$ by $1/2-x$, we get
$$G_2(x)-F(x)=(1/2-x)\cdot 2/(1-\alpha).$$
Now eliminate $F_2$ and $G_2$ from the system of equations on $F,G,F_2,G_2$ using $F,G$, which leads to the following system of equations:
$$ \begin{cases} 
F'(x)=2+2\alpha G(x)-\frac{2\alpha}{1-\alpha}x & 0 \leq x \leq 1/2 \\ 
G'(x)=\frac{2\alpha}{1-\alpha}-2\alpha F(x)-\frac{2\alpha}{1-\alpha}(\frac{1}{2}-x) & 0 \leq x \leq 1/2.
\end{cases} $$.

The general solution to this differential equation is
$$\begin{cases}
F(x)=\frac{1}{1-\alpha}x-\frac{1}{2\alpha}+A\cos(2\alpha x)+B\sin(2\alpha x)\\
G(x)=\frac{1}{1-\alpha}x+\frac{2\alpha-1}{2\alpha(1-\alpha)}-A\sin(2\alpha x)+B\cos(2\alpha x).
\end{cases}$$
$F$ and $G$ satisfy the boundary conditions $F(0)=0$ and $G(1/2)=1/(1-\alpha)$. From this we can solve $A,B$; we get $A=\frac{1}{2\alpha}$ and $B=\frac{\tan\alpha+\sec\alpha}{2\alpha}$. Thus,
$$\begin{cases}
F(x)=\frac{1}{1-\alpha}x-\frac{1}{2\alpha}+\frac{1}{2\alpha}\cos(2\alpha x)+\frac{\tan\alpha+\sec\alpha}{2\alpha}\sin(2\alpha x)\\
G(x)=\frac{1}{1-\alpha}x+\frac{2\alpha-1}{2\alpha(1-\alpha)}-\frac{1}{2\alpha}\sin(2\alpha x)+\frac{\tan\alpha+\sec\alpha}{2\alpha}\cos(2\alpha x).
\end{cases}$$
Thus the final expression of $\Phi(\alpha,x)$ is:
$$\begin{cases}
0 \leq x \leq 1/2 & \Phi(\alpha,x)=\frac{1}{1-\alpha}x-\frac{1}{2\alpha}+\frac{1}{2\alpha}\cos(2\alpha x)+\frac{\tan\alpha+\sec\alpha}{2\alpha}\sin(2\alpha x)\\
1/2 \leq x \leq 1 &\Phi(\alpha,x)=\frac{1}{1-\alpha}(x-\frac{1}{2})+\frac{2\alpha-1}{2\alpha(1-\alpha)}\\
&-\frac{1}{2\alpha}\sin(2\alpha (x-\frac{1}{2}))+\frac{\tan\alpha+\sec\alpha}{2\alpha}\cos(2\alpha (x-\frac{1}{2})).
\end{cases}$$
\end{proof}
Now Gessel-Monsky's result is a consequence of Remark 6.9, plus the fact
$$\frac{\partial \Phi(\alpha,0)}{\partial x^+}=\frac{1}{1-\alpha}+\tan\alpha+\sec\alpha.$$
Also, the corresponding result on limit $F$-signature comes from the fact
$$\frac{\partial \Phi(\alpha,1)}{\partial x^-}=\frac{1}{1-\alpha}-\cos(\alpha)-(\tan\alpha+\sec\alpha)\sin\alpha=\frac{1}{1-\alpha}-\tan\alpha-\sec\alpha.$$

\subsection{Gessel-Monsky type result in dimension 3}
The above method applies to Fermat hypersurface of any degree $d$. However, the computation becomes harder for larger $d$. We show the computation for $d=3$ here. 

Set $K(x,t)=D_\infty(x,t,1/3)$. Let $\Delta_i=B_i\cap \{t_3=1/3\}$, $\Delta_0=T_0\cap\{t_3=1/3\}$ and $\Delta_5=[0,1]\times[1,\infty)$. The following computations are straightforward:
\begin{figure}
    \centering
    \begin{tikzpicture}[scale=1.5, thick]
    \draw[->] (-3.0,0) -- (-1.0,0) node[right]{$x$};
    \draw[->] (-2.5,-0.5) -- (-2.5,1.5) node[above]{$t$};

    \draw (-2.5,0) rectangle (-1.5,1);

    \draw (-2.5,1/3) -- (-1.833,1) -- (-1.5,2/3) -- (-2.167,0) -- cycle;

    \node[anchor=north east] at (-2.25,0.25) {\tiny $xt$};
    \node[anchor=north west] at (-1.9,0.35) {\scriptsize $t/3$};
    \node[anchor=south west] at (-2.5,0.7) {\scriptsize $x/3$};
    \node[anchor=south east] at (-0.48,0.75) {\tiny $xt-2x/3$};
    \node[anchor=south east] at (-0.5,0.63) {\tiny $-2t/3+2/3$};
    \node at (-1.5,0.9) {$\longleftarrow$};
    \node at (-3.3,0.5) {\tiny $xt/2+x/6+t/6$};
    \node at (-3.5,0.3) {\tiny $-x^2/4-t^2/4-1/36$};
    \node at (-2.4,0.4) {$\longrightarrow$};
    \node at (-2.0,1.2) {\small $x/3$};

    \draw (-1.5,1) -- (-1.5,1.5);
    \draw (-1.5,1) -- (-1.0,1);

    \filldraw (-2.5,0) circle (1pt) node[below left]{};
    \filldraw (-1.5,0) circle (1pt) node[below right]{$1$};
    \filldraw (-1.5,1) circle (1pt) node[above right]{};
    \filldraw (-2.5,1) circle (1pt) node[above left]{$1$};
    \filldraw (-2.167,0) circle (1pt);
    \filldraw (-1.5,2/3) circle (1pt);
    \filldraw (-1.833,1) circle (1pt);
    \filldraw (-2.5,1/3) circle (1pt);
     \draw[->] (-0.5,0) -- (1.5,0) node[right]{$x$};
    \draw[->] (0,-0.5) -- (0,1.5) node[above]{$t$};

    \draw (0,0) rectangle (1,1);

    \draw[blue, thick] (0,1/3) -- (2/3,1) -- (1,2/3) -- (1/3,0) -- cycle;

    \node[anchor=north east] at (0.25,0.25) {\small $x$};
    \node[anchor=north west] at (0.55,0.3) {\scriptsize $1/3$};
    \node[anchor=south west] at (0,0.7) {\scriptsize $0$};
    \node[anchor=south east] at (1.85,0.63) {\tiny $x-2/3$};
    \node at (1.05,0.85) {$\longleftarrow$};
    \node at (-0.7,0.4) {\tiny $x/2-t/2+1/6$};
    \node at (0.1,0.4) {$\longrightarrow$};
    \node at (0.5,1.2) {\small $0$};

    \draw (1,1) -- (1,1.5);
    \draw (1,1) -- (1.5,1);

    \filldraw (0,0) circle (1pt) node[below left]{};
    \filldraw (1,0) circle (1pt) node[below right]{};
    \filldraw (1,1) circle (1pt) node[above right]{};
    \filldraw (0,1) circle (1pt) node[above left]{};
    \filldraw (1/3,0) circle (1pt);
    \filldraw (1,2/3) circle (1pt);
    \filldraw (2/3,1) circle (1pt);
    \filldraw (0,1/3) circle (1pt);

    \draw[red,thick] (2/3,1) -- (1,1) node[above]{};
    \draw[->] (1.9,0) -- (3.5,0) node[right]{$x$};
    \draw[->] (2,-0.5) -- (2,1.5) node[above]{$t$};

    \draw (2,0) rectangle (3,1);

    \draw[blue, thick] (2,1/3) -- (2.667,1) -- (3,2/3) -- (2.333,0) -- cycle;

    \node[anchor=north east] at (2.25,0.25) {\small $0$};
    \node[anchor=north west] at (2.7,0.3) {\small $0$};
    \node[anchor=south west] at (2,0.7) {\small $0$};
    \node[anchor=south east] at (3.0,0.8) {\tiny $0$};
    \node at (2.5,0.5) {\small $-1/2$};
    \node at (2.5,1.2) {\small $0$};
    \node at (2.9,1.6) {\tiny $(2/3-x)\delta_1(t)$};
    \draw[red,->] (2.8,1.5) -- (2.8,1.1);

    \draw (3,1) -- (3,1.5);
    \draw (3,1) -- (3.5,1);

    \filldraw (2,0) circle (1pt) node[below left]{};
    \filldraw (3,0) circle (1pt) node[below right]{};
    \filldraw (3,1) circle (1pt) node[above right]{};
    \filldraw (2,1) circle (1pt) node[above left]{};
    \filldraw (2.333,0) circle (1pt);
    \filldraw (3,2/3) circle (1pt);
    \filldraw (2.667,1) circle (1pt);
    \filldraw (2,1/3) circle (1pt);

    \draw[red,thick] (2.667,1) -- (3,1) node[above]{};
    \end{tikzpicture}
    \caption{Evaluating $K,\frac{\partial K(x,t)}{\partial t},\frac{\partial^2 K(x,t)}{\partial t^2}$}
    \label{fig: Kxt value 1/3}
\end{figure}

$$ K(x, t) = D_\infty\left(x, t, \frac{1}{3}\right) = \begin{cases} 
x t & \text{in } \Delta_1 \\ 
\frac{1}{3} t & \text{in } \Delta_2 \\ 
\frac{1}{3} x & \text{in } \Delta_3\cup\Delta_5 \\ 
x t - \frac{2}{3} x - \frac{2}{3} t + \frac{2}{3} & \text{in } \Delta_4\\
\frac{xt}{2}+\frac{x}{6}+\frac{t}{6}-\frac{x^2}{4}-\frac{t^2}{4}-\frac{1}{36} & \text{in }\Delta_0,
\end{cases} $$
$$ \frac{\partial K(x,t)}{\partial t} = \begin{cases} 
x & \text{in } \Delta_1 \\ 
\frac{1}{3} & \text{in } \Delta_2 \\ 
0 & \text{in } \Delta_3\cup\Delta_5 \\ 
x- \frac{2}{3}& \text{in } \Delta_4\\
\frac{x}{2}+\frac{1}{6}-\frac{t}{2}& \text{in }\Delta_0.
\end{cases} $$
The continuity of $\frac{\partial K(x,t)}{\partial t}$ on $\partial \Delta_0$ still holds, and there is a nonzero jump on the segement joining $(2/3,1)$ and $(1,1)$ and taking derivative again would produce a nonzero delta distribution.
$$ \frac{\partial^2 K(x,t)}{\partial t^2} = \begin{cases} 
-1/2 & \text{in } \Delta_0 \\ 
(2/3-x)\delta_1(t) & 2/3\leq x \leq 1,t=1 \\ 
0 & \text{otherwise. } 
\end{cases} $$
The above computations are shown in Figure \ref{fig: Kxt value 1/3}.

We set $\phi_n=h_{\sum_{0 \leq i \leq n}x_i^3}(x)$. Let $\alpha$ be a small real number, and let $\Phi(\alpha,x)=\sum_{i \geq 0}\alpha^i\phi_i(x)$. Then $\Phi(\alpha,x)$ satisfies the following integral equation
$$\Phi(\alpha,x)-\int_0^{1^+}3\alpha\Phi(\alpha,t)((-\frac{\partial^2}{\partial t^2})K(x,t))dt=\phi_0(x).$$
Both $((-\frac{\partial^2}{\partial t^2})K(x,t))$ and $\phi_0(x)$ are piecewise on $[0,1/3]$, $[1/3,2/3]$, $[2/3,1]$. We list the integral equations also piecewisely. We apply the boundary condition $\int\Phi(\alpha,t)\delta_1(t)dt=\Phi(\alpha,1)=1/(1-\alpha)$.

$$
\begin{cases} 
\text{If } 0 \leq x \leq \frac{1}{3}, & \Phi(\alpha,x) + 3\alpha \int_{\frac{1}{3}-x}^{\frac{1}{3}+x} \left(-\frac{1}{2}\right) \Phi(\alpha,t) \, dt = 3x \\ 
\text{If } \frac{1}{3} \leq x \leq \frac{2}{3}, & \Phi(\alpha,x) + 3\alpha \int_{x-\frac{1}{3}}^{x+\frac{1}{3}} \left(-\frac{1}{2}\right) \Phi(\alpha,t) \, dt = 1 \\
\text{If } \frac{2}{3} \leq x \leq 1, & \Phi(\alpha,x) + 3\alpha \left( \int_{x-\frac{1}{3}}^{\frac{5}{3}-x} \left(-\frac{1}{2}\right) \Phi(\alpha,t) \, dt + \left(\frac{2}{3} - x\right) \frac{1}{1-\alpha} \right) = 1.
\end{cases}
$$
We differentiate to get
$$
\begin{cases}
0 \leq x \leq \frac{1}{3}, &\frac{\partial}{\partial x}\Phi(\alpha,x) - \frac{3\alpha}{2} \Phi(\alpha,\frac{1}{3}+x) - \frac{3\alpha}{2} \Phi(\alpha,\frac{1}{3}-x) = 3\\
\frac{1}{3} \leq x \leq \frac{2}{3}, &\frac{\partial}{\partial x}\Phi(\alpha,x) - \frac{3\alpha}{2} \Phi(\alpha,\frac{1}{3}+x) + \frac{3\alpha}{2} \Phi(\alpha,x-\frac{1}{3}) = 0\\
\frac{2}{3} \leq x \leq 1, &\frac{\partial}{\partial x}\Phi(\alpha,x) + \frac{3\alpha}{2} \Phi(\alpha,\frac{5}{3}-x) + \frac{3\alpha}{2} \Phi(\alpha,x-\frac{1}{3}) - 3\alpha \frac{1}{1-\alpha} = 0.
\end{cases}
$$
Set new functions $F_1=\Phi(\alpha,\cdot)|_{[0,1/3]}$, $F_2(x)=\Phi(\alpha,\cdot)|_{[1/3,2/3]}(x+1/3)$,$F_3(x)=\Phi(\alpha,\cdot)|_{[2/3,1]}(x+2/3)$, then we get

$$
\begin{cases}
F'_1(x) - \frac{3\alpha}{2} F_2\left(x\right) - \frac{3\alpha}{2} F_1\left(\frac{1}{3}-x\right) = 3 \\
F'_2(x) - \frac{3\alpha}{2} F_3\left(x\right) + \frac{3\alpha}{2} F_1\left(x\right) = 0 \\
F'_3(x) + \frac{3\alpha}{2} F_3\left(\frac{1}{3}-x\right) + \frac{3\alpha}{2} F_2\left(x\right) - \frac{3\alpha}{1-\alpha} = 0.
\end{cases}
$$

Setting $G_i=F_i(1/3-x)$, we get

$$
\begin{cases}
G'_1(x) + \frac{3\alpha}{2} F_2\left(\frac{1}{3}-x\right) + \frac{3\alpha}{2} F_1\left(x\right) = -3 \\
G'_2(x) + \frac{3\alpha}{2} F_3\left(\frac{1}{3}-x\right) - \frac{3\alpha}{2} F_1\left(\frac{1}{3}-x\right) = 0 \\
G'_3(x) - \frac{3\alpha}{2} F_3\left(x\right) - \frac{3\alpha}{2} F_2\left(\frac{1}{3}-x\right) + \frac{3\alpha}{1-\alpha} = 0.
\end{cases}
$$

Rewrite the above equation as

\begin{equation}\label{equation Fermat degree 3}
\begin{cases}
F'_1(x) - \frac{3\alpha}{2} F_2\left(x\right) - \frac{3\alpha}{2} G_1\left(x\right) = 3 \\
F'_2(x) - \frac{3\alpha}{2} F_3\left(x\right) + \frac{3\alpha}{2} F_1\left(x\right) = 0 \\
F'_3(x) + \frac{3\alpha}{2} G_3\left(x\right) + \frac{3\alpha}{2} F_2\left(x\right) - \frac{3\alpha}{1-\alpha} = 0\\
G'_1(x) + \frac{3\alpha}{2} G_2\left(x\right) + \frac{3\alpha}{2} F_1\left(x\right) = -3 \\
G'_2(x) + \frac{3\alpha}{2} G_3\left(x\right) - \frac{3\alpha}{2} G_1\left(x\right) = 0 \\
G'_3(x) - \frac{3\alpha}{2} F_3\left(x\right) - \frac{3\alpha}{2} G_2\left(x\right) + \frac{3\alpha}{1-\alpha} = 0.
\end{cases}    
\end{equation}

Here $F_i,G_i,1\leq i \leq 3$ are real-valued functions on $[0,1/3]$ satisfying the following boundary conditions:

$$
\begin{cases}
F_1(0)=G_1(1/3)=0\\
F_1(1/3)=F_2(0)=G_1(0)=G_2(1/3)\\
F_2(1/3)=F_3(0)=G_2(0)=G_3(1/3)\\
F_3(1/3)=G_3(0)=1/(1-\alpha).
\end{cases}
$$

We put the process of solving Equation \ref{equation Fermat degree 3} in \nameref{appendix}. The general solution to this equation is

\begin{align*}
F_1(x) &= \tfrac{1}{1-\alpha}x + \tfrac{D_1}{3} - \tfrac{1}{3\alpha} + \tfrac{1}{3(1-\alpha)} - \tfrac{1}{3\alpha(1-\alpha)}\\
&+ \tfrac{1}{2}\left(A + \tfrac{D}{\sqrt{3}}\right)\cos\left(\tfrac{3\sqrt{3}\alpha}{2}x\right) + \tfrac{1}{2}\left(B + \tfrac{C}{\sqrt{3}}\right)\sin\left(\tfrac{3\sqrt{3}\alpha}{2}x\right), \\
F_2(x) &= \tfrac{1}{1-\alpha}x - \tfrac{D_2}{3} - \tfrac{2}{3\alpha} + \tfrac{2}{3(1-\alpha)}\\
&+\left(\tfrac{B}{\sqrt{3}}\right)\cos\left(\tfrac{3\sqrt{3}\alpha}{2}x\right) + \left(- \tfrac{A}{\sqrt{3}}\right)\sin\left(\tfrac{3\sqrt{3}\alpha}{2}x\right), \\
F_3(x)  &= \tfrac{1}{1-\alpha}x + \tfrac{D_1}{3} - \tfrac{1}{3\alpha} + \tfrac{1}{3(1-\alpha)} + \tfrac{1}{3\alpha(1-\alpha)}\\
&+ \tfrac{1}{2}\left(\tfrac{D}{\sqrt{3}} - A\right)\cos\left(\tfrac{3\sqrt{3}\alpha}{2}x\right) + \tfrac{1}{2}\left(\tfrac{C}{\sqrt{3}} - B\right)\sin\left(\tfrac{3\sqrt{3}\alpha}{2}x\right), \\
G_1(x)  &= -\tfrac{1}{1-\alpha}x + \tfrac{D_2}{3} - \tfrac{1}{3\alpha} + \tfrac{1}{3(1-\alpha)} - \tfrac{1}{3\alpha(1-\alpha)}\\
&+ \tfrac{1}{2}\left(C + \tfrac{B}{\sqrt{3}}\right)\cos\left(\tfrac{3\sqrt{3}\alpha}{2}x\right) - \tfrac{1}{2}\left(D + \tfrac{A}{\sqrt{3}}\right)\sin\left(\tfrac{3\sqrt{3}\alpha}{2}x\right), \\
G_2(x) &= -\tfrac{1}{1-\alpha}x - \tfrac{D_1}{3} - \tfrac{2}{3\alpha} + \tfrac{2}{3(1-\alpha)}\\
&+\left(\tfrac{D}{\sqrt{3}}\right)\cos\left(\tfrac{3\sqrt{3}\alpha}{2}x\right) + \left(\tfrac{C}{\sqrt{3}}\right)\sin\left(\tfrac{3\sqrt{3}\alpha}{2}x\right), \\
G_3(x)  &= -\tfrac{1}{1-\alpha}x + \tfrac{D_2}{3} - \tfrac{1}{3\alpha} + \tfrac{1}{3(1-\alpha)} + \tfrac{1}{3\alpha(1-\alpha)}\\
&+ \tfrac{1}{2}\left(\tfrac{B}{\sqrt{3}} - C\right)\cos\left(\tfrac{3\sqrt{3}\alpha}{2}x\right) + \tfrac{1}{2}\left(-\tfrac{A}{\sqrt{3}} + D\right)\sin\left(\tfrac{3\sqrt{3}\alpha}{2}x\right).
\end{align*}
Here $A,B,C,D,D_1,D_2$ are constants. The special solution subject to the boundary condition is given by
$$
\begin{cases}
A=\frac{6\cos(\sqrt{3}\alpha)-2\sqrt{3}\sin(\frac{\sqrt{3}\alpha}{2})}{3(\alpha+2\alpha\cos(\sqrt{3}\alpha))}\\  
B=\frac{2 \left( \sqrt{3} \cos \left( \frac{\sqrt{3} \alpha}{2} \right) + \sin \left( \sqrt{3} \alpha \right) \right)}{\alpha + 2 \alpha \cos \left( \sqrt{3} \alpha \right)}\\
C=\frac{2 \left( \sqrt{3} \cos \left( \frac{\sqrt{3} \alpha}{2} \right) + \sin \left( \sqrt{3} \alpha \right) \right)}{\sqrt{3}(\alpha + 2 \alpha \cos \left( \sqrt{3} \alpha \right))}\\
D=\frac{2 \left( 2 + \cos \left( \sqrt{3} \alpha \right) + \sqrt{3} \sin \left( \frac{\sqrt{3} \alpha}{2} \right) \right)}{\sqrt{3} \left( \alpha + 2 \alpha \cos \left( \sqrt{3} \alpha \right) \right)}\\
D_1=0\\
D_2=\frac{1}{1-\alpha}.
\end{cases}
$$
\begin{remark}
We remark that to solve this equation explicitly, a careful choice of the boundary condition is essential. In general, we encounter the solution of $Mx=b$ for a $6*6$ matrix $M$, and the determinant of $M$ can be rather hard to expand. Only a wise choice of the boundary condition allows us to solve $A,B,C,D,D_1,D_2$ explicitly. Please see the Appendix for the choice of boundary conditions.   
\end{remark}
\begin{theorem}\label{8 solution d=3}
The above solution of $F_i,G_i,1 \leq i \leq 3$ gives $\Phi(\alpha,x)$ on $[0,1]$.
\end{theorem}
\begin{corollary}\label{8 corollary d=3}
We have
\begin{align*}
F_1(x) &=\tfrac{1}{1-\alpha}x-\tfrac{2}{3\alpha}+\tfrac{2}{3\alpha}\cos\left(\tfrac{3\sqrt{3}\alpha}{2}x\right)+\tfrac{4}{3}(\tfrac{\sqrt{3} \cos \left( \tfrac{\sqrt{3} \alpha}{2} \right) + \sin \left( \sqrt{3} \alpha \right) }{\alpha + 2 \alpha \cos \left( \sqrt{3} \alpha \right)})\sin\left(\tfrac{3\sqrt{3}\alpha}{2}x\right) \\
F'_1(0)&=\tfrac{1}{1-\alpha}+2\sqrt{3}\cdot(\tfrac{\sqrt{3} \cos \left( \tfrac{\sqrt{3} \alpha}{2} \right) + \sin \left( \sqrt{3} \alpha \right) }{1 + 2 \cos \left( \sqrt{3} \alpha \right)})\\
G_3(x)&=-\tfrac{1}{1-\alpha}x+\tfrac{1}{1-\alpha}+\tfrac{2(2\sin \left( \tfrac{\sqrt{3} \alpha}{2} \right)+\sqrt{3})}{3(\alpha + 2 \alpha \cos \left( \sqrt{3} \alpha \right))}\sin\left(\tfrac{3\sqrt{3}\alpha}{2}x\right)\\
G'_3(0)&=-\tfrac{1}{1-\alpha}+\tfrac{\sqrt{3}(2\sin \left( \tfrac{\sqrt{3} \alpha}{2} \right)+\sqrt{3})}{1 + 2\cos \left( \sqrt{3} \alpha \right)}.
\end{align*}
In particular, setting
$\lim_{p \to \infty}e_{HK}(S_{p,n,3})=1+c_n$ and $\lim_{p \to \infty}s(S_{p,n,3})=c'_n,$
then 
\begin{align*}
\sum_{n \geq 0}c_n\alpha^n=2\sqrt{3}\cdot(\frac{\sqrt{3} \cos \left( \frac{\sqrt{3} \alpha}{2} \right) + \sin \left( \sqrt{3} \alpha \right) }{1 + 2 \cos \left( \sqrt{3} \alpha \right)}),\\
\sum_{n \geq 0}c'_n\alpha^n=-\frac{1}{1-\alpha}+\frac{\sqrt{3}(2\sin \left( \frac{\sqrt{3} \alpha}{2} \right)+\sqrt{3})}{1 + 2\cos \left( \sqrt{3} \alpha \right)}.    
\end{align*}
\end{corollary}
Inspired by the results in degree $d=2,3$, we propose the following conjecture.
\begin{conjecture}\label{8 conjecture for general d}
For $d \geq 2$, set $\lim_{p \to \infty}e_{HK}(S_{p,n,d})=c_n$ and $\lim_{p \to \infty}s(S_{p,n,d})=c'_n,$
Then $\sum_{n \geq 0}c_n\alpha^n$ and $\sum_{n \geq 0}c'_n\alpha^n$ are rational functions in $\alpha,\cos(\lambda_i\alpha),\sin(\lambda_i\alpha)$ where the coefficients of the rational functions and $\lambda_i$'s lie in $\mathbb{Q}(\xi_{2d})$.
\end{conjecture}

\section*{Acknowledgement}

Part of this research was performed when the author was in residence at the Simons Laufer Mathematical Sciences Institute (formerly MSRI) in Berkeley, California, during the Spring 2024 semester, which is supported by the National Science Foundation Grant No. DMS-1928930 and by the Alfred P. Sloan Foundation under grant G-2021-16778. 

The author would like to thank Linquan Ma for reading an early draft. The author would like to thank Joel Castillo-Rey for helpful suggestions, especially for the formulation of Watanabe-Yoshida's conjecture in characteristic $2$ and its references. The author would also like to thank Kevin Tucker and Srikanth Iyengar for useful discussions. 

\bibliographystyle{plain}
\bibliography{refQNoe2}

\section*{Appendix}\label{appendix}
Here we simplify Equation \ref{equation Fermat degree 3} and find its solution. It can be rewritten as
$$
\begin{cases}
F'_1(x) - \frac{3\alpha}{2} F_2\left(x\right) - \frac{3\alpha}{2} G_1\left(x\right) = 3 \\
F'_2(x) - \frac{3\alpha}{2} F_3\left(x\right) + \frac{3\alpha}{2} F_1\left(x\right) = 0 \\
F'_3(x) + \frac{3\alpha}{2} G_3\left(x\right) + \frac{3\alpha}{2} F_2\left(x\right) - \frac{3\alpha}{1-\alpha} = 0\\
G'_1(x) + \frac{3\alpha}{2} G_2\left(x\right) + \frac{3\alpha}{2} F_1\left(x\right) = -3 \\
G'_2(x) + \frac{3\alpha}{2} G_3\left(x\right) - \frac{3\alpha}{2} G_1\left(x\right) = 0 \\
G'_3(x) - \frac{3\alpha}{2} F_3\left(x\right) - \frac{3\alpha}{2} G_2\left(x\right) + \frac{3\alpha}{1-\alpha} = 0.
\end{cases}
$$

Here $F_i,G_i,1\leq i \leq 3$ are real-valued functions on $[0,1/3]$. From the above equation we deduce
$$\begin{cases}
F'_1(x)+F'_3(x)-G'_2(x)=3/(1-\alpha)\\
G'_1(x)+G'_3(x)-F'_2(x)=-3/(1-\alpha).
\end{cases}$$
Thus there are constants $D_1,D_2$ such that
$$\begin{cases}
F_1(x)+F_3(x)-G_2(x)=3x/(1-\alpha)+D_1\\
G_1(x)+G_3(x)-F_2(x)=-3x/(1-\alpha)+D_2.
\end{cases}$$
In other words,
$$\begin{cases}
G_2(x)=F_1(x)+F_3(x)-3x/(1-\alpha)-D_1\\
F_2(x)=G_1(x)+G_3(x)+3x/(1-\alpha)-D_2.
\end{cases}$$

Eliminate $F_2,G_2$ to get

$$
\begin{cases}
F'_1(x) - \frac{3\alpha}{2} (G_1(x)+G_3(x)+\frac{3}{1-\alpha}x-D_2) - \frac{3\alpha}{2} G_1\left(x\right) = 3 \\
F'_3(x) + \frac{3\alpha}{2} G_3\left(x\right) + \frac{3\alpha}{2} (G_1(x)+G_3(x)+\frac{3}{1-\alpha}x-D_2) - \frac{3\alpha}{1-\alpha} = 0\\
G'_1(x) + \frac{3\alpha}{2} (F_1(x)+F_3(x)-\frac{3}{1-\alpha}x-D_1) + \frac{3\alpha}{2} F_1\left(x\right) = -3 \\
G'_3(x) - \frac{3\alpha}{2} F_3\left(x\right) - \frac{3\alpha}{2} (F_1(x)+F_3(x)-\frac{3}{1-\alpha}x-D_1) + \frac{3\alpha}{1-\alpha} = 0.
\end{cases}
$$

The above equation reduces to

$$
\begin{cases}
F'_1(x)- 3\alpha G_1\left(x\right) - \frac{3\alpha}{2} G_3(x)-\frac{3\alpha}{2}\frac{3}{1-\alpha}x+\frac{3\alpha}{2}D_2  = 3 \\
F'_3(x)+ \frac{3\alpha}{2} G_1(x) + 3\alpha G_3\left(x\right) +\frac{3\alpha}{2}\frac{3}{1-\alpha}x-\frac{3\alpha}{2}D_2 - \frac{3\alpha}{1-\alpha} = 0\\
G'_1(x)+ 3\alpha F_1\left(x\right) + \frac{3\alpha}{2} F_3(x)-\frac{3\alpha}{2}\frac{3}{1-\alpha}x-\frac{3\alpha}{2}D_1  = -3 \\
G'_3(x)- \frac{3\alpha}{2} F_1(x) - 3\alpha F_3\left(x\right) +\frac{3\alpha}{2}\frac{3}{1-\alpha}x+\frac{3\alpha}{2}D_1 + \frac{3\alpha}{1-\alpha}= 0.
\end{cases}
$$

So

$$
\begin{cases}
F'_1(x)-F'_3(x)- \frac{9\alpha}{2} G_1\left(x\right) - \frac{9\alpha}{2} G_3(x)-3\alpha\frac{3}{1-\alpha}x+3\alpha D_2 + \frac{3\alpha}{1-\alpha}  = 3 \\
F'_1(x)+F'_3(x)- \frac{3\alpha}{2} G_1(x) +\frac{3\alpha}{2}G_3\left(x\right) - \frac{3\alpha}{1-\alpha} = 3\\
G'_1(x)-G'_3(x)+ \frac{9\alpha}{2} F_1\left(x\right) + \frac{9\alpha}{2} F_3(x)-3\alpha\frac{3}{1-\alpha}x-3\alpha D_1 - \frac{3\alpha}{1-\alpha}   = -3 \\
G'_1(x)+G'_3(x)+ \frac{3\alpha}{2} F_1(x) - \frac{3\alpha}{2} F_3\left(x\right) + \frac{3\alpha}{1-\alpha}= -3.
\end{cases}
$$

Let $H_1=F_1-F_3,H_2=F_1+F_3,H_3=G_1-G_3,H_4=G_1+G_3$, then

$$
\begin{cases}
H'_1(x)- \frac{9\alpha}{2} H_4\left(x\right)-3\alpha\frac{3}{1-\alpha}x+3\alpha D_2 -3+\frac{3\alpha}{1-\alpha}  = 0 \\
H'_2(x)- \frac{3\alpha}{2} H_3(x) -\frac{3}{1-\alpha} = 0\\
H'_3(x)+ \frac{9\alpha}{2} H_2\left(x\right) -3\alpha\frac{3}{1-\alpha}x-3\alpha D_1 +3-\frac{3\alpha}{1-\alpha} = 0 \\
H'_4(x)+ \frac{3\alpha}{2} H_1(x) + \frac{3}{1-\alpha}= 0.
\end{cases}
$$

It can be separated as two independent systems of equations:

$$
\begin{cases}
H'_1(x)- \frac{9\alpha}{2} H_4\left(x\right)-3\alpha\frac{3}{1-\alpha}x+3\alpha D_2 -3+\frac{3\alpha}{1-\alpha}  = 0 \\
H'_4(x)+ \frac{3\alpha}{2} H_1(x) + \frac{3}{1-\alpha}= 0.
\end{cases}
$$

$$
\begin{cases}
H'_2(x)- \frac{3\alpha}{2} H_3(x) -\frac{3}{1-\alpha} = 0\\
H'_3(x)+ \frac{9\alpha}{2} H_2\left(x\right) -3\alpha\frac{3}{1-\alpha}x-3\alpha D_1 +3-\frac{3\alpha}{1-\alpha} = 0.
\end{cases}
$$

Its special solution of polynomial type is

\begin{align*}
H_1(x) &= -\tfrac{2}{3\alpha(1-\alpha)}, &
H_2(x) &= \tfrac{2}{1-\alpha}x + \tfrac{2}{3}D_1 - \tfrac{2}{3\alpha} + \tfrac{2}{3(1-\alpha)}, \\
H_3(x) &= -\tfrac{2}{3\alpha(1-\alpha)}, &
H_4(x) &= -\tfrac{2}{1-\alpha}x + \tfrac{2}{3}D_2 - \tfrac{2}{3\alpha} + \tfrac{2}{3(1-\alpha)}.
\end{align*}

Its general solution is

\begin{align*}
H_1(x) &= -\tfrac{2}{3\alpha(1-\alpha)}+A\cos(\tfrac{3\sqrt{3}\alpha}{2}x)+B\sin(\tfrac{3\sqrt{3}\alpha}{2}x), \\
H_2(x) &= \tfrac{2}{1-\alpha}x + \tfrac{2}{3}D_1 - \tfrac{2}{3\alpha} + \tfrac{2}{3(1-\alpha)}+\tfrac{C}{\sqrt{3}}\sin(\tfrac{3\sqrt{3}\alpha}{2}x)+\tfrac{D}{\sqrt{3}}\cos(\tfrac{3\sqrt{3}\alpha}{2}x), \\
H_3(x) &= -\tfrac{2}{3\alpha(1-\alpha)}+C\cos(\tfrac{3\sqrt{3}\alpha}{2}x)-D\sin(\tfrac{3\sqrt{3}\alpha}{2}x), \\
H_4(x) &= -\tfrac{2}{1-\alpha}x + \tfrac{2}{3}D_2 - \tfrac{2}{3\alpha} + \tfrac{2}{3(1-\alpha)}-\tfrac{A}{\sqrt{3}}\sin(\tfrac{3\sqrt{3}\alpha}{2}x)+\tfrac{B}{\sqrt{3}}\cos(\tfrac{3\sqrt{3}\alpha}{2}x).
\end{align*}

We have
\begin{align*}
F_1=\tfrac{1}{2}(H_1+H_2),F_3=\tfrac{1}{2}(H_2-H_1),\\
G_1=\tfrac{1}{2}(H_3+H_4),G_3=\tfrac{1}{2}(H_4-H_3),\\
F_2(x)=G_1(x)+G_3(x)+\tfrac{3}{1-\alpha}x-D_2,\\
G_2(x)=F_1(x)+F_3(x)-\tfrac{3}{1-\alpha}x-D_1.    
\end{align*}

So the general solution for $F_i,G_i, 1\leq i \leq 3$ is

\begin{align*}
F_1(x) &= \tfrac{1}{1-\alpha}x + \tfrac{D_1}{3} - \tfrac{1}{3\alpha} + \tfrac{1}{3(1-\alpha)} - \tfrac{1}{3\alpha(1-\alpha)}\\
&+ \tfrac{1}{2}\left(A + \tfrac{D}{\sqrt{3}}\right)\cos\left(\tfrac{3\sqrt{3}\alpha}{2}x\right) + \tfrac{1}{2}\left(B + \tfrac{C}{\sqrt{3}}\right)\sin\left(\tfrac{3\sqrt{3}\alpha}{2}x\right), \\
F_2(x) &= \tfrac{1}{1-\alpha}x - \tfrac{D_2}{3} - \tfrac{2}{3\alpha} + \tfrac{2}{3(1-\alpha)}\\&+\left(\tfrac{B}{\sqrt{3}}\right)\cos\left(\tfrac{3\sqrt{3}\alpha}{2}x\right) + \left(- \tfrac{A}{\sqrt{3}}\right)\sin\left(\tfrac{3\sqrt{3}\alpha}{2}x\right), \\
F_3(x)  &= \tfrac{1}{1-\alpha}x + \tfrac{D_1}{3} - \tfrac{1}{3\alpha} + \tfrac{1}{3(1-\alpha)} + \tfrac{1}{3\alpha(1-\alpha)}\\
&+ \tfrac{1}{2}\left(\tfrac{D}{\sqrt{3}} - A\right)\cos\left(\tfrac{3\sqrt{3}\alpha}{2}x\right) + \tfrac{1}{2}\left(\tfrac{C}{\sqrt{3}} - B\right)\sin\left(\tfrac{3\sqrt{3}\alpha}{2}x\right), \\
G_1(x)  &= -\tfrac{1}{1-\alpha}x + \tfrac{D_2}{3} - \tfrac{1}{3\alpha} + \tfrac{1}{3(1-\alpha)} - \tfrac{1}{3\alpha(1-\alpha)}\\
&+ \tfrac{1}{2}\left(C + \tfrac{B}{\sqrt{3}}\right)\cos\left(\tfrac{3\sqrt{3}\alpha}{2}x\right) - \tfrac{1}{2}\left(D + \tfrac{A}{\sqrt{3}}\right)\sin\left(\tfrac{3\sqrt{3}\alpha}{2}x\right), \\
G_2(x) &= -\tfrac{1}{1-\alpha}x - \tfrac{D_1}{3} - \tfrac{2}{3\alpha} + \tfrac{2}{3(1-\alpha)}\\
&+\left(\tfrac{D}{\sqrt{3}}\right)\cos\left(\tfrac{3\sqrt{3}\alpha}{2}x\right) + \left(\tfrac{C}{\sqrt{3}}\right)\sin\left(\tfrac{3\sqrt{3}\alpha}{2}x\right), \\
G_3(x)  &= -\tfrac{1}{1-\alpha}x + \tfrac{D_2}{3} - \tfrac{1}{3\alpha} + \tfrac{1}{3(1-\alpha)} + \tfrac{1}{3\alpha(1-\alpha)}\\
&+ \tfrac{1}{2}\left(\tfrac{B}{\sqrt{3}} - C\right)\cos\left(\tfrac{3\sqrt{3}\alpha}{2}x\right) + \tfrac{1}{2}\left(-\tfrac{A}{\sqrt{3}} + D\right)\sin\left(\tfrac{3\sqrt{3}\alpha}{2}x\right).
\end{align*}
The boundary condition that $F_i,G_i,1\leq i \leq 3$ must satisfy is:
$$
\begin{cases}
F_1(0)=G_1(1/3)=0\\
F_1(1/3)=F_2(0)=G_1(0)=G_2(1/3)\\
F_2(1/3)=F_3(0)=G_2(0)=G_3(1/3)\\
F_3(1/3)=G_3(0)=1/(1-\alpha).
\end{cases}
$$
There are $6$ variables and $10$ relations, so $4$ of them are redundant. We choose the following boundary conditions as in Figure \ref{fig: boundary condition choice}:
$$
\begin{cases}
F_1(0)=0,F_1(1/3)=G_2(1/3),G_2(0)=F_3(0)\\
G_3(0)=1/(1-\alpha),G_3(1/3)=F_2(1/3),F_2(0)=G_1(0).
\end{cases}
$$

We first check the following 4 conditions:

$$
\begin{cases}
F_1(0) = \frac{D_1}{3} - \frac{1}{3\alpha} + \frac{1}{3(1-\alpha)} - \frac{1}{3\alpha(1-\alpha)} + \frac{1}{2}\left(A + \frac{D}{\sqrt{3}}\right) = 0 \\
G_3(0) = \frac{D_2}{3} - \frac{1}{3\alpha} + \frac{1}{3(1-\alpha)} + \frac{1}{3\alpha(1-\alpha)} + \frac{1}{2}\left(\frac{B}{\sqrt{3}} - C\right) = \frac{1}{1-\alpha} \\
G_2(0) - F_3(0) = \left(-\frac{D_1}{3} - \frac{2}{3\alpha} + \frac{2}{3(1-\alpha)} + \frac{D}{\sqrt{3}}\right) \\
\quad - \left(\frac{D_1}{3} - \frac{1}{3\alpha} + \frac{1}{3(1-\alpha)} + \frac{1}{3\alpha(1-\alpha)} + \frac{1}{2}\left(\frac{D}{\sqrt{3}} - A\right)\right) = 0 \\
F_2(0) - G_1(0) = \left(-\frac{D_2}{3} - \frac{2}{3\alpha} + \frac{2}{3(1-\alpha)} + \frac{B}{\sqrt{3}}\right) \\
\quad - \left(\frac{D_2}{3} - \frac{1}{3\alpha} + \frac{1}{3(1-\alpha)} - \frac{1}{3\alpha(1-\alpha)} + \frac{1}{2}\left(C + \frac{B}{\sqrt{3}}\right)\right) = 0.    
\end{cases}
$$
\begin{figure}
    \centering
    \begin{tikzpicture}[scale=1, thick]

    \draw (0,0) rectangle (1,1);
    \draw (-1,0) rectangle (0,1);
    \draw (1,0) rectangle (2,1);

    \draw[red,->] (-1,1) -- (0,1) node[above left]{$F_1$};
    \draw[red,->] (1,1) -- (0,1) node[above right]{$G_2$};
    \draw[red,->] (1,1) -- (2,1) node[above]{$F_3$};
    \draw[red,->] (2,0) -- (1,0) node[above right]{$G_3$};
    \draw[red,->] (0,0) -- (1,0) node[above left]{$F_2$};
    \draw[red,->] (0,0) -- (-1,0) node[above right]{$G_1$};

    \draw[blue, thick] (-1,0) -- (-1,1) ;
    \draw[blue, thick] (0,0) -- (0,1);
    \draw[blue, thick] (1,0) -- (1,1);
    \draw[blue, thick] (2,0) -- (2,1);
    
    \filldraw (0,0) circle (1.5pt) node[below]{$1/3$};
    \filldraw (1,0) circle (1.5pt) node[below]{$2/3$};
    \filldraw (2,0) circle (1.5pt) node[below]{$1$};
    \filldraw (-1,1) circle (1.5pt) node[below]{};
    \filldraw (0,1) circle (1.5pt) node[below]{};
    \filldraw (1,1) circle (1.5pt) node[below]{};

\end{tikzpicture}  
    \caption{The 6 dots represent the 6 boundary conditions chosen, and the vertical segments represent the 4 redundant boundary conditions.}
    \label{fig: boundary condition choice}
\end{figure}

The first and third equations only depend on 
$D_1$ and $A+\frac{D}{\sqrt{3}}$, and the second and fourth equations depend on $D_2$ and $\frac{B}{\sqrt{3}}-C$. So we can solve

$$D_1=0,D_2=\frac{1}{1-\alpha}, A+\frac{D}{\sqrt{3}}=\frac{4}{3\alpha},\frac{B}{\sqrt{3}}-C=0.$$

The remaining two conditions are 
$$
\begin{cases}
F_1(\frac{1}{3})-G_2(\frac{1}{3})=\frac{1}{2}(A-\frac{D}{\sqrt{3}})\cos(\frac{\sqrt{3}\alpha}{2})+\frac{1}{2}(B-\frac{C}{\sqrt{3}})\sin(\frac{\sqrt{3}\alpha}{2})=0\\
G_3(\frac{1}{3})-F_2(\frac{1}{3})=\frac{2}{3\alpha}-\frac{1}{2}(\frac{B}{\sqrt{3}}+C)\cos(\frac{\sqrt{3}\alpha}{2})+\frac{1}{2}(\frac{A}{\sqrt{3}}+D)\sin(\frac{\sqrt{3}\alpha}{2})=0.
\end{cases}
$$
Plugging in $A=-\frac{D}{\sqrt{3}}+\frac{4}{3\alpha}$ and $C=\frac{B}{\sqrt{3}}$, we get
$$
\begin{cases}
(-\frac{D}{\sqrt{3}}+\frac{2}{3\alpha})\cos(\frac{\sqrt{3}\alpha}{2})+\frac{1}{3}B\sin\frac{\sqrt{3}\alpha}{2}=0\\
\frac{2}{3\alpha}-\frac{B}{\sqrt{3}}\cos(\frac{\sqrt{3}\alpha}{2})+(\frac{D}{3}+\frac{2}{3\sqrt{3}\alpha})\sin(\frac{\sqrt{3}\alpha}{2})=0.
\end{cases}
$$
The solution is

$$
\begin{cases}
B = \frac{2 \left( \sqrt{3} \cos \left( \frac{\sqrt{3} \alpha}{2} \right) + \sin \left( \sqrt{3} \alpha \right) \right)}{\alpha + 2 \alpha \cos \left( \sqrt{3} \alpha \right)}\\
D = \frac{2 \left( 2 + \cos \left( \sqrt{3} \alpha \right) + \sqrt{3} \sin \left( \frac{\sqrt{3} \alpha}{2} \right) \right)}{\sqrt{3} \left( \alpha + 2 \alpha \cos \left( \sqrt{3} \alpha \right) \right)}.
\end{cases}
$$

Therefore,
$$
\begin{cases}
C=\frac{2 \left( \sqrt{3} \cos \left( \frac{\sqrt{3} \alpha}{2} \right) + \sin \left( \sqrt{3} \alpha \right) \right)}{\sqrt{3}(\alpha + 2 \alpha \cos \left( \sqrt{3} \alpha \right))}\\
A=\frac{6\cos(\sqrt{3}\alpha)-2\sqrt{3}\sin(\frac{\sqrt{3}\alpha}{2})}{3(\alpha+2\alpha\cos(\sqrt{3}\alpha))}.
\end{cases}
$$
In sum, we have
$$
\begin{cases}
A=\frac{6\cos(\sqrt{3}\alpha)-2\sqrt{3}\sin(\frac{\sqrt{3}\alpha}{2})}{3(\alpha+2\alpha\cos(\sqrt{3}\alpha))}\\  
B=\frac{2 \left( \sqrt{3} \cos \left( \frac{\sqrt{3} \alpha}{2} \right) + \sin \left( \sqrt{3} \alpha \right) \right)}{\alpha + 2 \alpha \cos \left( \sqrt{3} \alpha \right)}\\
C=\frac{2 \left( \sqrt{3} \cos \left( \frac{\sqrt{3} \alpha}{2} \right) + \sin \left( \sqrt{3} \alpha \right) \right)}{\sqrt{3}(\alpha + 2 \alpha \cos \left( \sqrt{3} \alpha \right))}\\
D=\frac{2 \left( 2 + \cos \left( \sqrt{3} \alpha \right) + \sqrt{3} \sin \left( \frac{\sqrt{3} \alpha}{2} \right) \right)}{\sqrt{3} \left( \alpha + 2 \alpha \cos \left( \sqrt{3} \alpha \right) \right)}\\
D_1=0\\
D_2=\frac{1}{1-\alpha}.
\end{cases}
$$

\end{sloppypar}
\end{document}